\documentclass[11pt, a4paper, oneside, reqno]{amsart}

\usepackage[margin=1in]{geometry}
\usepackage{amsfonts}
\usepackage{amsmath}
\usepackage{amssymb,bbm}
\usepackage{amsthm}
\usepackage{hyperref}
\usepackage{mathtools}
\usepackage{tikz}
\usepackage{tikz-cd}
\usepackage{enumerate}
\usepackage{stmaryrd} 
\usepackage{mathabx}  
\usepackage{xcolor}
\makeatletter
\usepackage{svg}
\renewcommand{\tocsection}[3]{%
	\indentlabel{\@ifnotempty{#2}{\bfseries\ignorespaces#1 #2.\,\,}}\bfseries#3}
\renewcommand{\tocsubsection}[3]{%
	\indentlabel{\@ifnotempty{#2}{\ignorespaces#1 #2\quad}}#3}
\renewcommand{\tocsubsubsection}[3]{%
	\quad\quad\quad\indentlabel{\@ifnotempty{#2}{\ignorespaces#1 #2\quad}}#3}

\def\bN{\mathbb{N}}

\def\bC{\mathbb{C}}

\def\norm#1{\Vert #1 \Vert}
\def\ip#1{\left \langle #1 \right \rangle}

\def\1{\mathbf{1}}

\def\Mn{{\mathbb{M}_{N}}}

\DeclareMathOperator{\tr}{tr}
\DeclareMathOperator{\Tr}{Tr}
\DeclareMathOperator{\free}{free}
\DeclareMathOperator{\Bool}{Bool}
\DeclareMathOperator{\mono}{mono}
\DeclareMathOperator{\ten}{ten}
\DeclareMathOperator{\id}{id}

\newtheorem{theorem}{Theorem}[section]
\newtheorem{proposition}[theorem]{Proposition}
\newtheorem*{proposition*}{Proposition}

\newtheorem*{theorem*}{Theorem}
\newtheorem{lemma}[theorem]{Lemma}
\newtheorem{fact}[theorem]{Fact}

\theoremstyle{remark}
\newtheorem{remark}[theorem]{Remark}
\newtheorem{example}[theorem]{Example}

\theoremstyle{definition}
\newtheorem{definition}[theorem]{Definition}
\newtheorem{observation}[theorem]{Observation}

\numberwithin{equation}{section}

\author{Nicolas Gilliers}
\address{\parbox{\linewidth}{CNRS, MAP5 \\ Universit{\'e} Paris Cit{\'e} \\ 45 rue des Saints-P{\`e}res, F-75006, Paris, France}}
\email{nicolas.gilliers@u-paris.fr}

\author{David Jekel}
\address{\parbox{\linewidth}{Department of Mathematical Sciences \\ University of Copenhagen \\ Universitetsparken 5, 2100 K{\o}benhavn {\O}, Denmark}}
\email{daj@math.ku.dk}

\title[Bigraph independence]{Bigraph independence: \\ a mixture of the five natural independences}

\newcommand{\pcg}{{\mathcal{P}(c,\mathcal{G})}}

\begin{document}

\begin{abstract}
	We introduce a notion of non-commutative joint independence for \emph{multiple} algebras in a non-commutative probability space. The \emph{pairwise} relationships between these algebras are encoded by a graph with two edge sets—a combinatorial structure we call a \emph{bigraph}— and naturally encompass the five fundamental types of independence: tensor, free, (anti)monotone, and Boolean. It subsumes the BMT independence of Arizmendi--Mendoza--Vazquez-Becerra (when all pairwise relationships are Boolean, (anti)monotone, or tensor) and the $\epsilon$ or $\Lambda$-independence of M{\l}otkowski (when the pairwise relationships are tensor and free). We present explicit combinatorial moment formulas, a Hilbert space construction, and natural \emph{associativity} relations within this setting. Furthermore, we demonstrate that bigraph independence emerges in the asymptotic behavior of tensor product random matrix models with respect to a vector state, encompassing the Charlesworth--Collins model for $\varepsilon$-independence as a special case and offering a random matrix perspective on BMT independence.
\end{abstract}

\maketitle
\tableofcontents

\section{Introduction}

\subsection{Motivation}

Free probability originated largely from the work of Voiculescu in the 80’s and has since grown into a surprisingly effective framework for the study of high-dimensional phenomena in random matrix theory, among many other achievements of Voiculescu's theory. While freeness, as a non-commutative analogue of classical independence, was the main focus of Voiculescu’s work, it was soon realised that the notion of non-commutative independence, formalised in the early 2000s, contains more than freeness and classical stochastic independence. In particular, Speicher \cite{speicher1997}, Schürmann, Ben Ghorbal \cite{ghorbal2002non} and Muraki \cite{muraki2003five} independently classified the five ``natural’’ notions of non-commutative independence (that arise from an associative \emph{binary product} on non-commutative probability spaces):  tensor (or \emph{classical}), \emph{free}, \emph{Boolean}, \emph{monotone}, and \emph{anti-monotone} independence.  The parallels between these five independences go strikingly far: the joint moments for several independent random variables can be described using the combinatorial machinery of cumulants, expressing joint moments as a sum of cumulants \cite{nica2006lectures,SpeicherWoroudi1997BooleanConvolution,arizmendi2015relations, lehner2004cumulants} indexed by partitions (for tensor independence, all partitions are used; for free, non-crossing partitions; for Boolean, interval partitions; and for monotone, labelled partitions satisfying a certain ordering condition).  Further, there are parallel central limit theorems (and limit theorems for iterated convolution in general) with both additive and multiplicative versions  \cite{nica2006lectures,SpeicherWoroudi1997BooleanConvolution,muraki2000monotonic, tucci2010limits}.

Beyond these five natural independences, numerous constructions have emerged that ``independently'' join non-commutative probability spaces, broadening the concept beyond the historical categorical formulation.  Generally, by a non-commutative independence relation, we mean a rule prescribing how to compute \emph{joint moments} of elements from several subalgebras $A_1$, \dots, $A_N$ of a non-commutative probability space $(A,\varphi)$, in terms of the individual moments $\varphi|_{A_j}$. For the five natural independences, this rule for any number $N$ of algebras follows from the prescription at $N=2$. However, this need not always be the case—and in the present work, it is not.

Our focus is on \emph{mixed} non-commutative independences, where the \emph{pairwise relationship} between $A_i$ and $A_j$ can be any one of Muraki’s five types. Thus, different pairs of algebras may exhibit distinct independences: for instance, $A_1$ and $A_2$ might be freely independent, $A_2$ and $A_3$ tensor independent, and $A_1$ and $A_3$ Boolean independent. Importantly, these pairwise relations alone do not determine the joint moments of $A_1$, $A_2$, and $A_3$ (see Remark \ref{remark: BMF mixture}; a mixed independence relation must explicitly prescribe these moments.

There are already several examples of such mixed independences in the literature, appearing in various contexts including graph products of algebras, random matrix models, and operator algebras.  In particular, we mention the following:
\begin{itemize}
	\item A mixture of tensor and free independence is studied in \cite{mlotkowski2004} ($\Lambda$ or $\varepsilon$ independence).  This arises from graph products of von Neumann algebras.
	\item A mixture of free, Boolean, and monotone (anti-monotone) independence is studied in \cite{JekelLiu2020,jekel2024general}.  This generalises the mixtures of monotone and Boolean independence in \cite{wysoczanski2010bm} and the mixtures of free and Boolean independence in \cite{kula2013example}.
	\item A mixture of Boolean, monotone (anti-monotone), and tensor products is given in \cite{arizmendi2025bmt}.
\end{itemize}

Thus, two distinct models yield the pairwise relations of Boolean, monotone, and anti-monotone independence as described in \cite{arizmendi2025bmt} and \cite{jekel2024general}. Beyond these, multiple approaches exist for combining different notions of independence in frameworks involving two or more states, such as the \emph{c-free independence} of Bo$\dot{\rm z}$ejko, Leinert and Speicher \cite{bozejko1996convolution}. In this work, however, we restrict our attention to non-commutative probability spaces with a single state. We also mention traffic independence, traffic spaces and the recently introduced notion of freeness for tensors \cite{bonnin2024freeness}, which extend the original categorical framework to the category of operadic algebras over a certain operad of graphs. We emphasise that our work is only concerned with plain standard algebras, although associativity for our mixed independence will naturally be formulated by using the language of operads in \S \ref{sec:operad associativity}.

It is important to note that, in all previously cited works, the mixed independences studied involve only \emph{some} of the five natural types. To date, no mixed independence relation has accommodated all of Muraki’s five. In this paper, we define for the first time a mixed independence that permits \emph{all five pairwise relations}—tensor, free, Boolean, monotone, and anti-monotone independence—to coexist with a \emph{single state}. We term this new notion \emph{bigraph independence}, as it is characterised by a pair of graphs on the same vertex set encoding the pairwise relations.

Our mixed independence generalises both the tensor-free graph independence of \cite{mlotkowski2004} and the BMT independence of \cite{arizmendi2025bmt}, and thus differs from \cite{JekelLiu2020,jekel2024general} for free-Boolean-(anti-)monotone mixtures in general. Our bigraph independence arises naturally in random matrix models on a tensor product over several sites (or ``strings'') as in \cite{MorLau2019,charlesworth2021matrix}, incorporating rank $1$ projections in some tensorands to mimic the tensor product model of \cite{arizmendi2025bmt}. We also present a Hilbert space model realizing our mixed independence (see \ref{sec:hilbertspacemodel}, thereby enabling the construction of independent copies of arbitrary non-commutative probability spaces.

\subsection{Bigraph independence}

We assume familiarity with the basic theory of $\mathrm{C}^*$-algebras.  By a \emph{$\mathrm{C}^*$-probability space} we mean a pair $(A,\varphi)$ where $A$ is a $\mathrm{C}^*$-algebra and $\varphi: A \to \mathbb{C}$ is a state.  We will begin by introducing the combinatorial data and stating the formula for joint moments in our model of mixed independence.

\begin{definition} \label{def: bigraph}
	A \emph{bigraph} is a triple $\mathcal{G} = (\mathcal{V},\mathcal{E}_1,\mathcal{E}_2)$ where
	\begin{itemize}
		\item $\mathcal{V}$ is a (finite or countable) set of vertices.
		\item $\mathcal{E}_1, \mathcal{E}_2 \subseteq \mathcal{V} \times \mathcal{V}$ are edge sets of two types.
		\item $\mathcal{E}_1$ is reflexive: $\Delta := \{(v,v): v \in \mathcal{V}\} \subseteq \mathcal{E}_1$.
		\item $\mathcal{E}_2$ is irreflexive: $\mathcal{E}_2 \cap \Delta = \varnothing$.
		\item $\mathcal{E}_2$ is symmetric: $(v,w) \in \mathcal{E}_2 \iff (w,v) \in \mathcal{E}_2$.
	\end{itemize}
	We also write $\overline{\mathcal{E}}_1 := \{(v,w) \in \mathcal{V}^2 : (w,v) \in \mathcal{E}_1\}$.
\end{definition}
We note that $(\mathcal{V},\mathcal{E}_1)$ is a directed graph with reflexive loops, while $(\mathcal{V},\mathcal{E}_2)$ is an undirected, loopless graph. In Proposition \ref{prop: pairwise independence}, we show how Boolean, monotone, free, and tensor pairwise relations are encoded by these two edge sets.
Let $(A,\varphi)$ be a non-commutative probability space, and let $(A_v)_{v \in \mathcal{V}}$ be a family of subalgebras of $A$.
We aim to express $\varphi(a_1\cdots a_k)$ for $a_j \in A_{c(j)}$ given a labeling $c:[k]\to\mathcal{V}$.
As usual, joint moment formulas are indexed by set partitions of $[k]$. A partition $\pi$ is a family of disjoint nonempty blocks whose union is $[k]$; the set of all partitions is denoted by $\mathcal{P}(k)$. For $i,j\in[k]$, we write $i\sim_\pi j$ if $i$ and $j$ lie in the same block of $\pi$.

\begin{definition} \label{def: compatible partitions}
	Given a bigraph $\mathcal{G}$, $k\in\mathbb{N}$, and a labeling $c:[k]\to\mathcal{V}$, define $\mathcal{P}(c,\mathcal{G})\subseteq \mathcal{P}(k)$ to be the set of partitions $\pi$ satisfying:
	\begin{enumerate}
		\item 
		      If $i\sim_\pi j$, then $c(i)=c(j)$.
		\item 
		      If $i_1<j<i_2$ and $i_1\sim_\pi i_2$, then $(c(i_1),c(j))\in\mathcal{E}_1$.
		\item 
		      If $i_1<j_1<i_2<j_2$ and $i_1\sim_\pi i_2$ and $j_1\sim_\pi j_2$ and $i_1 \not \sim_\pi j_1$, then $(c(i_2),c(j_1))\in\mathcal{E}_2$.
	\end{enumerate}
\end{definition}
\begin{figure}[htb]
\centering
\includegraphics[scale=0.8]{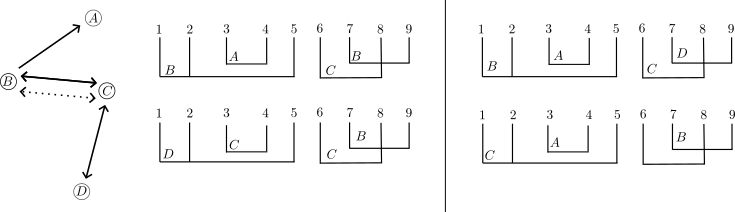}
\caption{\label{fig:exemplepartition} On the left, example of a bigraph $\mathcal{G}$ : edges of type $1$ are drawn as solid black edges and edges of type $2$ are drawn as dashed edges. On the right : partitions in the leftmost column belong to $\mathcal{P}(c,\mathcal{G})$, those in the rightmost column do \emph{not} belong to $\mathcal{P}(c,\mathcal{G})$.}
\end{figure}
\begin{remark}
	The set $\mathcal{P}(c,\mathcal{G})$ depends only on the induced sub-bigraph on ${\rm Im}(c)$; thus if $\mathcal{G}_{|c}$ denotes the restriction of $\mathcal{G}$ to vertices appearing in $c$, then $\mathcal{P}(c,\mathcal{G})=\mathcal{P}(c,\mathcal{G}_{|c})$.
\end{remark}
Let $k\geq 1$ and $\pi \in \mathcal{P}(k)$ a partition, it will be convenient to work
sometimes with colorings $c : [k] \rightarrow \mathcal{V}$ and sometimes with colorings defined on
the blocks of $\pi$, that is, functions $\pi \rightarrow \mathcal{V}$.
The space $\mathcal{V}^{\pi}$ can be canonically identified with the subspace of $\mathcal{V}^k$ consisting of functions which are constant on each block. We use this identification throughout the rest of the paper.

\begin{definition}[Bigraph independence] \label{def: bigraph independence}
	Let $\mathcal{G}=(\mathcal{V},\mathcal{E}_1,\mathcal{E}_2)$ be a bigraph. In a non-commutative probability space $(A,\varphi)$, a family of subalgebras $(A_v)_{v\in\mathcal{V}}$ is $\mathcal{G}$-independent if for every $k\in\mathbb{N}$, labeling $c:[k]\to\mathcal{V}$, and elements $a_j\in A_{c(j)}$,
	\begin{equation}\label{eqn:formula moments}
		\varphi(a_1\cdots a_k)=\sum_{\pi\in\mathcal{P}(c,\mathcal{G})} K^{\free}_{\pi}(a_1,\ldots,a_k),
	\end{equation}
	where $K^{\free}_{\pi}$ denotes the \emph{partitioned free cumulant} (see \cite{nica2006lectures}).
\end{definition}

This definition is not vacuous: in Section \ref{sec:hilbertspacemodel} we construct $\mathcal{G}$-independent copies of arbitrary $\mathrm{C}^*$-probability spaces.
Moreover, the pairwise relations of Boolean, (anti)monotone, free, and tensor independence are exactly encoded by the sets $\mathcal{E}_1$ and $\mathcal{E}_2$ as follows.

\begin{proposition} \label{prop: pairwise independence}
	Consider the situation of Definition \ref{def: bigraph independence}. Let $v,w\in\mathcal{V}$ with $v\neq w$. Then $A_v$ and $A_w$ are:
	\begin{enumerate}
		\item Boolean independent if $(v,w)\in \mathcal{E}_{\Bool} := \mathcal{V}\times\mathcal{V}\setminus (\mathcal{E}_1 \cup \overline{\mathcal{E}}_1)$;
		\item monotone independent if $(v,w)\in \mathcal{E}_{\mono} := \mathcal{E}_1 \setminus \overline{\mathcal{E}}_1$;
		\item anti-monotone independent if $(v,w)\in \overline{\mathcal{E}}_{\mono} = \overline{\mathcal{E}}_1 \setminus \mathcal{E}_1$;
		\item freely independent if $(v,w)\in \mathcal{E}_{\free} := (\mathcal{E}_1 \cap \overline{\mathcal{E}}_1)\setminus \mathcal{E}_2$;
		\item tensor independent if $(v,w)\in \mathcal{E}_{\ten} := \mathcal{E}_1 \cap \overline{\mathcal{E}}_1 \cap \mathcal{E}_2$;
		\item if $\mathcal{E}_1 = \mathcal{V}\times\mathcal{V}$, then $(A_v)_{v\in\mathcal{V}}$ are graph-independent ($\varepsilon$-independent) in the sense of \cite{mlotkowski2004} with respect to $(\mathcal{V},\mathcal{E}_2)$;
		\item if $\mathcal{E}_2 \supseteq (\mathcal{E}_1 \cap \overline{\mathcal{E}}_1)\setminus \Delta$, then $(A_v)_{v\in\mathcal{V}}$ are BMT-independent in the sense of \cite{arizmendi2025bmt} with respect to the digraph $(\mathcal{V},\mathcal{E}_1\setminus\Delta)$.
	\end{enumerate}
\end{proposition}
For proofs, see Section \ref{sec:relationsotherindependences}.

\begin{remark} \label{remark: BMF mixture}
	In the case where $\mathcal{E}_2 = \varnothing$, the pairwise relations between the algebras will be free, (anti)monotone, or Boolean independence.  Thus it is naturally to compare these cases of bigraph independence with the mixtures coming from digraph independence in \cite{jekel2024general}. We can already see from \cite[Remark 3.24]{jekel2024general} that these independences provide different joint moments when the pairwise relations are only (anti)monotone and Boolean, since bigraph independence agrees with BMT independence \cite{arizmendi2025bmt} in this case.  In the case of a general free, (anti)monotone, and Boolean mixture, both bigraph independence and digraph independence are special cases of the tree independence of \cite{JekelLiu2020}; can therefore evaluate the joint moments of both in terms of the moment formula from that paper, and show that they disagree.
\end{remark}

\begin{remark}
	\label{rk:wlog}
	The reader may notice that the edges in $\mathcal{E}_2 \setminus (\mathcal{E}_1 \cap \overline{\mathcal{E}}_1)$ do not play any role in the definitions.  Indeed, in Definition \ref{def: compatible partitions} (3), if $i_1 < j_1 < i_2 < j_2$ and $i_1 \sim_\pi i_2$ and $j_1 \sim_\pi j_2$, we have by (2) that $(c(i_1),c(j_1)) \in \mathcal{E}_1$ and also $(c(j_1),c(i_2)) \in \mathcal{E}_1$, and therefore $(c(i_1),c(j_1))$ is automatically in $ \mathcal{E}_2 \cap \mathcal{E}_1 \cap \overline{\mathcal{E}}_1$.  Moreover, in Proposition \ref{prop: pairwise independence}, if $(v,w) \not \in \mathcal{E}_1 \cap \overline{\mathcal{E}}_1$, the pairwise relation of $A_v$ and $A_w$ is unaffected by whether $(v,w) \in \mathcal{E}_2$.  Hence, in general, \emph{we can assume without loss of generality} that $\mathcal{E}_2 \subseteq \mathcal{E}_1 \cap \overline{\mathcal{E}}_1$ by discarding the other elements from $\mathcal{E}_2$, and the definition of $\mathcal{G}$-independence will be unaffected. However, allowing more flexibility in the choice of $\mathcal{E}_2$ is convenient for stating our random matrix construction in Theorem \ref{thm: matrix models main}.
\end{remark}

\begin{remark}
	Any pairwise combination of Boolean, (anti)monotone, free, and tensor independence can be realized by some bigraph, which can additionally be taken with $\mathcal{E}_2 \subseteq \mathcal{E}_1 \cap \overline{\mathcal{E}}_1$.  Indeed, suppose we are given a vertex set $\mathcal{V}$ and sets
	\[
		\mathcal{E}_{\Bool}, \mathcal{E}_{\mono}, \mathcal{E}_{\free}, \mathcal{E}_{\ten} \subseteq \mathcal{V} \times \mathcal{V}
	\]
	satisfying
	\begin{itemize}
		\item $\mathcal{V} \times \mathcal{V} = \mathcal{E}_{\Bool} \sqcup \mathcal{E}_{\mono} \sqcup \overline{\mathcal{E}}_{\mono} \sqcup \mathcal{E}_{\free} \sqcup \mathcal{E}_{\ten}$,
		\item $\overline{\mathcal{E}}_{\Bool} = \mathcal{E}_{\Bool}$ and $\overline{\mathcal{E}}_{\free} = \mathcal{E}_{\free}$.
		\item $\Delta \subseteq \mathcal{E}_{\free}$.
	\end{itemize}
	We can then set
	\[
		\mathcal{E}_1 = \mathcal{E}_{\mono} \cup \mathcal{E}_{\free} \cup \mathcal{E}_{\ten}, \qquad \mathcal{E}_2 = \mathcal{E}_{\ten},
	\]
	and this choice will satisfy
	\[
		\mathcal{E}_{\Bool} = (\mathcal{V} \times \mathcal{V}) \setminus (\mathcal{E}_1 \cup \overline{\mathcal{E}}_1), \quad
		\mathcal{E}_{\mono} = \mathcal{E}_1 \setminus \overline{\mathcal{E}}_1,
		\quad
		\mathcal{E}_{\free} = \mathcal{E}_1 \cap \overline{\mathcal{E}}_1 \setminus \mathcal{E}_2,
		\quad
		\mathcal{E}_{\ten} = \mathcal{E}_1 \cap \overline{\mathcal{E}}_1 \cap \mathcal{E}_2.
	\]
\end{remark}

In \ref{sec:hilbertspacemodel} we provide a Hilbert space model for bigraph independence (Theorem \ref{thm: Hilbert realization}) that allows us to realize $\mathcal{G}$-independent copies of any $\mathrm{C}^*$-probability spaces.

\begin{proposition} \label{prop: independent C star copies}
	Let $\mathcal{G}$ be a bigraph, and let $(A_v,\varphi_v)$ be $\mathrm{C}^*$-probability spaces.  Then there exists a $\mathrm{C}^*$-probability space $(A,\varphi)$ and state-preserving $*$-homomorphisms $\lambda_v: A_v \to A$ such that $(\lambda_v(A_v))_{v \in \mathcal{V}}$ are $\mathcal{G}$-independent.
\end{proposition}

We remark that, as in the monotone and Boolean cases for instance, the maps $\lambda_v$ need not be unital.  Our Hilbert space model contains as special cases the free--tensor graph product Hilbert space of \cite{mlotkowski2004}, and injects into the BMT Hilbert space model from \cite{arizmendi2025bmt}.  The Hilbert space is a direct sum of subspaces indexed by words on the alphabet $\mathcal{V}$; we \emph{simultaneously} place some restrictions on which words may occur, as in \cite{JekelLiu2020,jekel2024general} and also consider words up to equivalence via swapping consecutive letters which are connected by an edge in $\mathcal{E}_{\ten}$ as in \cite{mlotkowski2004}.
The proof of Theorem \ref{thm: Hilbert realization} that the Hilbert space construction produces the moment relation of $\mathcal{G}$-independence from Definition \ref{def: bigraph independence} proceeds along similar lines to many previous works, such as \cite[Theorem 3.13]{jekel2024general}.

\subsection{Matrix models}

We next describe how our bigraph independence arises naturally from random matrix models in a tensor-product setting, which generalize the construction of \cite{charlesworth2021matrix}.

\begin{theorem}~ \label{thm: matrix models main}
	\begin{itemize}
		\item Consider two sets $\mathcal{S}$ and $\mathcal{V}$.
		\item For each $v \in \mathcal{V}$, let $S_v^{(1)}$, $S_v^{(2)}$, $S_v^{(3)}$ be a partition of $\mathcal{S}$ such that $S_v^{(1)}$ is non-empty.
		\item Let $\mathcal{G} = (\mathcal{V},\mathcal{E}_1,\mathcal{E}_2)$, where
		      \[
			      \mathcal{E}_1 = \{(v,w): S_v^{(1)} \cap S_w^{(2)} = \varnothing \}, \qquad \mathcal{E}_2 = \{(v,w): S_v^{(1)} \cap S_w^{(1)} = \varnothing \}.
		      \]
		\item Let $\Mn$ denote the set of $N \times N$ complex matrices; fix a unit vector $\xi \in \bC^N$; and consider $\Mn^{\otimes \mathcal{S}}$ as a non-commutative probability space with the state given by the vector $\xi^{\otimes \mathcal{S}}$.
		\item Define the (not necessarily unital) $*$-homomorphism
		      \[
			      \lambda_v^{(N)}: \Mn^{\otimes S_v^{(1)}} \to \Mn^{\otimes \mathcal{S}}, \qquad \lambda_v^{(N)}(A) = A \otimes (\xi \xi^*)^{\otimes S_v^{(2)}} \otimes I_N^{\otimes S_v^{(3)}},
		      \]
		      where $\xi \xi^*$ denotes the rank-one projection onto the span of $\xi$.
		\item For each $v \in V$, let $I_v$ be a finite index set, and let $\mathbf{A}_i^{(N)} = (A_{v,i}^{(N)})_{i \in I_v}$ be random matrices satisfying, almost surely,
		      \[
			      \sup_{N \in \bN} \norm{A_{v,i}^{(N)}} < \infty.
		      \]
		\item For each $v \in \mathcal{V}$, let $U_v^{(N)}$ be a Haar random unitary matrix in $\Mn^{\otimes S_v^{(1)}}$, such that the $U_v^{(N)}$'s for $v \in \mathcal{V}$ and $N \in \bN$ are jointly independent of the $(A_{v,i}^{(N)})$'s for $v \in \mathcal{V}$, $i \in I_v$, and $N \in \bN$, and let
		      \[
			      \mathbf{B}_v^{(N)} = (\lambda_v^{(N)}(U_v^{(N)} A_{v,i}^{(N)} (U_v^{(N)})^*))_{i \in I_v}.
		      \]
	\end{itemize}
	Then almost surely the families $\mathbf{B}_v^{(N)}$ for $v \in \mathcal{V}$ are asymptotically $\mathcal{G}$-independent in the following sense:  For every $k \in \bN$ and $c: [k] \to \mathcal{V}$ and non-commutative polynomials $p_j$ in variables indexed by $I_{c(j)}$ with no constant term, we have
	\begin{equation} \label{eq: matrix convergence statement main}
		\lim_{N \to \infty} \biggl| \ip{\xi^{\otimes \mathcal{S}}\,,\, p_1(\mathbf{B}_{c(1)}^{(N)}) \dots p_k(\mathbf{B}_{c(k)}) \xi^{\otimes \mathcal{S}}} - \sum_{\pi \in \mathcal{P}(c,\mathcal{G})} K_{\free,\pi}(p_1(\mathbf{B}_{c(1)}^{(N)}),\dots, p_k(\mathbf{B}_{c(k)})) \biggr| = 0
	\end{equation}
	almost surely, where the free cumulants are computed with respect to the state given by $\xi^{\otimes \mathcal{S}}$.  We also remark that for each $v$ and polynomial $p$ with no constant term,
	\begin{equation} \label{eq: matrix trace versus vector main}
		\lim_{N \to \infty} \Bigl|\ip{\xi^{\otimes \mathcal{S}}\,,\, p(\mathbf{B}_v^{(N)}) \xi^{\otimes \mathcal{S}}} - \tr_{N^{S_v^{(1)}}}[p(\mathbf{A}_v)] \Bigr| = 0 \text{ almost surely.}
	\end{equation}
\end{theorem}

This theorem generalizes \cite{charlesworth2021matrix}.  Indeed, the random matrix models from \cite{charlesworth2021matrix} occur when $S_v^{(2)} = \varnothing$ for all $v$ (and in this case using a fixed vector state as in our theorem gives asymptotically the same result as using the trace as in \cite{charlesworth2021matrix}). When $S_v^{(2)} = \varnothing$, then $\mathcal{E}_1 = \mathcal{V} \times \mathcal{V}$.  Thus, in Proposition \ref{prop: pairwise independence}, we would obtain $\mathcal{E}_{\ten} = \mathcal{E}_2$ and $\mathcal{E}_{\free} = (\mathcal{V} \times \mathcal{V}) \setminus \mathcal{E}_2$, so we recover the setting of $\varepsilon$-independence from \cite{mlotkowski2004}.

Meanwhile, the matrix model also includes the BMT product construction of \cite{arizmendi2025bmt} as a special case when $\mathcal{S} = \mathcal{V}$ and $S_v^{(1)} = \{v\}$ for each $v$.  Then our tensor product becomes simply $\Mn^{\otimes \mathcal{V}}$ and $\lambda_v$ is simply the inclusion of the $v$th tensorand with copies of $\xi \xi^*$ on the vertices indexed by $S_v^{(2)}$.  Of course, in \cite{arizmendi2025bmt}, there would not be any conjugation by random unitaries, and the unitary conjugation effectively changes what the state would be on a polynomial in each $\mathbf{A}_v$, but it is needed in our more general setting to obtain the asymptotic free independence.  We also point out that our matrix model produces non-obvious models for BMT independence for graphs where $\mathcal{E}_{\free} = \varnothing$.  Indeed, the sets $S_v^{(1)}$ and $S_w^{(1)}$ are allowed to overlap in our model when $(v,w) \not \in \mathcal{E}_{\ten}$, even though the deterministic construction of BMT independence would require them to be disjoint.  Our model still asymptotically produces BMT independence through the interactions of the unitary conjugation and the rank-one projections.

The proof of Theorem \ref{thm: matrix models main} in \S \ref{sec: random matrix model} also proceeds along similar lines to that of \cite{charlesworth2021matrix} using Weingarten calculus, with additional combinatorial arguments that obtain $\mathcal{G}$-independence from the behavior of the rank-one projections.

\begin{remark}
	Similarly to \cite{charlesworth2021matrix}, every  bigraph $\mathcal{G}$ can be realized using sets $\mathcal{S}$ and $S_v^{(j)}$ for $v \in \mathcal{V}$ and $j \in \{1,2,3\}$ as in Theorems \ref{thm: matrix models main}.  Indeed, fix a bigraph $\mathcal{G} = (\mathcal{V},\mathcal{E}_1,\mathcal{E}_2)$.  Let
	\[
		\mathcal{S} = \mathcal{V} \times \mathcal{V}.
	\]
	For each $v \in \mathcal{V}$, let
	\[
		S_v^{(1)} = \{(v,w), (w,v): w \in \mathcal{V} \text{ such that } (v,w) \not \in \mathcal{E}_2 \},
	\]
	and let
	\[
		S_v^{(2)} = \{(w,w): w \in \mathcal{V} \text{ such that } (w,v) \not \in \mathcal{E}_1 \}
	\]
	and finally let
	\[
		S_v^{(3)} = \mathcal{S} \setminus (S_v^{(1)} \cup S_v^{(2)}).
	\]
	We claim that the bigraph induces by the collection $S_v^{(j)}$ is the original bigraph $\mathcal{G}$.  Indeed, first note that since $\mathcal{E}_2$ is irreflexive, $S_v^{(1)}$ contains $(v,v)$; in particular $S_v^{(1)}$ is nonempty.  In fact, $S_v^{(1)} \cap \Delta = \{(v,v)\}$.  We also constructed $S_w^{(2)}$ so that $S_w^{(2)} \subseteq \Delta$.  Hence, $S_v^{(1)} \cap S_w^{(2)} \neq \varnothing$ if and only if $(v,v) \in S_w^{(2)}$ if and only if $(v,w) \not \in \mathcal{E}_1$.  Hence,
	\[
		\{(v,w): S_v^{(1)} \cap S_w^{(2)} = \varnothing \} = \mathcal{E}_1.
	\]
	Meanwhile, note that for $v \neq w$, we have
	\[
		S_v^{(1)} \cap S_w^{(1)} = \begin{cases}
			\{(v,w), (w,v)\} & \text{if } (v,w) \not \in \mathcal{E}_2 \\
			\varnothing      & \text{if } (v,w) \in \mathcal{E}_2,
		\end{cases}
	\]
	and hence
	\[
		\mathcal{E}_2 = \{(v,w): S_v^{(1)} \cap S_w^{(1)} = \varnothing \}.
	\]
\end{remark}

\subsection{Further questions}

We propose the following questions as a natural continuation of our work:
\begin{enumerate}
	\item Prove central limit theorem, Poisson limit theorem, and more general limit theorems for sums of independent random variables with respect to a sequence of bigraphs $\mathcal{G}_n$ such that
	      \[
		      \lim_{n \to \infty} \frac{|\operatorname{Hom}(\mathcal{G}',\mathcal{G}_n)|}{|\mathcal{V}_n|^{|\mathcal{V}'|}}
	      \]
	      exists for every bigraph $\mathcal{V}'$.  Note that in many situations the limit is described by the homomorphisms into a \emph{bigraphon}, providing the analogue of graphon convergence.  We expect that limit theorems analogous to \cite[Theorem 1.1]{jekel2024general} and \cite[Theorem 1.1]{COSY2024} will hold.  The main challenge will be the generalization to measures with unbounded support since we have not yet developed the appropriate analytic theory of convolution.
	\item Prove analogous theorems for products of independent variables.  We remark that a challenge in this case is that the independent subalgebras $A_v$ may not be \emph{unitally} included into $A$, but a unitary in $A_v$ can be extended to a unitary in $A$ by adding $1_A - 1_{A_v}$ as in \cite[\S 10.3]{JekelLiu2020}.
	\item Given a bigraphon, construct a Fock space and appropriate operators to model the central limit distribution, and more generally, the analogues infinitely divisible distributions.  Such a construction should generalize \cite[\S 2]{COSY2024} and proceed analogously to \cite[\S 6]{jekel2024general}
	\item The matrix models in Theorem \ref{thm: matrix models main} could be generalized in various ways.  For instance, it would be natural to consider a Haar random unitary chosen from the subgroup that fixes $\bC \xi^{\otimes A} \otimes (\bC^n)^{\otimes B}$ for some disjoint subsets $A, B \subseteq \mathcal{S}$.  For instance, we could force the unitary acting $(\bC^n)^{\otimes S_j^{(1)}}$ to fix a subspace of this form.  In principle, the same techniques with Weingarten calculus could be used to analyze this model as well, leading potentially to new moment conditions.
	\item In Theorem \ref{thm: matrix models main}, one could also consider uniformly random permutation matrices in place of Haar random unitaries, and attempt to analyze these models using similar techniques as \cite{ACDGM2021} and \cite{CdSHJKEN2025}.
\end{enumerate}

\subsection{Organization}

The remainder of the paper is divided into three sections.

In Section \ref{sec: bigraph independence}, we prove a formula for the joint moments of $\mathcal{G}$-independent variables in terms of Boolean cumulants. We use this formula to compare with other existing independences. We also prove at the end of this section that our $\mathcal{G}$-independence is associative, in a certain sense.

In Section \ref{sec:hilbertspacemodel}, we construct a Hilbert space and inclusions allowing the construction of $\mathcal{G}$-independence copies of variables.

Finally, in Section \ref{sec: random matrix model}, we present our random matrix model, which exhibits asymptotic $\mathcal{G}$-independence.

\subsection{Acknowledgements}

NG acknowledges the financial support of NYU Abu Dhabi. This work began during his postdoctoral position funded by NYU Abu Dhabi, under the supervision of Marwa Banna. DJ was partially supported by the Danish Independent Research Fund, grant 1026-00371B, and the Horizon Europe Marie Sk{\l}odowska Curie Action FREEINFOGEOM, project 101209517.

We thank Patrick Oliveira Santos, Marwa Banna, and Pierre Youssef for discussions at NYU Abu Dhabi in December 2024 where this work begin, especially Patrick Oliveira Santos for noting the parallel between the moment formulas for digraph independence and for $\varepsilon$-independence that suggests the possibility of combination.  This occurred during DJ visit to speak in the Abu Dhabi Stochastic Days, for which travel funding was provided by NYU Abu Dhabi and Sorbonne Univesrity Abu Dhabi.  Travel funding for NG's visit to Copenhagen in November 2025 was provided by the Marie Curie grant and Universit{\'e} Paris Cit{\'e}.

\section{Bigraph independence and relations to other independences} \label{sec: bigraph independence}

\subsection{Combinatorial setup}

We begin with some combinatorial definitions and lemmas that will be useful in the sequel.
\begin{definition}[Partitions]
	\label{def:partitions}
	We use the notation $[k] = \{1,\dots,k\}$.  A \emph{partition} of $[k]$ is a collection $\pi$ of non-empty subsets of $[k]$, called \emph{blocks}, such that $[k] = \bigsqcup_{B \in \pi} B$.  The set of partitions of $[k]$ is denoted $\mathcal{P}(k)$.
\end{definition}

\begin{definition}[Crossings, non-crossing partitions]
	\label{def:nccrossings}
	For a partition $\pi$ of $[k]$, a \emph{crossing} is a sequence of indices $i < i' < j < j'$ such that $i$ and $j$ are in some block $B$ and $i'$ and $j'$ are in some block $B' \neq B$.  In this case, we also say that $B$ and $B'$ are \emph{crossing}.

	We say that a partition $\pi$ is \emph{non-crossing} if it has no crossings.  We denote by $\mathcal{NC}(k)$ the set of non-crossing partitions of $[k]$.
\end{definition}

\begin{definition}[Nested blocks]
	\label{def:nested}
	If $\pi$ is a partition of $[k]$ and $B', B \in \pi$, we say that $B'$ is \emph{nested inside} $B$ if there exist $i, j \in B$ such that $i < j$ and $B' \subseteq \{i+1,\dots,j-1\} \subseteq [k] \setminus B$; in other words, there are no intervening indices of $B$ between $i$ and $j$, and $B'$ lies entirely between $i$ and $j$.  In this case, we write $B \prec B'$.
\end{definition}

\begin{definition}[Separated blocks]
	\label{def:sepblocks}
	If $\pi$ is a partition of $[k]$ and $B, B' \in \pi$, we say that $B$ and $B'$ are \emph{separated} if there exists $j \in [k]$ such that either $B \subseteq \{1,\dots,j\}$ and $B' \subseteq \{j+1,\dots,k\}$ or $B' \subseteq \{1,\dots,j\}$ and $B \subseteq \{j+1,\dots,k\}$.
\end{definition}

\begin{lemma} \label{lem: tetrachotomy}
	Let $\pi \in \mathcal{P}(k)$ and let $B$ and $B'$ be two distinct blocks in $\pi$.  Then exactly one of the following occurs:
	\begin{itemize}
		\item $B$ and $B'$ are crossing.
		\item $B$ and $B'$ are separated.
		\item $B$ is nested inside $B'$.
		\item $B'$ is nested inside $B$.
	\end{itemize}
\end{lemma}

This fact is well-known.  For proof, see \cite[Lemma 2.6]{jekel2024general}; note that there the entire partition $\pi$ is assumed to be non-crossing, but the proof shows that if the two blocks $B$ and $B'$ are non-crossing, then $B$ and $B'$ are either separated or one is nested inside the other.

We next describe the natural choice of morphisms for bigraphs, which will then allow us to reformulate Definition \ref{def: compatible partitions} in terms of morphisms of bigraphs.

\begin{definition}[Bigraph morphisms]
	Let $\mathcal{G}_1 = (\mathcal{V}^{(1)},\mathcal{E}_1^{(1)},\mathcal{E}_2^{(1)})$ and $\mathcal{G}_2 = (\mathcal{V}^{(2)},\mathcal{E}_1^{(2)},\mathcal{E}_2^{(2)})$ be two bigraphs (see Definition \ref{def: bigraph}). A \emph{bigraph morphism} $\mathcal{G}_1\to\mathcal{G}_2$ is a set-map $f \colon \mathcal{V}^{(1)}\to\mathcal{V}^{(2)}$ sending an edge of type $1$ (resp. of type $2$) to an edge of type $1$ (resp. of type 2), i.e.,
	$$
		(v,w) \in \mathcal{E}_i^{(1)} \implies (f(v),f(w)) \in \mathcal{E}_{i}^{(2)}
	$$
	for $i = 1,2$ and $v,w \in \mathcal{V}$.
\end{definition}

Each partition of $[k]$ has a naturally associated bigraph defined as follows.

\begin{definition}[Bigraph associated to a partition]
	Let $\pi$ be a partition of $[k]$. We define a bigraph $\mathcal{G}(\pi) = (\pi,\mathcal{E}_1(\pi),\mathcal{E}_2(\pi))$ where the vertices are the blocks of $\pi$ and the two edge sets are given by
	$$
		\mathcal{E}^{\pi}_1 = \{ (B,B') \in \pi\times\pi \colon \exists i_1,i_2 \in B, j \in B', i_1 < j < i_2 \} \cup \Delta,
	$$
	where $\Delta = \{(B,B): B \in \pi\}$, and
	$$
		\mathcal{E}^{\pi}_2 = \mathcal{E}^{\pi}_1 \cap \overline{\mathcal{E}_1^{\pi}} \backslash \Delta,
	$$
\end{definition}

\begin{remark}
	In light of Lemma \ref{lem: tetrachotomy}, for a partition $\pi$ and $B, B' \in \pi$:
	\begin{itemize}
		\item $(B,B') \in \mathcal{E}_1(\pi)$ if and only if $B'$ and $B$ are crossing or $B'$ is nested inside $B$.
		\item $(B,B') \in \mathcal{E}_2(\pi)$ if and only if $B$ and $B'$ are crossing.
	\end{itemize}
\end{remark}

\begin{proposition}
	Let $\pi \in \mathcal{P}(k)$, such that $c$ is constant on the blocks of $\pi$ then $\pi \in \pcg$ if and only the labelling $c$ defines a bigraph morphism $\mathcal{G}(\pi) \to \mathcal{G}$.
\end{proposition}

\begin{proof}
	Let $\pi \in \mathcal{P}(k)$ such that $c$ is constant on the blocks of $\pi$. The labelling $c$ defines a bigraph morphism $\mathcal{G}(\pi) \to \mathcal{G}$ if and only if for any two blocks $B,B' \in \pi$,
	$$
		(B,B') \in \mathcal{E}^{\pi}_1 \Rightarrow (c(B),c(B')) \in \mathcal{E}_1,
	$$
	$$
		(B,B') \in \mathcal{E}^{\pi}_2 \Rightarrow (c(B),c(B')) \in \mathcal{E}_2.
	$$
	By definition of $\mathcal{E}^{\pi}_1$ and $\mathcal{E}^{\pi}_2$, this is equivalent to conditions (2) and (3) of Definition \ref{def: compatible partitions}. Condition (1) is satisfied by hypothesis.
\end{proof}

\subsection{The five natural independences, $\varepsilon$-independence, and BMT independence}
\label{sec:relationsotherindependences}

In this subsection, we prove Proposition \ref{prop: pairwise independence} on the pairwise relationships between the algebras in bigraph independence, and its relationship with $\varepsilon$-independence \cite{mlotkowski2004} and BMT independence \cite{arizmendi2025bmt}.  We will first prove (6) and (7) of Proposition \ref{prop: pairwise independence} and then deduce the pairwise relationships from those.

First, we handle the relationship with $\varepsilon$-independence or graph independence from \cite{mlotkowski2004} (M{\l}otkowski's original paper used $\Lambda$ in place of $\varepsilon$, but we use $\varepsilon$ following Speicher and Wysocza{\'n}ski \cite{SpWy2016}).
In the terminology of \cite{SpWy2016}, $\varepsilon$ is a $\{0,1\}$-valued matrix indexed by $\mathcal{V} \times \mathcal{V}$ (i.e. an adjacency matrix for a graph).
We can equivalently view $\varepsilon$ as a subset of $\mathcal{V} \times \mathcal{V}$.
We will show that if $\mathcal{E}_1 = \mathcal{V} \times \mathcal{V}$, then bigraph independence reduces to $\varepsilon$-independence with $\varepsilon = \mathcal{E}_2$, which will establish Proposition \ref{prop: pairwise independence} (6).

\begin{proposition} \label{prop: epsilon independence}
	Let $\mathcal{G}$ be a bigraph with $\mathcal{V} \times \mathcal{V}=\mathcal{E}_1 = \bar{\mathcal{E}}_1$.  Then bigraph independence with respect to $\mathcal{G}$ is equivalent to $\varepsilon$-independence with $\varepsilon = \mathcal{E}_2$ \cite{mlotkowski2004,SpWy2016}.
\end{proposition}

\begin{proof}
	The claim will follow from comparing our moment formula with the one for $\varepsilon$-independence in \cite[Theorem 5.2]{SpWy2016}.  Fix a bigraph $\mathcal{G}$ with $\mathcal{E}_1 = \mathcal{V} \times \mathcal{V}$, and let $\varepsilon = \mathcal{E}_2$.  Let $(A,\varphi)$ be a $\mathrm{C}^*$-probability space and let $(A_v)_{v \in \mathcal{V}}$ be $*$-subalgebras.  Fix $c: [k] \to \mathcal{V}$, and let $a_j \in A_{c(j)}$ for $j = 1, \dots, k$.  If $(A_v)_{v \in \mathcal{V}}$ are $\mathcal{E}_2$-independent in the sense of \cite{mlotkowski2004}, then by \cite[Theorem 5.2]{SpWy2016}, we have
	\begin{equation} \label{eq: epsilon independence moment formula}
		\varphi(a_1 \dots a_k) = \sum_{\pi \in \mathcal{NC}^{\varepsilon}[c]} K_{\pi}^{\free}[a_1,\dots,a_k],
	\end{equation}
	where $\mathcal{NC}^\varepsilon[c]$ is the set of partitions described as follows:  $\pi \in \mathcal{NC}^\varepsilon[c]$ if $\pi \leq \ker(c)$ and whenever two blocks $B$ and $B'$ in $\pi$ are crossing, then $(c(B),c(B')) \in \varepsilon = \mathcal{E}_2$.  Note that this is equivalent to conditions (1) and (3) in Definition \ref{def: compatible partitions} for $\mathcal{P}(c,\mathcal{G})$.  Meanwhile, Definition \ref{def: compatible partitions} (2) is vacuously true when $\mathcal{E}_1 = \mathcal{V} \times \mathcal{V}$.  Thus, $\mathcal{NC}^{\varepsilon}[c] = \mathcal{P}(c,\mathcal{G})$, and so \eqref{eq: epsilon independence moment formula} reduces to \eqref{eqn:formula moments}, which means that $(A_v)_{v \in \mathcal{V}}$ are bigraph independent with respect to $\mathcal{G}$.  Conversely, if $(A_v)_{v \in \mathcal{V}}$ are di-graph independent with respect to $\mathcal{G}$, then they satisfy \eqref{eq: epsilon independence moment formula}; this means that the joint moments agree with those of $\varepsilon$-independent copies of $A_1$, \dots, $A_k$, and hence they are $\varepsilon$-independent.
\end{proof}

Next, we prove Proposition \ref{prop: pairwise independence} (7) concerning BMT independence.  First, we recall the definition of BMT independence from \cite{arizmendi2025bmt}.  Let $G = (V,E)$ be a digraph with $E$ irreflexive.  First, from \cite[Definition 2.8]{arizmendi2025bmt}, for $c: [k] \to V$ and for $i_1, i_2 \in [k]$, we say that $i_1 \sim_{G,c} i_2$  if
\begin{itemize}
	\item $c(i_1) = c(i_2)$;
	\item whenever $j$ is between $i_1$ and $i_2$ (that is, $i_1 < j < i_2$ or $i_2 < j < i_1$), then $c(j) = c(i_1)$ or $(c(j),c(i_1)) \in E$.
\end{itemize}
It is straighforward to check that $\sim_{G,c}$ is an equivalence relation.  Let $\ker_G(c)$ be the partition of $[k]$ whose blocks are the equivalence classes.  Now by \cite[Definition 3.4]{arizmendi2025bmt}, given a digraph $G$, a $\mathrm{C}^*$-probability space $(A,\varphi)$ and $*$-subalgebras $(A_v)_{v \in V}$, we say that $(A_v)_{v \in V}$ are BMT-independent with respect to $G$ if for every $c: [k] \to V$ and $a_j \in A_{c(j)}$, we have
\begin{equation} \label{eq: BMT definition}
	\varphi(a_1 \dots a_k) = \prod_{B \in \ker_G(c)} \varphi( \vec{\prod_{j \in B}} a_j),
\end{equation}
where $\vec{\prod}$ denotes the product of the terms in increasing order from left to right.

\begin{proposition} \label{prop: BMT independence}
	Let $\mathcal{G}$ be a bigraph with $\mathcal{E}_2 \supseteq \mathcal{E}_1 \cap \overline{\mathcal{E}}_1 \setminus \Delta$.  Then bigraph independence with respect to $\mathcal{G}$ is equivalent to BMT independence with respect to $G = (\mathcal{V},\overline{\mathcal{E}}_1 \setminus \Delta)$ in the sense of \cite{arizmendi2025bmt}.
\end{proposition}

To this end, we use the following lemma:

\begin{lemma}
	Fix a bigraph $\mathcal{G}$ and $c: [k] \to \mathcal{V}$.  Assume that $\mathcal{E}_2 \subseteq \mathcal{E}_1 \cap \overline{\mathcal{E}}_1 \setminus \Delta$.  Then $\pi \in \mathcal{P}(c,\mathcal{G})$ if and only if $\pi \leq \ker_{\mathcal{G}}(c)$ and the restriction of $\pi$ to each block of $\ker_G(c)$ is non-crossing.
\end{lemma}

\begin{proof}
	Assuming that $\mathcal{E}_2 \subseteq \mathcal{E}_1 \cap \overline{\mathcal{E}}_1 \setminus \Delta$, we claim that Definition \ref{def: compatible partitions} (3) can be replaced by
	\begin{itemize}
		\item[(3')] For each $v \in V$, the restriction of $\pi$ to $c^{-1}(v)$ is non-crossing.
	\end{itemize}
	Indeed, (3) implies that crossings between two blocks of the same color are impossible because $\mathcal{E}_2$ is irreflexive, so (3') holds.  On the other hand, suppose (3') holds and that $i_1 < j_1 < i_2 < j_2$ and $i_1 \sim_\pi i_2$ and $j_1 \sim_\pi j_2$ and $i_1 \not \sim_\pi j_1$.  Then $c(i_1) \neq c(j_1)$ by (3').  Meanwhile, two applications of Definition \ref{def: compatible partitions} (2) imply that $(c(i_1),c(j_1)) \in \mathcal{E}_1$ and $(c(j_1),c(i_1)) \in \mathcal{E}_1$.  Hence, $(c(i_1),c(j_1)) \in \mathcal{E}_1 \cap \overline{\mathcal{E}}_1 \setminus \Delta = \mathcal{E}_2$, so (3) holds.

	Now we prove the main claim that $\pi \in \mathcal{P}(c,\mathcal{G})$ if and only if $\pi \leq \ker_G(\pi)$ and the restriction of $\pi$ to each block of $\ker_G(c)$ is non-crossing.  Indeed, conditions (1) and (2) of Definition \ref{def: compatible partitions} are equivalent to the condition that if $i_1 \sim_\pi i_2$, then $i_1 \sim_{G,c} i_2$, which is equivalent to $\pi \leq \ker_G(c)$.  Moreover, assuming that (1) and (2) hold, (3') implies that the restriction of $\pi$ to each block of $\ker_G(c)$ is non-crossing, since each block of $\ker_G(\pi)$ is monochromatic.  Conversely, assume that $\pi \leq \ker_G(c)$ and the restriction of $\pi$ to each block of $\ker_G(c)$ is non-crossing.  Suppose that $v \in V$ and $i_1 < j_1 < i_2 < j_2$ with $i_1 \sim_\pi i_2$ and $j_1 \sim_\pi j_2$ and $c(i_1) = c(j_1) = c(i_2) = c(j_2) = v$.  Since $\pi \leq \ker_G(c)$, all the indices $j'$ between between $i_1$ and $i_2$ must satisfy $c(j') = v$ or $(v,c(j')) \in \mathcal{E}_1 \setminus \Delta$.  In particular, this holds for all indices between $i_1$ and $j_1$, and so $i_1$ and $j_1$ are in the same block of $\ker_\pi(G)$.  Since $\pi$ restricted to this block is non-crossing, $i_1$ and $j_1$ must be in the same block of $\pi$.  Hence, $\pi$ restricted to $c^{-1}(v)$ is non-crossing, so (3') holds.
\end{proof}

\begin{proof}[Proof of Proposition \ref{prop: BMT independence}]
	Let $(A,\varphi)$ be a $\mathrm{C}^*$-probability space and let $(A_v)_{v \in \mathcal{V}}$ be $*$-subalgebras that are bigraph independent with respect to $\mathcal{G}$.  Let $k \in \bN$, let $c: [k] \to V$, and let $a_j \in A_{c(j)}$ for $j \in [k]$.  Since $\pi \in \mathcal{P}(c,\mathcal{G})$ if and only if $\pi \leq \ker_G(c)$ and $\pi$ is non-crossing on each block of $\ker_G(c)$, there is a bijection
	\begin{equation} \label{eq: first bijection}
		\mathcal{P}(c,\mathcal{G}) \to \prod_{B \in \ker_G(c)} \mathcal{NC}(B)
	\end{equation}
	given by sending $\pi$ to $(\pi|_B)_{B \in \ker_G(c)}$.  Thus, \eqref{eqn:formula moments} leads to
	\begin{align*}
		\varphi(a_1 \dots a_k) & = \sum_{\pi \in \mathcal{P}(c,\mathcal{G})} K_\pi^{\free}[a_1,\dots,a_k]                     \\
		                       & = \prod_{B \in \ker_G(\pi)} \sum_{\pi_B \in \mathcal{NC}(B)} K_{\pi_B}^{\free}[a_j: j \in B] \\
		                       & = \prod_{B \in \ker_G(\pi)} \varphi(\vec{\prod}_{j \in B} a_j),
	\end{align*}
	where the arguments for $K_{\pi_B}^{\free}[a_j: j \in B]$ are ordered from left to right, and where the last equality follows from the free moment-cumulant formula applied to $\varphi(\vec{\prod}_{j \in B} a_j)$.  This shows that $(A_v)_{v \in \mathcal{V}}$ satisfy \eqref{eq: BMT definition}, that is, they are BMT independent with respect to $G$.  The reserve manipulations show that conversely if $(A_v)_{v \in V}$ are BMT independent with respect to $G$, then they are bigraph independent with respect to $\mathcal{G}$.
\end{proof}

To conclude the proof of Proposition \ref{prop: pairwise independence}, we also record the following easy observation.

\begin{lemma} \label{lem: restriction to subgraph}
	Let $\mathcal{G}$ be a bigraph.  Let $\mathcal{V}' \subseteq \mathcal{V}$ and let $\mathcal{E}_j' = \mathcal{E}_j \cap (\mathcal{V}' \times \mathcal{V}')$ for $j = 1$, $2$.  Then $\mathcal{G}' = (\mathcal{V}',\mathcal{E}_1',\mathcal{E}_2')$ is a bigraph.  If $*$-subalgebras $(A_v)_{v \in \mathcal{V}}$ in a $\mathrm{C}^*$-probability space $(A,\varphi)$ are $\mathcal{G}$-independent, then $(A_v)_{v \in \mathcal{V}'}$ are $\mathcal{G}'$-independent.
\end{lemma}

\begin{proof}
	First, note that if $c: [k] \to \mathcal{V}'$, then $\mathcal{P}(c,\mathcal{G}) = \mathcal{P}(c,\mathcal{G}')$, which follows by inspection from Definition \ref{def: compatible partitions}.  From this, the conclusion is immediate from Definition \ref{def: bigraph independence}.
\end{proof}

\begin{proof}[Proof of Proposition \ref{prop: pairwise independence}]
	We have already shown claims (6) and (7) in Proposition \ref{prop: epsilon independence} and \ref{prop: BMT independence} respectively.  For the rest of the proof, assume without loss of generality that $\mathcal{E}_2 \subseteq \mathcal{E}_1 \cap \overline{\mathcal{E}}_1$ (see Remark).

	For claim (1), suppose that $(v,w) \in \mathcal{E}_{\Bool}$.  By Lemma \ref{lem: restriction to subgraph}, $A_v$ and $A_w$ are bigraph independent with respect to $(\{v,w\}, \varnothing,\varnothing)$.  By Proposition \ref{prop: BMT independence}, they are BMT independent with respect to $(\{v,w\},\varnothing)$, which means that they are Boolean independent by \cite[Proposition 3.12(ii)]{arizmendi2025bmt}.

	Claim (2) follows in the same way.  If $(v,w) \in \mathcal{E}_{\mono}$, then $A_v$ and $A_w$ are bigraph independent with respect to $(\{v,w\}, \{(v,w)\}, \varnothing)$, hence BMT-independent with respect to $(\{v,w\}, \{(w,v)\})$, hence monotone independent by \cite[Proposition 3.12(iii)]{arizmendi2025bmt}.

	Claim (3) on the anti-monotone case follows by symmetry.

	Claim (5) for the tensor case follows in the same way using BMT independence over the graph $(\{v,w\},\{(v,w),(w,v)\})$ and \cite[Proposition 3.12(i)]{arizmendi2025bmt}.

	For claim (4) on the free case, we similarly use Proposition \ref{prop: epsilon independence} to show that $A_v$ and $A_w$ are $\varepsilon$-independent with respect to $(\{v,w\},\varnothing)$, which means that they are freely independent by the results of \cite{mlotkowski2004,SpWy2016}.
\end{proof}

\begin{remark} \label{rem: restrictions of bigraph independence}
	The tensor case in claim (5) can alternatively be proved using $\varepsilon$-independence.  Of course, the reasoning of our proof also applies more generally:  One can use Lemma \ref{lem: restriction to subgraph} in conjunction with Propositions \ref{prop: epsilon independence} and \ref{prop: BMT independence} to describe the behavior of $(A_v)_{v \in \mathcal{V}'}$ for any subset of the vertices where the induced sub-bigraph $\mathcal{G}'$ has one of the special forms in Propositions \ref{prop: epsilon independence} and \ref{prop: BMT independence}.
\end{remark}

\begin{figure}
	\renewcommand{\arraystretch}{1.15}
	$$
		\begin{tabular}{|l|c|}  \hline
			$\mathcal{E}_1 = \emptyset, $                                                & boolean independence       \\
			\hline
			$\mathcal{E}_1$ total order                                                  & monotone independence      \\
			\hline
			$\mathcal{V}\times\mathcal{V}= {\mathcal{E}}_1$, $\mathcal{E}_2 = \emptyset$ & free independence          \\
			\hline
			$\mathcal{V}\times\mathcal{V}= {\mathcal{E}}_1=\mathcal{E}_2$                & tensor independence        \\
			\hline

			$\mathcal{V}\times\mathcal{V}= {\mathcal{E}}_1$                              & $\varepsilon$-independence \\
			\hline
			$\mathcal{E}_2 = \bar{\mathcal{E}}_1 \cap \mathcal{E}_1$                     & BMT independence           \\
			\hline
			$\mathcal{E}_1$ partial order                                                & BM independence            \\
			\hline
		\end{tabular}
	$$
	\caption{\label{tab:relations} Relation of the $\mathcal{G}$-independences to other independences. We have assumed $\mathcal{E}_2 \subset \mathcal{E}_1\cap\bar{\mathcal{E}}_1$.}
\end{figure}

\subsection{Bigraph independence via Boolean and classical cumulants} \label{subsec: boolean and classical cumulants}

In this section, we give a formula for joint moments of $\mathcal{G}$-independent variables in terms of Boolean cumulants rather than free cumulants.  This formula will be needed in \S \ref{subsec: tree independence}, \S \ref{sec:operad associativity}, and \S \ref{sec:hilbertspacemodel}.  To this end, we need to introduce a certain subclass $\mathcal{P}(c,\mathcal{G})^0$ of $\mathcal{P}(c,\mathcal{G})$.

Recall the definition of $\prec$ from Definition \ref{def:nested}. In the sequel, we use $\leq$ to designate the containment order on partitions: $\sigma \leq \pi$, $\sigma,\pi$ partitions of $[k]$, if $\pi$ is obtained from $\sigma$ by merging some blocks of $\sigma$.

\begin{definition} \label{def: unobstructed nesting}
	Fix $\mathcal{G}$ and $c: [k] \to \mathcal{V}$.  Let $\pi \in \mathcal{P}(c,\mathcal{G})$.  We define the \emph{unobstructed nesting} relation $\prec_{un}$ on $\pi$ as follows:  let $B\neq B'$, $B \prec_{un} B'$ if and only if
	\begin{itemize}
		\item $c(B) = c(B')$;
		\item $B \prec B'$ (equivalently, $B'$ is nested inside of $B$);
		\item If $B''$ satisfies $B \prec B'' \prec B'$, then $(c(B),c(B'')) \in \mathcal{E}_2 \cup \Delta$.
	\end{itemize}
\end{definition}

\begin{definition}
	We denote by $\mathcal{P}(c,\mathcal{G})^0$ the set of $\pi \in \mathcal{P}(c,\mathcal{G})$ such that $\prec_{un}$ is the empty relation.  Equivalently, $\pi \in \mathcal{P}(c,\mathcal{G})^0$ whenever $\pi \in \mathcal{P}(c,\mathcal{G})$ and for every $B$ and $B' \in \pi$, if $c(B) = c(B')$ and $B \prec B'$, there exists some $B''$ such that $B \prec B'' \prec B''$ and $(c(B),c(B'')) \not \in \mathcal{E}_2 \cup \Delta$.
\end{definition}

\begin{remark}
	In the Definition, \ref{def: unobstructed nesting}, last condition, it is also necessary that $(c(B),c(B'')) \in \mathcal{E}_1 \cap \overline{\mathcal{E}}_1$ because $\pi \in \mathcal{P}(c,\mathcal{G})$ and $B \prec B'' \prec B'$.  Hence, we see that $(c(B),c(B'')) \in \mathcal{E}_{\free} \setminus \Delta$.
\end{remark}

The first main goal of this section is to prove the following statement.

\begin{proposition} \label{prop: formula via Boolean cumulants}
	Fix a bigraph $\mathcal{G}$, a $\mathrm{C}^*$-probability space $(A,\varphi)$, and $*$-subalgebras $(A_v)_{v \in \mathcal{V}}$.  Then $(A_v)_{v \in \mathcal{V}}$ are bigraph independent with respect to $\mathcal{G}$ if and only if the following holds:  For $k \in \bN$ and $c: [k] \to \mathcal{V}$ and $a_j \in A_{c(j)}$ for $j \in [k]$, we have
	\begin{equation} \label{eq: joint moments via Boolean cumulants}
		\varphi(a_1 \dots a_k) = \sum_{\pi \in \mathcal{P}(c,\mathcal{G})^0} K_\pi^{\Bool}[a_1,\dots,a_k].
	\end{equation}
\end{proposition}

The main idea of the proof is that we can substitute the expression for Boolean cumulants in terms of free cumulants into \eqref{eq: joint moments via Boolean cumulants} in order to obtain \eqref{eqn:formula moments}.  To achieve this, we need to show that a partition in $\mathcal{P}(c,\mathcal{G})$ can be uniquely expressed by taking a partition in $\mathcal{P}(c,\mathcal{G})^0$ and subdividing each of its blocks according to an irreducible non-crossing partition (see Lemma \ref{lem: P zero decomposition}).  We begin with some technical lemmas on $\prec_{un}$ and $\mathcal{P}(c,\mathcal{G})$.

\begin{lemma} \label{lem: unobstructed nesting}
	Fix $c$ and $\mathcal{G}$ as above and let $\pi \in \mathcal{P}(c,\mathcal{G})$.
	\begin{enumerate}
		\item $\prec_{un}$ is a strict partial order on $\pi$.
		\item For every $B \in \pi$, the set $\{B' \in \pi: B' \prec_{un} B\}$ is totally ordered with respect to $\prec_{un}$.
		\item For every $B \in \pi$, there exists a unique $B' \prec_{un} B$ which is minimal with respect to $\prec$.
	\end{enumerate}
\end{lemma}


\begin{proof}
	(1) First, note that since $\prec_{un} \subseteq \prec$ and $\prec$ is a strict partial order, we have that $\prec_{un}$ is antisymmetric and irreflexive.  So it remains to show that $\prec_{un}$ is transitive.  Suppose that $B_0 \prec_{un} B_1 \prec_{un} B_2$.  Then $c(B_0) = c(B_1) = c(B_2)$ and $B_0 \prec B_1 \prec B_2$, which verifies the first two conditions of Definition \ref{def: unobstructed nesting}.  For the third condition, suppose that $B_0 \prec B' \prec B_2$.  We apply Lemma \ref{lem: tetrachotomy} to $B_1$ and $B'$, and note that $B'$ and $B_1$ cannot be separated since $B' \prec B_2$ and $B_1 \prec B_2$.
	\begin{itemize}
		\item If $B'$ crosses $B_1$, then since $\pi \in \mathcal{P}(c,\mathcal{G})$, we must have $(c(B'),c(B_1)) \in \mathcal{E}_2 \subseteq \mathcal{E}_2 \cup \Delta$ as desired.
		\item If $B_1 \prec B'$, then $B_1 \prec B' \prec B_2$ and hence $(c(B'),c(B_1)) \in \mathcal{E}_2 \cup \Delta$ by definition of $B_1 \prec_{un} B_2$.
		\item If $B' \prec B_1$, then $B_0 \prec B' \prec B_2$, and hence $(c(B_0),c(B')) \in \mathcal{E}_2 \cup \Delta$ by definition of $B_0 \prec_{un} B_1$.
	\end{itemize}
	Hence, $B_0 \prec_{un} B_2$.

	(2) Suppose that $B_1 \prec_{un} B$ and $B_2 \prec_{un} B$.  Since $c(B_1) = c(B_2)$ and $\pi \in \mathcal{P}(c,\mathcal{G})$, the blocks $B_1$ and $B_2$ cannot cross.  They also cannot be separated since $B_1 \prec B$ and $B_2 \prec B$.  Hence, by Lemma \ref{lem: tetrachotomy}, we have $B_1 \prec B_2$ or $B_2 \prec B_1$.  If $B_1 \prec B_2$, then we must also have $B_1 \prec_{un} B_2$.  Indeed, it is clear that $c(B_1) = c(B_2)$ and $B_1 \prec B_2$; moreover, if $B_1 \prec B' \prec B_2$, then $B_1 \prec B' \prec B$, and hence $(c(B_1),c(B')) \in \mathcal{E}_2 \cup \Delta$ by definition of $B_1 \prec_{un} B$.

	(3) This is immediate from (2).
\end{proof}

\begin{lemma} \label{lem: P zero envelope}
	Fix $c$ and $\mathcal{G}$ as above.  Then for each $\pi \in \mathcal{P}(c,\mathcal{G})$, there exists a unique $\sigma \in \mathcal{P}(c,\mathcal{G})^0$ such that $\sigma \geq \pi$ and for every $B \in \sigma$, $\min B \sim_\pi \max B$.
\end{lemma}

\begin{proof}
	For each block $B \in \pi$, let $m(B)$ be the minimal block such that $m(B) \prec_{un} B$ using Lemma \ref{lem: unobstructed nesting}.  Define $B \sim_{un} B'$ if $m(B) = m(B')$.  Let $\sigma$ be the partition whose blocks are the unions of the equivalence classes of $\sim_{un}$.  This means in particular that if $B \prec_{un} B'$ in $\pi$, then $B$ and $B'$ are in the same block of $\sigma$.

	We first show that for $B \in \sigma$, we have $\min B \sim_\sigma \max B$.  Each $B \in \sigma$ is the union of some equivalence class of $\sim_{un}$.  By construction of $\sim_{un}$, the equivalence class has a unique minimal element $B'$.  Since $B' \prec B''$ for all the other $B''$ in the equivalence class, we have that $\min B = \min B'$ and $\max B = \max B'$.  Hence, $\min B$ and $\max B$ are in the same block $B'$ of $\pi$.

	Next, we show that $\sigma \in \mathcal{P}(c,\mathcal{G})$, checking each condition of Definition \ref{def: compatible partitions}:
	\begin{enumerate}
		\item Clearly, the blocks of $\sigma$ are monochromatic since each equivalence class of $\sim_{un}$ is monochromatic.
		\item Suppose that $i_1 < j < i_2$ and $i_1$ and $i_2$ are in the same block $B$ of $\sigma$.  By making $i_1$ smaller and $i_2$ larger if necessary, we can assume without loss of generality that $i_1 = \min B$ and $i_2 = \max B$.  Hence, by the preceding paragraph, $i_1 \sim_\pi i_2$.  Since $\pi \in \mathcal{P}(c,\mathcal{G})$, we have $(c(i_1),c(j)) \in \mathcal{E}_2$.
		\item Suppose that two distinct blocks $B_1$ and $B_2$ of $\sigma$ cross.  If there exist blocks $B_1'$ and $B_2'$ of $\pi$ with $B_1' \subseteq B_1$ and $B_2' \subseteq B_2$ such that $B_1'$ and $B_2'$ cross, then we must have $(c(B_1'),c(B_2')) \in \mathcal{E}_2$ because $\pi \in \mathcal{P}(c,\mathcal{G})$, so in this case we are done.

		      Hence, suppose that no blocks of $\pi$ within $B_1$ and $B_2$ cross each other.  Recall that $\min B_1$ and $\max B_1$ are in the same block $B_1'$ of $\pi$, and $\min B_2$ and $\max B_2$ are in the same block $B_2'$ of $\pi$.  Note that $B_1'$ and $B_2'$ cannot be separated since then $B_1$ and $B_2$ could not cross.  We assumed that $B_1'$ and $B_2'$ do not cross, and therefore $B_1' \prec B_2'$ or $B_2 \prec B_1'$.  By symmetry, assume without loss of generality that $B_1' \prec B_2'$.  Now because $B_1$ crosses $B_2$, there must be some element $j$ of $B_1$ between $\min B_2'$ and $\max B_2'$.  The block $B_1''$ of $\pi$ that contains $j$ must be nested inside in $B_2'$ (since by assumption it cannot cross it).  Therefore, $B_1' \prec B_2' \prec B_1''$, and $B_1' \prec_{un} B_1''$ since $m(B_1'') = B_1'$.  By definition of $\prec_{un}$, we have $(c(B_1'),c(B_2')) \in \mathcal{E}_2 \cup \Delta$.

		      Finally, note that $c(B_1)$ cannot equal $c(B_2)$; indeed, if this happened, then $B_1' \prec B_2' \prec B_1''$ would imply that $B_1' \prec_{un} B_2'$ by the same reasoning as in Lemma \ref{lem: unobstructed nesting} (2), which would force $B_1$ to equal $B_2$.  Hence, $(c(B_1),c(B_2)) \in \mathcal{E}_2$ as desired.
	\end{enumerate}

	Next, we show that $\sigma \in \mathcal{P}(c,\mathcal{G})^0$.  Suppose for contradiction that $B_1 \prec_{un} B_2$ in $\sigma$.  For $j = 1, 2$, let $B_j'$ be the block of $\pi$ that contains $\min B_j$ and $\max B_j$.  We claim that $B_1' \prec_{un} B_2'$.  It is immediate that $c(B_1') = c(B_2')$ and $B_1' \prec B_2'$.  Suppose that $B'$ is a block of $\pi$ with $B_1' \prec B' \prec B_2'$.  Let $B$ be the block of $\sigma$ that contains $B'$, and we will show that $(c(B_1),c(B)) \in \mathcal{E}_2 \cup \Delta$.  If $B$ crosses $B_1$ or $B_2$, then this holds because $\sigma \in \mathcal{P}(c,\mathcal{G})$, and we are done.  Suppose that $B$ does not cross $B_1$ or $B_2$.  Then, because $B$ contains $B' \prec B_2'$, we must have that $B \prec B_2'$ and hence $B \prec B_2$.  Similarly, since $B_1$ contains $B_1' \prec B'$, we have $B_1' \prec B$, and hence also $B_1 \prec B$ since they do not cross.  Therefore, $B_1' \prec_{un} B_2'$.  It follows that $B_1'$ and $B_2'$ must be in the same block of $\sigma$, which contradicts that $B_1 \prec B_2$.

	Finally, we show that $\sigma$ is unique.  Let $\sigma'$ be another partition satisfying the conditions in the lemma.  First, we show that $\sigma \leq \sigma'$, or equivalently, if two blocks $B_1'$ and $B_2'$ of $\pi$ are in the same equivalence class, they must be in the same block of $\sigma'$.  It suffices to show that $B_1'$ must be in the same block as $m(B_1')$ in $\sigma'$.  Hence, it suffices to show that if $B_1' \prec_{un} B_2'$ in $\pi$, then $B_1'$ and $B_2'$ are in the same block of $\sigma'$.

	Suppose for contradiction that $B_1'$ is in the block $B_1$ and $B_2'$ is in the block $B_2$ in $\sigma$ with $B_1 \neq B_2$.  Then we claim that $B_1 \prec_{un} B_2$ in $\sigma'$.  Clearly, $B_1$ and $B_2$ are the same color.  Since $B_1$ and $B_2$ are the same color (hence cannot cross) and $B_1' \prec B_2'$, we must have $B_1 \prec B_2$.  Lastly, suppose that $B_1 \prec B \prec B_2$ for some block $B$ in $\sigma'$.  Let $B'$ be the block of $\pi$ containing $\max B$ and $\min B$, and let $B_1''$ be the block of $\pi$ containing $\max B_1$ and $\min B_1$.  Then $B_1'' \prec B' \prec B_2'$.  Then the desired conclusion that $(c(B_1''),c(B')) \in \mathcal{E}_2 \cup \Delta$ will follow if we can show that $B_1'' \prec_{un} B_2'$.  This is immediate if $B_1'' = B_1'$.  If $B_1'' \neq B_1'$, then we show that $B_1'' \prec_{un} B_1'$.  Clearly, they are the same color and $B_1'' \prec B_1'$.  Moreover, if $B_1'' \prec C \prec B_1'$ in $\pi$, then $C$ is either contained in $B_1$ or it crosses $B_1$ and hence $(c(B_1),c(C)) \in \mathcal{E}_2 \cup \Delta$.  So $B_1'' \prec_{un} B_1'' \prec_{un} B_2'$ as claimed.

	This completes the proof that $\sigma' \geq \sigma$.  To conclude that $\sigma' = \sigma$, it remains to show that if $B_1$ and $B_2$ are different blocks in $\sigma$, then they cannot be contained in the same block of $\sigma'$.  Of course, this is immediate if $B_1$ and $B_2$ are different colors.  Suppose $B_1$ and $B_2$ are the same color, and hence do not cross.
	\begin{itemize}
		\item Suppose that $B_1 \prec B_2$.  Since $\sigma \in \mathcal{P}(c,\mathcal{G})^0$, we cannot have $B_1 \prec_{un} B_2$, and so there exists another block $B$ of $\sigma$ with $B_1 \prec B \prec B_2$ and $(c(B_1),c(B)) \not \in \mathcal{E}_2 \cup \Delta$.  If $B_1$ and $B_2$ were in the same block of $\sigma'$, then this block would cross the block containing $B$, which contradicts $\sigma' \in \mathcal{P}(c,\mathcal{G})$.
		\item Of course, the same applies when $B_2 \prec B_1$.
		\item Suppose that $B_1$ and $B_2$ are separated, and they are contained in the same block $A$ of $\sigma'$.  Since nested blocks of $\sigma$ cannot be in the same block of $\sigma'$, the block $A$ must the union of separated blocks $A_1$, \dots, $A_m$ of $\sigma$ (listed in order from left to right).  By our hypotheses on $\sigma'$, $\min A$ and $\max A$ are in the same block of $\pi$.  But $\min A \in A_1$ and $\max A \in A_m$ are in different blocks of $\sigma$, hence different blocks of $\pi$, which is a contradiction.
	\end{itemize}
\end{proof}

\begin{definition}
	Given $\pi \in \mathcal{P}(c,\mathcal{G})$, the partition $\pi'$ in the previous lemma will be denoted $\operatorname{Env}^0(\pi)$ and called the $\mathcal{P}(c,\mathcal{G})^0$-\emph{envelope} of $\pi$.
\end{definition}

We also recall the following more standard definition:
\begin{definition}
	The \emph{interval envelope} $\operatorname{Env}^I(\pi)$ of a partition $\pi$ is the meet of all interval partitions $\sigma \geq \pi$.

	A partition of $[k]$ is \emph{irreducible} if $\operatorname{Env}^I(\pi) = \{[k]\}$.

	We denote by $\mathcal{P}_{\operatorname{irr}}(k)$ the \emph{irreducible partitions} and by $\mathcal{NC}_{\operatorname{irr}}(k)$ the irreducible non-crossing partitions.
\end{definition}

\begin{lemma} \label{lem: P zero decomposition}
	Fix $\mathcal{G}$ and $c: [k] \to \mathcal{V}$.  Then there is a bijection
	\[
		\Phi: \bigsqcup_{\sigma \in \mathcal{P}(c,\mathcal{G})^0} \prod_{B \in \sigma} \mathcal{NC}_{\operatorname{irr}}(B) \to \mathcal{P}(c,\mathcal{G}),
		\quad \Phi((\pi_B)_{B \in \sigma}) = \bigcup_{B \in \sigma} \pi_B.
	\]
	The inverse $\Phi^{-1}$ is obtained as follows:  Given $\pi \in \mathcal{P}(c,\mathcal{G})$, let $\sigma = \operatorname{Env}^0(\pi)$, and for $B \in \sigma$, let $\pi_B$ be the restriction of $\sigma$ to $B$, and then $\Phi^{-1}(\pi) = (\pi_B)_{B \in \sigma}$.
\end{lemma}

\begin{proof}
	Let $\Psi$ be the map that sends $\pi \in \mathcal{P}(c,\mathcal{G})$ to $(\pi|_B)_{B \in \sigma}$ where $\sigma = \operatorname{Env}^0(\pi)$.  To prove the lemma, we proceed by verifying that $\Phi$ and $\Psi$ map into the codomains that we have asserted, and that they are inverses of each other.

	First, we show that $\Phi$ maps into $\mathcal{P}(c,\mathcal{G})$.  Fix $\sigma \in \mathcal{P}(c,\mathcal{G})^0$ and $\pi_B \in \mathcal{NC}_{\operatorname{irr}}(B)$ for $B \in \sigma$.  Let $\pi = \Phi((\pi_B)_{B \in \sigma})$, and we will show that $\pi \in \mathcal{P}(c,\mathcal{G})$.  We verify the conditions of Definition \ref{def: compatible partitions} for $\pi$ in turn:
	\begin{enumerate}[(1)]
		\item The blocks of $\pi$ are monochromatic since $\pi \leq \sigma$.
		\item Suppose that $i_1 < j < j_2$ with $i_1 \sim_\pi i_2$.  Then $i_1 \sim_\sigma i_2$ since $\pi \leq \sigma$.  Hence, because $\sigma \in \mathcal{P}(c,\mathcal{G})$, we have $(c(i_1),c(j)) \in \mathcal{E}_1$.
		\item Suppose that two blocks $B_1'$ and $B_2'$ in $\pi$ cross.  Since each $\pi_B$ is non-crossing, $B_1'$ and $B_2'$ must be in different blocks of $\sigma$.  Call these blocks $B_1$ and $B_2$.  Then $B_1$ and $B_2$ are crossing.  Hence, $(c(B_1),c(B_2)) \in \mathcal{E}_2$ as desired.
	\end{enumerate}

	Next, we show that $\Psi$ maps into $\bigsqcup_{\sigma \in \mathcal{P}(c,\mathcal{G})^0} \prod_{B \in \sigma} \mathcal{NC}_{\operatorname{irr}}(B)$.  Fix $\pi \in \mathcal{P}(c,\mathcal{G})$, and let $\sigma = \operatorname{Env}^0(\pi)$.  Let $\pi_B = \pi|_B$ for $B \in \sigma$.  By construction and Lemma \ref{lem: P zero envelope}, $\sigma \in \mathcal{P}(c,\mathcal{G})^0$.   Since $B$ is monochromatic and since $\pi \in \mathcal{P}(c,\mathcal{G})$, each $\pi_B$ is non-crossing.  Moreover, by construction or by Lemma \ref{lem: P zero envelope}, for each $B \in \sigma$, we have $\min B \sim_\pi \max B$, which means that the minimal and maximal elements of $\pi_B$ are in the same block of $\pi_B$, which means that $\pi_B$ is irreducible.

	Next, we show that $\Psi \circ \Phi = \id$.  Fix $\sigma \in \mathcal{P}(c,\mathcal{G})^0$ and $\pi_B \in \mathcal{NC}_{\operatorname{irr}}(B)$ for $B \in \sigma$.  Let $\pi = \Phi((\pi_B)_{B \in \sigma})$, and we will show that $\Psi(\pi) = (\pi_B)_{B \in \sigma}$.  Since the $\pi_B$'s are uniquely determined from $\sigma$, it suffices to show that $\operatorname{Env}^0(\pi) = \sigma$.  By the preceding paragraph $\sigma \in \mathcal{P}(c,\mathcal{G})^0$ and by construction $\sigma \geq \pi$.  Moreover, for each $B \in \sigma$, since $\pi_B$ is irreducible and non-crossing, $\min B$ and $\max B$ are in the same block of $\pi$.  Therefore, by the uniqueness claim in Lemma \ref{lem: P zero envelope}, $\sigma$ must equal $\operatorname{Env}^0(\pi)$.

	Finally, we show that $\Phi \circ \Psi = \id$.  Fix $\pi \in \mathcal{P}(c,\mathcal{G})$ and let $(\pi_B)_{B \in \sigma} = \Psi(\pi)$.  Since $\pi_B$ are groupings of the blocks of $\pi$, we have $\pi = \bigcup_{B \in \sigma} \pi = \Phi((\pi_B)_{\sigma \in B})$.
\end{proof}

\begin{proof}[Proof of Proposition \ref{prop: formula via Boolean cumulants}]
	Suppose that $(A_v)_{v \in \mathcal{V}}$ are bigraph independent with respect to $\mathcal{G}$.  Consider the expression
	\begin{equation} \label{eq: start of Boolean proof}
		\sum_{\sigma \in \mathcal{P}(c,\mathcal{G})^0} K_\sigma^{\Bool}[a_1,\dots,a_k].
	\end{equation}
	Recall the formula for expressing Boolean cumulants in terms of free cumulants (see \cite{Lehner2002cumulants}):
	\begin{equation} \label{eq: free to Boolean cumulants}
		K_k^{\Bool}[a_1,\dots,a_k] = \sum_{\pi \in \mathcal{NC}_{\operatorname{irr}}(k)} K_\pi^{\free}[a_1,\dots,a_k].
	\end{equation}
	On each block $B$ of $\sigma$ in \eqref{eq: start of Boolean proof} substitute \eqref{eq: free to Boolean cumulants} to obtain
	\[
		\sum_{\sigma \in \mathcal{P}(c,\mathcal{G})^0} \prod_{B \in \sigma} \sum_{\pi_B \in \mathcal{NC}_{\operatorname{irr}}(B)} K_{\pi_B}^{\free}[a_j: j \in B].
	\]
	By Lemma \ref{lem: P zero decomposition}, this can be equivalently expressed as
	\[
		\sum_{\pi \in \mathcal{P}(c,\mathcal{G})} K_\pi^{\free}[a_1,\dots,a_k],
	\]
	which is $\varphi(a_1 \dots a_k)$ by \eqref{eqn:formula moments}.  The converse implication that \eqref{eq: joint moments via Boolean cumulants} implies \eqref{eqn:formula moments} proceeds in a similar way.
\end{proof}

We next give a formula for joint moments of $\mathcal{G}$-independent variables in terms of classical cumulants.  For this, we introduce a third set of partitions besides $\mathcal{P}(c,\mathcal{G})$ and $\mathcal{P}(c,\mathcal{G})$.  The set $\tilde{\mathcal{P}}(c,\mathcal{G})$ is obtained from $\mathcal{P}(c,\mathcal{G})$ by relaxing the condition that blocks with same color can not have a crossing; in other words, we replace $\mathcal{E}_2$ with $\mathcal{E}_2 \cup \Delta$ in Definition \ref{def: compatible partitions} (3).

\begin{definition}
	Fix a bigraph $\mathcal{G}$ and $k \in \bN$ and $c: [k] \to \mathcal{V}$ be a string of colors $(c_i \in \mathcal{I})$. The sub-set $\tilde{\mathcal{P}}(c,\mathcal{G})$ comprises all partitions of $[k]$ satisfying the following conditions:
	\begin{enumerate}
		\item If $i$ and $j$ are in the same block $B \in \pi$, then $c(i) = c(j)$.
		\item \label{def: compatible partitions item:two}If $i_1 < j < i_2$ and $i_1 \sim_\pi i_2$, then $(c(i_1), c(j)) \in \mathcal{E}_1$.
		\item \label{def: compatible partitions item:three}If $i_1 < j_1 < i_2 < j_2$ and $i_1 \sim_\pi i_2$ and $j_1 \sim_\pi j_2$, then $(c(i_2),c(j_1)) \in \mathcal{E}_2 \cup \Delta$.
	\end{enumerate}
\end{definition}

We then have the following expression of joint moments.

\begin{proposition} \label{prop: formula via classical cumulants}
	Fix a bigraph $\mathcal{G}$, a $\mathrm{C}^*$-probability space $(A,\varphi)$, and $*$-subalgebras $(A_v)_{v \in \mathcal{V}}$.  Then $(A_v)_{v \in \mathcal{V}}$ are bigraph independent with respect to $\mathcal{G}$ if and only if the following holds:  For $k \in \bN$ and $c: [k] \to \mathcal{V}$ and $a_j \in A_{c(j)}$ for $j \in [k]$, we have
	\begin{equation} \label{eq: joint moments via classical cumulants}
		\varphi(a_1 \dots a_k) = \sum_{\pi \in \tilde{\mathcal{P}}(c,\mathcal{G})} K_\pi^{\operatorname{class}}[a_1,\dots,a_k].
	\end{equation}
\end{proposition}

\begin{definition}
	The \emph{non-crossing envelope} $\operatorname{Env}^{\operatorname{nc}}(\pi)$ of a partition $\pi$ is the meet of all non-crossing partitions $\sigma \geq \pi$.

	A partition of $[k]$ is said to be \emph{connected} if $\operatorname{Env}^{\operatorname{nc}}(\pi) = \{[k]\}$.

	We denote the set of \emph{connected partitions} of $[k]$ by $\mathcal{P}_{\operatorname{conn}}(k)$.  Similarly, if $B \subseteq [n]$, we denote by $\mathcal{P}_{\operatorname{conn}}(B)$ the corresponding set of partitions obtained by relabeling the elements of $B$ by $1$, \dots, $k$ in order.
\end{definition}

\begin{lemma} \label{lem: colorwise non-crossing envelope}
	Fix a bigraph $\mathcal{G}$, $k \in \bN$, and $c: [k] \to \mathcal{V}$.  Let $\pi \in \mathcal{P}(k)$ be a partition such that each block is monochromatic with respect to $c$.  Let $\sigma$ be the partition obtained by, for each $v \in \mathcal{V}$, replacing $\pi|_{c^{-1}(v)}$ with its non-crossing envelope.  Then $\pi \in \tilde{\mathcal{P}}(c,\mathcal{G})$ if and only if $\sigma \in \mathcal{P}(c,\mathcal{G})$.
\end{lemma}

\begin{definition}
	In the situation of the previous lemma, we call $\sigma$ the \emph{colorwise non-crossing envelope} of $\pi$.
\end{definition}

\begin{proof}[Proof of Lemma \ref{lem: colorwise non-crossing envelope}]
	Suppose that $\pi \in \tilde{\mathcal{P}}(c,\mathcal{G})$.  We check each of the necessary properties:
	\begin{enumerate}[(1)]
		\item Each block of $\sigma$ is monochromatic since it is obtained by joining together blocks of $\pi$ of the same color.
		\item Suppose that $i_1 < j < i_2$ and $i_1 \sim_\sigma i_2$.  Let $B$ be the block of $\sigma$ containing $i_1$ and $i_2$.  We claim that there is at least one block of $\pi|_B$ that intersects both $\{1,\dots,j-1\}$ and $\{j+1,\dots,k\}$.  If this were not the case, then $\pi|_B$ could be partitioned into the set $\pi_1$ of blocks contained in $\{1,\dots,j-1\}$ and the set $\pi_2$ blocks contained in $\{j+1,\dots,k\}$.  Then we could modify $\sigma$ by subdividing the block $B$ into the union $B_1$ of the blocks in $\pi_1$ and the union $B_2$ of the blocks in $\pi_2$.  This would contradict the fact that $\sigma$ is the least partition $\geq \pi$ that is non-crossing in each color.  Therefore, there is some block $B$ of $\pi$ and $i_1'$, $i_2' \in B'$ such that $i_1' < j < i_2'$.  Hence, $\pi \in \mathcal{P}(c,\mathcal{G})$ implies that $(c(i_1),c(j)) = (c(i_1'),c(j)) \in \mathcal{E}_1$.
		\item Suppose that two blocks $B_1$ and $B_2$ in $\sigma$ cross.  Since $\sigma$ is defined as the colorwise non-crossing envelope of $\pi$, it is impossible for $B_1$ and $B_2$ to cross when $c(B_1) = c(B_2)$.  So $c(B_1) \neq c(B_2)$.

		      We claim that there is some block $B_1'$ of $\pi$ contained in $B_1$ that crosses $B_2$.  Here note that the blocks of $\pi|_{B_1}$ are disjoint from $B_2$, and so $\pi|_{B_1} \cup \{B_2\}$ is a partition of some subset of $[k]$ to which we can apply Lemma \ref{lem: tetrachotomy}.  Write $B_2 = \{i_1 < \dots < i_m\}$.  If no blocks of $\pi|_{B_1}$ crossed $B_2$, then the blocks of $\pi|_{B_1}$ could be partitioned based on whether they fall into $\{i_1+1,\dots,i_2-1\}$, \dots, $\{i_{m-1}+1,\dots,i_m-1\}$ or $[k] \setminus \{i_1,i_1+1,\dots,i_m\}$; call the resulting subpartitions $\tau_1$, \dots, $\tau_m$.  Since $B_1$ crosses $B_2$, at least two of the $\tau_j$'s are non-empty.  Since the $\tau_j$'s cannot cross each other, this would contradict the fact that the colorwise non-crossing envelope of $\pi$ is $\sigma$.

		      Hence, let $B_1' \in \pi|_{B_1}$ cross $B_2$.  By repeating this argument, we see that since $B_2$ crosses $B_1'$, there must be some $B_2' \in \pi|_{B_2}$ that crosses $B_1'$, or otherwise we could contradict $\sigma$ being the colorwise non-crossing envelope of $\pi$ (as before by subdividing the block $B_1$).

		      Therefore, we have blocks $B_1' \in \pi|_{B_1}$ and $B_2' \in \pi|_{B_2}$ that cross each other.  Hence, since $\pi \in \tilde{\mathcal{P}}(c,\mathcal{G})$, we have $(c(B_1),c(B_2)) = (c(B_1'),c(B_2')) \in \mathcal{E}_2$.
	\end{enumerate}
	Therefore, $\sigma \in \mathcal{P}(c,\mathcal{G})$.

	Conversely, suppose that $\sigma \in \mathcal{P}(c,\mathcal{G})$, and we will show that $\pi \in \tilde{\mathcal{P}}(c,\mathcal{G})$.
	\begin{enumerate}[(1)]
		\item Since $\pi \leq \sigma$, each block of $\pi$ is monochromatic.
		\item Suppose that $i_1 < j < i_2$ and $i_1 \sim_\pi i_2$.  Then $i_1 \sim_\sigma i_2$, and hence $(c(i_1),c(j)) \in \mathcal{E}_1$.
		\item Suppose that two blocks $B_1'$ and $B_2'$ of $\pi$ cross.  If $B_1'$ and $B_2'$ have the same color, then $(c(B_1'),c(B_2')) \in \Delta$.  On the other hand, if $B_1'$ and $B_2'$ have different colors, then the blocks $B_1$ and $B_2$ of $\sigma$ that contain $B_1'$ and $B_2'$ will cross each other, and hence $(c(B_1'),c(B_2')) = (c(B_1),c(B_2)) \in \mathcal{E}_2$.  Either way, $(c(B_1'),c(B_2')) \in \mathcal{E}_2 \cup \Delta$.
	\end{enumerate}
	Therefore, $\pi \in \tilde{\mathcal{P}}(c,\mathcal{G})$.
\end{proof}

\begin{lemma} \label{lem: colorwise NC decomposition}
	There is a bijection
	\[
		\Phi: \bigsqcup_{\sigma \in \mathcal{P}(c,\mathcal{G})} \prod_{B \in \sigma} {\mathcal{P}}_{\operatorname{conn}}(B) \to \tilde{\mathcal{P}}(c,\mathcal{G}), \quad
	 (\pi_B)_{B \in \sigma} \mapsto \bigcup_{B \in \sigma} \pi_B.
	\]  For $\pi \in \tilde{\mathcal{P}}(c,\mathcal{G})$, the inverse $\Phi^{-1}(\pi)$ is given by taking $\sigma$ to be the colorwise non-crossing envelope of $\pi$ and $\pi_B = \pi|_B$ for $B \in \sigma$.
\end{lemma}

\begin{proof}
	Let $\Psi$ be the map that sends $\pi \in \mathcal{P}(c,\mathcal{G})$ to $(\pi|_B)_{B \in \sigma}$ where $\sigma$ is the colorwise non-crossing envelope of $\pi$.  To prove the lemma, we proceed by verifying that $\Phi$ and $\Psi$ map into the codomains that we have asserted, and that they are inverses of each other.

	First, suppose that $\sigma \in \mathcal{P}(c,\mathcal{G})$ and $\pi_B \in \mathcal{P}_{\operatorname{conn}}(B)$ for $B \in \sigma$.  We will show that $\pi = \Phi((\pi_B)_{B \in \sigma}) = \bigcup_{B \in \sigma} \pi_B$ is in $\tilde{\mathcal{P}}(c,\mathcal{G})$.  Note that the colorwise non-crossing envelope of $\pi$ is $\sigma$.  Indeed, let $\pi'$ denote the colorwise non-crossing envelope.  Then $\sigma$ is colorwise non-crossing and $\sigma \geq \pi$, so $\sigma \geq \pi'$.  Moreover, since each $\pi_B$ is connected, the blocks of $\pi_B$ must be one block in the colorwise non-crossing envelope, and so $\sigma \leq \pi'$.  Then by Lemma \ref{lem: colorwise non-crossing envelope}, we have $\pi = \Phi((\pi_B)_{B \in \sigma}) \in \tilde{\mathcal{P}}(c,\mathcal{G})$.

	The above argument also shows that $\Psi(\pi) = \sigma$ since $\sigma$ is the colorwise non-crossing envelope of $\pi$.  Thus, $\Psi \circ \Phi = \id$.

	Next, suppose that $\pi \in \tilde{\mathcal{P}}(c,\mathcal{G})$.  We will show that $\Psi(\pi) \in \bigsqcup_{\sigma \in \mathcal{P}(c,\mathcal{G})} \prod_{B \in \sigma} \mathcal{P}_{\operatorname{conn}}(B) \to \mathcal{P}(c,\mathcal{G})$.  Let $\sigma$ be the colorwise non-crossing envelope.  By Lemma \ref{lem: colorwise non-crossing envelope}, we have $\sigma \in \mathcal{P}(c,\mathcal{G})$.  For $B \in \sigma$, let $\pi_B = \pi|_B$.  Since $\sigma$ is the colorwise non-crossing envelope of $\pi$, we must have that the non-crossing envelope of $\pi_B$ is $B$, meaning that $\pi_B$ is irreducible.  Thus, $\Psi$ maps into the asserted codomain.

	We also clearly have that $\Phi((\pi_B)_{B \in \sigma}) = \bigcup_{B\in \sigma} \pi_B = \pi$, which means that $\Phi \circ \Psi = \id$.
\end{proof}

\begin{proof}[Proof of Proposition \ref{prop: formula via classical cumulants}]
	Suppose that $(A_v)_{v \in \mathcal{V}}$ are bigraph independent with respect to $\mathcal{G}$.  Let $k \in \bN$ and $c: [k] \to \mathcal{V}$ and $a_j \in A_{c(j)}$ for $j = 1$, \dots, $k$.  By \eqref{eqn:formula moments},
	\[
		\varphi(a_1 \dots a_k) = \sum_{\sigma \in\mathcal{P}(c,\mathcal{G})} \prod_{B\in \sigma} K_{|B|}^{\free}(a_j: j \in B).
	\]
	Recall the expression for free cumulants in terms of classical cumulants from \cite{arizmendi2015relations}:
	\[
		K_k^{\free}(a_1,\dots,a_k) = \sum_{\pi \in \mathcal{P}_{\operatorname{conn}}(k)} K_\pi^{\operatorname{class}}[a_1,\dots,a_k].
	\]
	Applying this formula on each block $B$ of $\sigma$ yields
	\[
		\varphi(a_1 \dots a_k) = \sum_{\sigma \in\mathcal{P}(c,\mathcal{G})} \prod_{B \in \sigma} \sum_{\pi_B \in \mathcal{P}_{\operatorname{conn}}(B)} K_{\pi_B}^{\operatorname{class}}(a_j: j \in B).
	\]
	Then applying Lemma \ref{lem: colorwise NC decomposition}, we obtain the desired relation \eqref{eq: joint moments via classical cumulants}.  The proof of the converse statement proceeds similarly.
\end{proof}

We conclude by giving a formula that expresses the Boolean cumulants, rather than the moments, of elements from $\mathcal{G}$-independent algebras.

\begin{proposition} \label{prop: boolean cumulants}
	Let
	$(A_v)_{v \in \mathcal{V}}$ be $\mathcal{G}$-independent. Let $k \in \bN$, $c: [k] \to \mathcal{V}$, and let $a_1 \in A_{c(1)}$, \dots, $a_k \in A_{c(k)}$.  Then
	\begin{align}
		K^{\rm Bool}(a_1,\ldots,a_k) & = \sum_{\pi \in \mathcal{P}_{\operatorname{irr}}(c,\mathcal{G})} K^{\rm free}_{\pi}(a_1,\ldots,a_k) \label{eq: joint boolean cumulants via free cumulants} \\
		                             & = \sum_{\pi \in \mathcal{P}_{\operatorname{irr}}(c,\mathcal{G})^0} K^{\rm Bool}_{\pi}(a_1,\ldots,a_k)                                                      \\
		                             & = \sum_{\pi \in \tilde{\mathcal{P}}_{\operatorname{irr}}(c,\mathcal{G})} K^{\rm class}_{\pi}(a_1,\ldots,a_k),
	\end{align}
	where the subscript $\operatorname{irr}$ indicates the intersection of each class of partitions with the class of irreducible partitions of $[k]$.  Conversely, each of these formulas for the Boolean cumulants implies $\mathcal{G}$-independence.
\end{proposition}

\begin{proof}
	Let $\mathcal{I}(k)$ denote the set of interval partitions of $[k]$.  Recall that there is a bijection
	\[
		\mathcal{P}(k) \to \bigsqcup_{\sigma \in \mathcal{I}(k)} \prod_{B\in \sigma} \mathcal{P}_{\operatorname{irr}}(B)
	\]
	defined as follows:  Given $\pi \in \mathcal{P}(k)$, let $\sigma$ be its interval envelope and let $\pi_B = \pi|_B$ for $B \in \sigma$.  It is also easy to see that $\pi \in \mathcal{P}(c,\mathcal{G})$ if and only if $\pi_B \in \mathcal{P}(c|_B,\mathcal{G})$ for each $B \in \sigma$, since each nesting or crossing in $\pi$ must happen inside one of the irreducible partitions $\pi_B$.  The analogous statements also hold for $\mathcal{P}(c,\mathcal{G})^0$ and $\tilde{\mathcal{P}}(c,\mathcal{G})$.

	Assume that $(A_v)_{v \in \mathcal{V}}$ are $\mathcal{G}$-independent, and we will prove the first equality by induction on $k$. The base case $k = 1$ is immediate since both sides evaluate to $\varphi(a_1)$.  Suppose $k > 1$.  Starting from \eqref{eqn:formula moments}, we apply the bijection $\mathcal{P}(k) \to \bigsqcup_{\sigma \in \mathcal{I}(k)} \prod_{B\in \sigma} \mathcal{P}_{\operatorname{irr}}(B)$ to obtain
	\[
		\varphi(a_1 \dots a_k) = \sum_{\sigma \in \mathcal{I}(k)} \prod_{B \in \sigma} \sum_{\pi_B \in \mathcal{P}_{\operatorname{irr}}(B)} K_\pi^{\free}[a_j: j \in B].
	\]
	Let $\mathbbm{1}_k$ denote the partition $\{[k]\}$.  If $\sigma \neq \mathbbm{1}_k$, then each block $B$ has size strictly less than $k$, so we can apply the induction hypothesis on $B$.  Thus,
	\begin{align*}
		\varphi(a_1 \dots a_k) & = \sum_{\pi_0 \in \mathcal{P}_{\operatorname{irr}}(k)} K_{\pi_0}^{\free}[a_1,\dots,a_k] + \sum_{\sigma \in \mathcal{I}(k) \setminus \{\mathbbm{1}_k\}} \prod_{B \in \sigma} \sum_{\pi_B \in \mathcal{P}_{\operatorname{irr}}(B)} K_\pi^{\free}[a_j: j \in B] \\
		                       & = \sum_{\pi_0 \in \mathcal{P}_{\operatorname{irr}}(k)} + \sum_{\sigma \in \mathcal{I}(k) \setminus \{\mathbbm{1}_k\}} K_{|B|}^{\Bool}[a_j: j \in B].
	\end{align*}
	Applying the definition of Boolean cumulants,
	\[
		\varphi(a_1 \dots a_k) = \sum_{\pi_0 \in \mathcal{P}_{\operatorname{irr}}(k)} K_{\pi_0}^{\free}[a_1,\dots,a_k] + (\varphi(a_1 \dots a_k) - K_k^{\Bool}[a_1,\dots,a_k]).
	\]
	This rearranges to yield $K_k^{\Bool}[a_1,\dots,a_k] = \sum_{\pi_0 \in \mathcal{P}_{\operatorname{irr}}(k)} K_{\pi_0}^{\free}[a_1,\dots,a_k]$, which is the desired formula.  The other formulas, and the converses, follow by similar reasonings.
\end{proof}

\begin{remark} Computing the free or classical cumulants of $\mathcal{G}$ independent variables is rather difficult (let alone the monotone cumulants). For instance, one \emph{cannot} use the same argument to show that $K_k^{\free}[a_1,\dots,a_k]$ is the sum over connected partitions in $\mathcal{P}(c,\mathcal{G})$.  Indeed, in the preceding proof, we used the fact that a partition is in $\mathcal{P}(c,\mathcal{G})$ if and only if each of its irreducible components is in $\mathcal{P}(c,\mathcal{G})$.  The same reasoning could \emph{not} be applied for the connected components of $\pi$.  This is because $\mathcal{P}(c,\mathcal{G})$ has requirements about which blocks can be nested inside each other, which cannot be tested on each of the connected components in isolation.  More precisely, if $\pi \in \mathcal{P}(c,\mathcal{G})$ and if $B_1 \prec B_2$ are blocks in its non-crossing envelope, then $B_1$ will contain some block $B_1'$ of $\pi$ with $B_1' \prec B_2$, and then all the blocks of $\pi$ in $B_2$ are nested inside $B_1'$ and so Definition \ref{def: compatible partitions} (2) applies to them.  Thus, not all partitions in $\mathcal{P}(c|_{B_2},\mathcal{G})$ are allowable.
\end{remark}

\subsection{Free--(anti)monotone--Boolean mixtures in the tree framework} \label{subsec: tree independence}

As mentioned in Remark \ref{remark: BMF mixture}, there are now two different constructions of a mixture of Boolean, monotone, and free independence.  The first is given by digraph independence from \cite{JekelLiu2020,jekel2024general}, while the second is the one in this paper, since in a graph with $\mathcal{E}_2 = \varnothing$ there will be no tensor edges.  The difference between these two constructions was noted in \cite[Remark 3.24]{jekel2024general} in the case of Boolean--monotone mixtures; this remark showed that for a given set of pairwise Boolean and (anti)monotone relationships, the mixture in the digraph framework does not agree with the mixture in the BMT framework of \cite{arizmendi2025bmt} in general (though they agree when the edge set is a partial order), and as shown above our bigraph framework extends the BMT framework.

We will now consider Boolean-(anti)monotone-free pairwise relations in general and comment on the relationship between the digraph framework of \cite{jekel2024general} and our bigraph framework.  Consider a bigraph $\mathcal{G}$ with $\mathcal{E}_2 = \varnothing$.  Let $G$ be the digraph $(\mathcal{V},\mathcal{E}_1 \setminus \Delta)$.  Consider algebras $A_v$ for $v \in \mathcal{V}$ in a non-commutative probability space $(A,\varphi)$.  Let $c: [k] \to \mathcal{V}$ and $a_i \in A_{c(i)}$.
\begin{itemize}
	\item If the algebras $A_v$ are digraph independent with respect to $G$, then by \cite[Definition 3.18]{jekel2024general},
	      \[
		      \varphi(a_1 \dots a_k) = \sum_{\pi \in \mathcal{NC}(c,G)} K^{\Bool}_\pi(a_1,\ldots,a_k),
	      \]
	      where $\mathcal{NC}(c,G)$ is given by \cite[Definition 3.12]{jekel2024general} (explained further below).
	\item If the algebras $A_v$ are bigraph independent with respect to $\mathcal{G}$, then
	      \[
		      \varphi(a_1 \dots a_k) = \sum_{\pi \in \pcg^{0}} K^{\Bool}_\pi(a_1,\ldots,a_k).
	      \]
\end{itemize}
The difference between these two formulas lies completely in the two sets of partitions.  In the both cases, the partitions are non-crossing; indeed, in the bigraph setting crossings are excluded when $\mathcal{E}_2 = \varnothing$.  In both cases, the partitions are required to be consistently colored by $c$.  However, the two frameworks have different requirements about what happens when one block is nested inside another:
\begin{itemize}
	\item For $\pi$ to be in $\mathcal{NC}(c,G)$ means that whenever block $V_2$ is \emph{immediately} nested inside block $V_1$, then $(c(V_1),c(V_2)) \in \mathcal{E}_1 \setminus \Delta$.
	\item For $\pi$ to be in $\pcg^{0}$ means that whenever block $V_2$ is nested inside block $V_1$, then $(c(V_1),c(V_2)) \in \mathcal{E}_1$.  Moreover, if $V_2$ is immediately nested inside block $V_1$, then $c(V_1) \neq c(V_2)$.
\end{itemize}
It is thus clear that $\pcg^{0} \subseteq \mathcal{NC}(c,G)$, but the inclusion is strict in general since $\mathcal{NC}(c,G)$ only considers immediate children in the nesting forest of $\pi$, while $\pcg^{0}$ consider all descendants.

Both the digraph and bigraph independences (when $\mathcal{E}_2 = \varnothing$) fit into the tree independence framework of \cite{JekelLiu2020}.  Let $\mathcal{T}_{\mathcal{V},\free}$ be the rooted tree whose the vertices are the alternating words on the alphabet $\mathcal{V}$, with $w$ being adjacent to $jw$ whenever $jw$ is an alternating word, and with the root being the empty word.  For every rooted subtree $\mathcal{T}$ of $\mathcal{T}_{\mathcal{V},\free}$ one can define $\mathcal{T}$-independence which satisfies the formula \cite[Theorem 4.21]{JekelLiu2020}
\[
	\varphi(a_1 \dots a_k) = \sum_{\pi \in \mathcal{NC}(c,\mathcal{T})} K^{\Bool}_\pi(a_1,\ldots,a_k),
\]
where $\mathcal{NC}(c,\mathcal{T})$ is the set of non-crossing partitions that are consistently colored by $c$, satisfying the following condition:  Given a chain $V_1$, $V_2$, \dots, $V_k$ where $V_1$ is an outer block of $\pi$ and $V_{j+1}$ is immediately nested inside $V_j$, then $c(V_k) \dots c(V_1) \in \mathcal{T}$ \cite[Definition 4.18]{JekelLiu2020}.  Digraph independence corresponds to the case where $\mathcal{T} = \operatorname{Walk}(G)$ is the set of walks in the digraph $G$, written from right to left; see \cite[Definition 3.18]{JekelLiu2020} and \cite{jekel2024general}.  Meanwhile, bigraph independence (when $\mathcal{E}_2 = \varnothing$) corresponds to the tree
\begin{equation} \label{eq: tree for bigraph ind}
	\mathcal{T} = \{v_1 \dots v_k \in \mathcal{T}_{\mathcal{V},\free}: i < j \implies (v_j,v_i) \in \mathcal{E}_1\}.
\end{equation}

While the tree $\mathcal{T}$ in \eqref{eq: tree for bigraph ind} is different from $\operatorname{Walk}(G)$ in general, they will agree when $\mathcal{E}_1$ is transitive.  Indeed, if $v_k \dots v_1$ is in $\operatorname{Walk}(G)$, then $(v_1,v_2) \in \mathcal{E}_1$, \dots, $(v_{k-1},v_k) \in \mathcal{E}_1$, and hence by transitivity, $(v_i,v_j) \in \mathcal{E}_1$ for $i < j$.  Therefore, we have the following observation.

\begin{proposition}
	Let $\mathcal{G}$ be a bigraph with $\mathcal{E}_2 = \varnothing$ and $\mathcal{E}_1$ transitive.  Then bigraph independence with respect to $\mathcal{G}$ is equivalent to digraph independence with respect to $G = (\mathcal{V},\mathcal{E}_1)$.
\end{proposition}

We note that in the special case where $\mathcal{E}_2 = \varnothing$ and $\mathcal{E}_1$ is transitive and symmetric, one obtains a mixture of Boolean and free independence, which by the above proposition agrees with the BF independence of Kula and Wysoczanski \cite{kula2013example}.  In general, however, the mixtures of Boolean and free independence from the bigraph framework and the digraph framework are distinct.

For general Boolean, (anti)monotone, free mixtures, there are several other ways to form a mixture which interpolate between the bigraph and digraph versions.  For instance, for $m \in \bN$, let
\[
	\mathcal{T}^{(m)} = \{v_1 \dots v_k \in \mathcal{T}_{\mathcal{V},\free}: i < j \leq i + m \implies (v_j,v_i) \in \mathcal{E}_1\}.
\]
Then $\mathcal{T}^{(1)} = \operatorname{Walk}(G)$ and $\bigcap_{m \in \bN} \mathcal{T}^{(m)} = \mathcal{T}$.  Another way to interpolate is to set
\[
	\mathcal{S}^{(m)} = \{v_1 \dots v_k \in \mathcal{T}_{\mathcal{V},\free}: i < j \leq \max(i+1,k-m) \implies (v_j,v_i) \in \mathcal{E}_1\}.
\]
Thus, $\mathcal{S}^{(m)}$ checks the adjacency condition $(v_j,v_i) \in \mathcal{E}_1$ when either $j = i+1$ or when $j \leq k-m$.  Then $\mathcal{S}^{(0)} = \operatorname{Walk}(G)$ and $\bigcup_{m \in \bN} \mathcal{S}^{(m)} = \mathcal{T}$.

\subsection{Operad associativity} \label{sec:operad associativity}

In this section, we present a natural associativity property under ``composing'' independences with respect to different bigraphs.  Similar to the setting of tree independence and digraph independence studied in \cite{JekelLiu2020}, we formulate a symmetric operad of bigraphs.  We refer to \cite{Leinster2004} for general background on operads, and we recall a self-contained definition here.

\begin{definition}
	A \emph{(plain) operad} consists of a sequence $(P(n))_{n \in \bN}$ of sets, a distinguished element $\id \in P(1)$, and composition maps
	\[
		\circ_{m,n_1,\dots,n_m}: P(m) \times P(n_1) \times \dots \times P(n_m) \to P(n_1 + \dots + n_m)
	\]
	denoted
	\[
		(f, f_1,\dots, f_m) \mapsto f(f_1,\dots,f_m),
	\]
	such that the following axioms hold:
	\begin{itemize}
		\item \emph{Identity:} For $f \in P(m)$, we have $f(\id,\dots,\id) = f$ and $\id(f) = f$.
		\item \emph{Associativity:} Given $f \in P(m)$ and $f_j \in P(n_j)$ for $j = 1$, \dots, $m$ and $f_{j,i} \in P(m_{j,i})$ for $i = 1$, \dots, $I_j$ and $j = 1$, \dots, $n$, we have
		      \begin{multline*}
			      f(f_1(f_{1,1},\dots,f_{1,I_1}), \dots, f_m(f_{m,1}, \dots, f_{m,I_m})) \\
			      = [f(f_1,\dots,f_m)](f_{1,1},\dots, f_{1,I_1}, \dots \dots , f_{1,m}, \dots, f_{m,I_m}).
		      \end{multline*}
	\end{itemize}
	The elements of $P(m)$ are said to have \emph{arity $m$}.
\end{definition}

\begin{definition}
	A \emph{symmetric operad} consists of an operad $(P(n))$ together with a right action of the symmetric group $S_m$ on $P(m)$, denoted $(f,\sigma) \mapsto f_\sigma$, satisfying the following axioms:
	\begin{itemize}
		\item Let $f \in P(m)$ and $f_j \in P(n_j)$ for $j = 1,\dots, m$.  Let $\sigma \in S_m$, and let $\tilde{\sigma} \in S_{n_1+\dots+n_m}$ denote the element that rearranges the order of the blocks $\{1,\dots,n_1\}$, $\{n_1+1,\dots,n_1 + n_2\}$, $\{n_1+n_2+1,\dots,n_1 + n_2 + n_3\}$ according to $\sigma$.  Then
		      \[
			      f_\sigma(f_{\sigma(1)},\dots,f_{\sigma(m)}) = [f(f_1,\dots,f_m)]_{\tilde{\sigma}}.
		      \]
		\item Let $f$ and $f_j$ be as above.  Let $\sigma_j \in S_{n_j}$, and let $\sigma \in S_{n_1+\dots+n_m}$ be the element which permute the elements within each block $\{n_1+\dots+n_{j-1}+1,\dots, n_1 + \dots + n_j\}$ by the permutation $s_j$, without changing the order of the blocks.  Then
		      \[
			      f((f_1)_{\sigma_1},\dots,(f_m)_{\sigma_m}) = f(f_1,\dots,f_m)_\sigma.
		      \]
	\end{itemize}
\end{definition}

\begin{definition}[Operad of bigraphs] \label{def: bigraph operad}
	We define the operad $\operatorname{BG}$ as follows.  For each $m \in \bN$, let $\operatorname{BG}(m)$ be the set of bigraphs on vertex set $[m]$.

	Let $\id$ be the bigraph $(\{1\},\{(1,1)\},\varnothing)$,

	Define the composition operation as follows.  Let $\tilde{\mathcal{G}} = ([m],\tilde{\mathcal{E}}_1,\tilde{\mathcal{E}}_2) \in \operatorname{BG}(m)$ and for $j = 1$, \dots, $m$ let $\mathcal{G}^j = ([n_j],\mathcal{E}_1^j,\mathcal{E}_2^j) \in \operatorname{BG}(n_j)$.  Let $N_j = n_1 + \dots + n_j$ and $N = N_m$, and define $\iota_j: [n_j] \to [N]$ by $\iota_j(i) = N_{j-1} + i$, so that $[N] = \bigsqcup_{j=1}^m \iota_j([n_j])$.

	Let $\tilde{\mathcal{G}}(\mathcal{G}^1,\dots,\mathcal{G}^m)$ be the bigraph $\mathcal{G} = ([N],\mathcal{E}_1,\mathcal{E}_2)$ where $\mathcal{E}_1$ and $\mathcal{E}_2$ are defined as follows:
	\begin{itemize}
		\item If $v, w \in [n_j]$ and $i \in \{1,2\}$, then $(\iota_j(v),\iota_j(w)) \in \mathcal{E}_i$ if and only if $(v,w) \in \mathcal{E}_i^j$.
		\item If $v \in [n_j]$ and $w \in [n_{j'}]$ with $j \neq j'$, then $(\iota_j(v),\iota_{j'}(w)) \in \mathcal{E}_i$ if and only if $(j,j') \in \tilde{\mathcal{E}}_i$.
	\end{itemize}

	The permutation action is defined as follows:  Given $\mathcal{G} = ([m],\mathcal{E}_1,\mathcal{E}_2) \in \operatorname{BG}(m)$, define $\mathcal{G}_\sigma = ([m],(\sigma^{-1} \times \sigma^{-1})(\mathcal{E}_1),(\sigma^{-1} \times \sigma^{-1})(\mathcal{E}_2))$.
\end{definition}

\begin{observation}
	$\operatorname{BG}$ with the composition and permutations operations described above is a symmetric operad.
\end{observation}

We leave the routine verification of this fact as an exercise.  We now turn to the statement and proof of associativity for bigraph independence.

\begin{theorem} \label{thm: operad associativity}
	Let $\tilde{\mathcal{G}} \in \operatorname{BG}(m)$ and let $\mathcal{G}^j \in \operatorname{BG}(n_j)$ for $j = 1$, \dots, $m$.  Let $\mathcal{G}$ be the composition $\tilde{\mathcal{G}}(\mathcal{G}^1,\dots,\mathcal{G}^m)$ given by Definition \ref{def: bigraph operad}.  Let $N = n_1 + \dots + n_m$ and let $\iota_j: [n_j] \to [N]$ be the $j$th inclusion as in Definition \ref{def: bigraph operad}.

	Let $(A,\varphi)$ be a $\mathrm{C}^*$-probability space and let $(A_i)_{i \in [N]}$ be $*$-subalgebras.  For $j = 1$, \dots, $m$, let $B_j$ be the $*$-subalgebra generated by $A_{\iota_j(i)}$ for $i = 1$, \dots, $n_j$.  The following are equivalent:
	\begin{enumerate}
		\item $(A_i)_{i \in [N]}$ are $\mathcal{G}$-independent.
		\item $(B_j)_{j \in [m]}$ are $\tilde{\mathcal{G}}$-independent and for every $j \in [m]$, the subalgebras $(A_{\iota_j(i)})_{i \in [n_j]}$ are $\mathcal{G}^j$-independent.
	\end{enumerate}
\end{theorem}

To prove this, we need another technical lemma similar to those in \S\ref{subsec: boolean and classical cumulants}.

\begin{lemma} \label{lem: decomposition for associativity}
	Let $\tilde{\mathcal{G}} \in \operatorname{BG}(m)$ and $\mathcal{G}^i \in \operatorname{BG}(n_i)$ for $i = 1$, \dots, $m$.  Let $\mathcal{G} = \tilde{\mathcal{G}}(\mathcal{G}_1,\dots,\mathcal{G}_m)$.  Let $k \in \bN$ and $c: [k] \to \mathcal{V}$.  Let $\tilde{c}: [k] \to [m]$ be map such that $c(j)$ is in the image of $\iota_{\tilde{c}(j)}$.  For $\sigma \in \mathcal{P}(\tilde{c},\tilde{\mathcal{G}})$ and $B \in \sigma$, let $\tilde{c}(B) \in [m]$ be the common color of the elements in $B$ and let $c_B: B \to [n_{\tilde{c}(B)}]$ be the corresponding map such that $c|_B = \iota_{\tilde{c}(B)} \circ c_B$.  There is a bijection
	\[
		\Phi: \bigsqcup_{\sigma \in \mathcal{P}(\tilde{c},\tilde{\mathcal{G}})^0} \prod_{B \in \sigma} \mathcal{P}_{\operatorname{irr}}(c_B,\mathcal{G}^{\tilde{c}(B)})
		\to \mathcal{P}(c,\mathcal{G})
	\]
	given by $\Phi((\pi_B)_{B \in \sigma}) = \bigcup_{B \in \sigma} \pi_B$.  For $\pi \in \mathcal{P}(c,\mathcal{G})$, the inverse $\Phi^{-1}(\pi)$ is given by taking $\sigma$ to be the $\mathcal{P}(\tilde{c},\mathcal{G})^0$-envelope of the colorwise non-crossing envelope of $\pi$ with respect to $\tilde{c}$, and then setting $\pi_B = \pi|_B$ for $B \in \sigma$.
\end{lemma}

\begin{proof}
	Let $\Psi$ be the map that sends $\pi \in \mathcal{P}(c,\mathcal{G})$ to $(\pi|_B)_{B \in \sigma}$ where $\sigma$ is the $\mathcal{P}(\tilde{c},\tilde{\mathcal{G}})^0$-envelope of the colorwise non-crossing envelope of $\pi$ with respect to $\tilde{c}$.   To prove the lemma, we proceed by verifying that $\Phi$ and $\Psi$ map into the codomains that we have asserted, and that they are inverses of each other.

	First, let us show that $\Phi$ maps into $\mathcal{P}(c,\mathcal{G})$.  Fix $\sigma \in \mathcal{P}(\tilde{c},\tilde{\mathcal{G}})^0$ and for each $B \in \sigma$, let $\pi_B \in \mathcal{P}_{\operatorname{irr}}(c_B,\mathcal{G}^{\tilde{c}(B)})$.  Let $\pi = \bigcup_{B \in \sigma} \pi_B$.  We first verify that $\pi \in \mathcal{P}(c,\mathcal{G})$.
	\begin{enumerate}
		\item It is clear that the blocks of $\pi$ are monochromatic since this holds for each $\pi_B$.
		\item Suppose that $i_1 < j < i_2$ and $i_1 \sim_\pi i_2$.  Consider first the case that $j \sim_\pi i_1$, so that $j$ and $i_1$ and $i_2$ are in the same block $B$ of $\sigma$.  Then since $\pi_B \in \mathcal{P}(c_B, \mathcal{G}^{\tilde{c}(B)})$, we have that $(c_B(i_1),c_B(j)) \in \mathcal{E}_1^{\tilde{c}(B)}$, which implies that $(c(i_1),c(j)) \in \mathcal{E}_1$.  Now suppose that $j \not \sim_\pi i_1$.  Then because $\pi \in \mathcal{P}(\tilde{c},\tilde{\mathcal{G}})$, we have that $(\tilde{c}(i_1),\tilde{c}(j)) \in \tilde{\mathcal{E}}_1$ and hence $(c(i_1),c(j)) \in \mathcal{E}_1$.
		\item Suppose that $i_1 < j_1 < i_2 < j_2$ and $i_1 \sim_\sigma i_2$ and $j_1 \sim_\sigma j_2$ and $i_1 \not \sim_\sigma i_2$.  Again, we split into two cases based on whether $i_1 \sim_\pi j_1$ or not, and the argument proceeds in a similar way as for point (2).
	\end{enumerate}

	Next, we show that $\Psi$ maps into $\bigsqcup_{\sigma \in \mathcal{P}(c,\mathcal{G})^0} \prod_{B \in \sigma} \mathcal{P}_{\operatorname{irr}}(c_B,\mathcal{G}^{\tilde{c}(B)})$.  Fix $\pi \in \mathcal{P}(c,\mathcal{G})$, and let $\pi'$ be the colorwise non-crossing envelope of $\pi$ with respect to $\tilde{c}$.  It is straightforward to check that $\pi' \in \tilde{\mathcal{P}}(\tilde{c},\mathcal{G})$.  Indeed, when we replace $c$ by $\tilde{c}$, some blocks of different colors may become the same color, but nestings and crossings of the same color are always allowed in $\tilde{\mathcal{P}}(\tilde{c},\mathcal{G})$.  For two blocks with $\tilde{c}(B_1) \neq \tilde{c}(B_2)$, we have $(\tilde{c}(B_1),\tilde{c}(B_2)) \in \tilde{\mathcal{E}}_i$ if and only if $(c(B_1),c(B_2)) \in \mathcal{E}_i$, so the conditions for $\mathcal{P}(c,\mathcal{G})$ easily imply those of $\tilde{\mathcal{P}}(\tilde{c},\mathcal{G})$.  Let $\sigma$ be the $\mathcal{P}(\tilde{c},\tilde{\mathcal{G}})^0$-envelope $\pi'$ given by Lemma \ref{lem: P zero envelope}, so by construction $\sigma \in \mathcal{P}(\tilde{c},\tilde{\mathcal{G}})^0$.  Let $\pi_B = \pi|_B$ for $B \in \sigma$.  Now we verify that $\pi_B \in \mathcal{P}_{\operatorname{irr}}(c_B,\mathcal{G}^{\tilde{c}(B)})$ for each $B \in \sigma$:
	\begin{enumerate}
		\item The blocks of $\pi_B$ are monochromatic with respect to $c$, hence with respect to $c_B$.
		\item Suppose that $i_1 < j < i_2$ and $i_1 \sim_{\pi_B} i_2$.  Then of course $i_1 \sim_\pi i_2$, hence $(c(i_1),c(j)) \in \mathcal{E}_1$.  By definition of $\mathcal{G}$, we have $(c_B(i_1),c_B(j)) \in \mathcal{E}_1^{\tilde{c}(B)}$.
		\item Suppose that $i_1 < j_1 < i_2 < j_2$ and $i_1 \sim_{\pi_B} i_2$ and $j_1 \sim_{\pi_B} j_2$ and $i_1 \not \sim_{\pi_B} j_1$.  Since the same relations hold with respect to $\pi$, we have $(c(i_1),c(j_1)) \in \mathcal{E}_2$, and hence $(c_B(i_1),c_B(j_1)) \in \mathcal{E}_2^{\tilde{c}(B)}$.
	\end{enumerate}

	Next, we show that $\Psi \circ \Phi = \id$.  Fix $\sigma \in \mathcal{P}(\tilde{c},\tilde{\mathcal{G}})^0$ and for each $B \in \sigma$, let $\pi_B \in \mathcal{P}_{\operatorname{irr}}(c_B,\mathcal{G}^{\tilde{c}(B)})$.  Let $\pi = \bigcup_{B \in \sigma} \pi_B$, and we will show that $(\pi_B)_{B \in \sigma} = \Psi(\pi)$.  Since $\pi_B$ is simply $\pi|_B$, it suffices to show that $\sigma$ is the $\mathcal{P}(\tilde{c},\mathcal{G})^0$-envelope of the colorwise non-crossing envelope of $\pi$.  Let $\pi'$ denote the colorwise non-crossing envelope of $\pi$.  Since $\sigma$ is colorwise non-crossing and $\sigma \geq \pi$, we have $\sigma \geq \pi'$.  Moreover, for each block $B \in \sigma$, the irreduciblity of $\pi_B$ implies irreducibility of its non-crossing envelope $\operatorname{Env}^{\operatorname{nc}}(\pi_B)$.  Hence, $\min B$ and $\max B$ are in the same block of $\operatorname{Env}^{\operatorname{nc}}(\pi_B)$.  This implies that they are in the same block of the colorwise non-crossing envelope of $\pi$, namely $\pi'$.  Hence, by the uniqueness claim of Lemma \ref{lem: P zero envelope}, $\sigma$ must be the $\mathcal{P}(\tilde{c},\tilde{\mathcal{G}})^0$-envelope of $\pi'$, as desired.

	Finally, $\Phi \circ \Psi = \id$ is immediate by the same reasoning as in Lemma \ref{lem: P zero decomposition}.
\end{proof}

The other ingredient that we need is a product formula derived from the product formula for free cumulants.

\begin{lemma} \label{lem: product formula}
	Fix a bigraph $\mathcal{G}$, $k \in \bN$, and $c: [k] \to \mathcal{V}$.  Let $m_1$, \dots, $m_k \in \bN$.  Let $(A,\varphi)$ be a $\mathrm{C}^*$-probability space.  For $i = 1$, \dots, $k$ and $j = 1$, \dots, $m_i$, let $a_{i,j} \in A$, and let $b_i = a_{i,1} \dots a_{i,m_i}$.  Let $\overline{c}: [m_1 + \dots m_k] \to \mathcal{V}$ be the map that takes value $c(i)$ on $\{m_{i-1}+1, \dots, m_i\}$ (where by convention $m_0 = 0$).  Then
	\begin{equation} \label{eq: product formula}
		\sum_{\pi \in \mathcal{P}(\overline{c},\mathcal{G})} K_\pi^{\free}[a_{1,1},\dots,a_{1,m_1},\dots \dots, a_{k,1},\dots,a_{k,m_k}] = \sum_{\sigma \in \mathcal{P}(c,\mathcal{G})} K_\sigma^{\free}[b_1,\dots,b_k].
	\end{equation}
\end{lemma}

\begin{proof}
	We first prove the following claim:  Let $\rho$ be the partition with blocks $J_i = \{m_{i-1}+1,\dots,m_i\}$ for $i = 1$, \dots, $k$.  Let $\pi$ be a partition which is consistently colored by $\overline{c}$ and non-crossing on each color.  Let $\pi'$ be the partition of $[m_1+\dots+m_k]$ obtained by taking the colorwise non-crossing join of $\pi$ and $\rho$.  Since $\pi' \geq \rho$, there is a unique partition $\sigma$ of $[k]$ such that $i \sim_\sigma j$ if and only if $J_i$ and $J_j$ are in the same block of $\pi'$.  We claim that $\pi \in \mathcal{P}(\overline{c},\mathcal{G})$ if and only if $\sigma \in \mathcal{P}(c,\mathcal{G})$.  Suppose that $\pi \in \mathcal{P}(\overline{c},\mathcal{G})$.  We verify the conditions of Definition \ref{def: compatible partitions} for $\sigma$.
	\begin{enumerate}
		\item It is clear that the blocks of $\sigma$ are monochromatic.
		\item Suppose that $i_1 < j < i_2$ and $i_1 \sim_\sigma i_2$.  If $c(j) = c(i_1)$, there is nothing to prove, so assume that $c(j) \neq c(i_1)$.  Let $B$ be the block of $\sigma$ containing $i_1$ and $i_2$, and let $B'$ be the corresponding block of $\pi'$.  Let $j' \in J_j$.  There must exist some $i_1'$ and $i_2'$ in $B'$ with $i_1' < j < i_2'$ since otherwise the blocks of $\pi$ in $B'$ could be partitioned into those to the left of $j'$ and those to the right of $j'$, which would contradict their being in the same block in the join $\pi'$.  Hence, $(c(i_1),c(j)) = (\overline{c}(i_1'),\overline{c}(j)) \in \mathcal{E}_1$ since $\mathcal{P}(\overline{c},\mathcal{G})$.
		\item Suppose two blocks $B_1$ and $B_2$ of $\sigma$ cross.  Then the corresponding blocks $B_1'$ and $B_2'$ of $\pi'$ cross.  Since $\rho$ is an interval partition, the only way that this can happen is if some blocks of $\pi$, contained in $B_1'$ and $B_2'$ respectively, cross each other.  (This follows by similar reasoning as in Lemma \ref{lem: colorwise non-crossing envelope}.)  We conclude that $(c(B_1),c(B_2)) \in \mathcal{E}_2$.
	\end{enumerate}
	Conversely, suppose that $\sigma \in \mathcal{P}(c,\mathcal{G})$.
	\begin{enumerate}
		\item The blocks of $\pi$ are monochromatic by construction.
		\item Suppose $i_1' < j' < i_2'$ in $\pi$.  If $c(i_1) = c(j)$, there is nothing to prove.  Otherwise, suppose $i_1' \in J_{i_1}$ and $j' \in J_j$ and $i_2' \in J_{i_2}$, so that $i_1 < j < i_2$ and $i_1 \sim_\sigma i_2$.  We thus get $(\overline{c}(i_1'),\overline{c}(j')) = (c(i_1),c(j)) \in \mathcal{E}_1$.
		\item Suppose that two blocks $B_1$ and $B_2$ of $\pi$ cross.  We assumed that $\pi$ is non-crossing on each color, so $c(B_1) \neq c(B_2)$ and so $B_1$ and $B_2$ are in different blocks of $\pi'$.  These two blocks must cross, and hence the corresponding blocks in $\sigma$ cross, which implies that $(c(B_1),c(B_2)) \in \mathcal{E}_2$.
	\end{enumerate}
	Let $\Phi^J$ denote the map sending $\pi$ to $\sigma$ as described above.

	We group the terms on the left-hand side of \eqref{eq: product formula} by the value of $\Phi^J(\pi)$, resulting in
	\begin{equation} \label{eq: product rule intermediate}
		\sum_{\sigma \in \mathcal{P}(c,\mathcal{G})} \sum_{\substack{\pi \in \mathcal{P}(\overline{c},\mathcal{G}) \\ \Phi^J(\pi) = \sigma}} K_\pi^{\free}[a_{1,1},\dots,a_{1,m_1},\dots \dots, a_{k,1},\dots,a_{k,m_k}].
	\end{equation}
	Given $\sigma \in \mathcal{P}(c,\mathcal{G})$, let $\overline{\sigma}$ denote the partition of $[m_1 + \dots + m_k]$ obtained by replacing each element $j$ with the interval $J_j$.  For $B \in \sigma$, let $\overline{B}$ be the corresponding block in $\overline{\sigma}$.  Note that $\sigma = \Phi^J(\pi)$ occurs if and only if $\pi$ is obtained from $\overline{\sigma}$ by replacing each block $\overline{B}$ of $\overline{\sigma}$ with a partition $\pi_B \in \mathcal{NC}(\overline{B})$ such that $\pi_B \vee \rho|_{\overline{B}} = \mathbbm{1}_{\overline{B}}$.  The product formula for free cumulants implies that
	\begin{equation} \label{eq: free cumulant product formula}
		\sum_{\substack{\pi_B \in \mathcal{NC}(\overline{B}) \\ \pi_B \vee \rho|_{\overline{B}} = \mathbbm{1}_{\overline{B}}}} K_{\pi_B}^{\free}[a_{i,j}: i \in B, j \in [m_i]] = K_{|B|}^{\free}[b_i: i \in B],
	\end{equation}
	where the terms $a_{i,j}$ on the left-hand side are ordered lexicographically from left to right.  Therefore, \eqref{eq: product rule intermediate} becomes
	\[
		\sum_{\sigma \in \mathcal{P}(c,\mathcal{G})} \prod_{B \in \sigma} \sum_{\substack{\pi_B \in \mathcal{NC}(\overline{B}) \\ \pi_B \vee \rho|_{\overline{B}'} = \mathbbm{1}_{\overline{B}}}} K_{\pi_B}^{\free}[a_{1,1},\dots,a_{1,m_1},\dots \dots, a_{k,1},\dots,a_{k,m_k}].
	\]
	and hence by \eqref{eq: free cumulant product formula}
	\[
		\sum_{\sigma \in \mathcal{P}(c,\mathcal{G})} \prod_{B \in \sigma} K_{B}^{\free}[b_i: i \in B],
	\]
	which is the right-hand side of \eqref{eq: product formula}.
\end{proof}

\begin{proof}[Proof of Theorem \ref{thm: operad associativity}]
	We start with (2) $\implies$ (1) because the argument is simpler.  Assume that the $(B_j)_{j \in [k]}$ are $\tilde{\mathcal{G}}$-independent and for each $j \in [k]$, the subalgebras $(A_{\iota_j(i)})_{i \in [n_j]}$ are $\mathcal{G}^j$-independent.

	Fix $c: [k] \to [N]$, and let $\tilde{c}$ be as in Lemma \ref{lem: decomposition for associativity}.  Let $a_j \in A_{c(j)}$ for $j = 1$, \dots, $k$.  Note in particular that $a_j \in B_{\tilde{c}(j)}$.  Hence, using $\tilde{\mathcal{G}}$-independence of $(B_v)_{v \in [m]}$ and Proposition \ref{prop: formula via Boolean cumulants}
	\[
		\varphi(a_1 \dots a_k) = \sum_{\sigma \in \mathcal{P}(\tilde{c},\mathcal{G})^0} K_\sigma^{\Bool}[a_1,\dots,a_k] = \sum_{\sigma \in \mathcal{P}(\tilde{c},\mathcal{G})^0} \prod_{B \in \sigma} K_{|B|}^{\Bool}[a_j: j \in B].
	\]
	Next, for each block $B$ of $\sigma$, we apply the $\mathcal{G}^{\tilde{c}(B)}$-independence of $(A_{\iota_{\tilde{c}(B)}(v)})_{v \in [n_{\tilde{c}(B)}]}$ and use \eqref{eq: joint boolean cumulants via free cumulants} to express the Boolean cumulants
	\[
		\varphi(a_1 \dots a_k) = \sum_{\sigma \in \mathcal{P}(\tilde{c},\tilde{\mathcal{G}})^0} \prod_{B \in \sigma} \sum_{\pi_B \in \mathcal{P}_{\operatorname{irr}}(c_B,\mathcal{G}^{\tilde{c}(B)})} K_{\pi_B}^{\free}[a_j: j \in B].
	\]
	Using Lemma \ref{lem: decomposition for associativity}, this results in
	\[
		\varphi(a_1 \dots a_k) = \sum_{\pi \in \mathcal{P}(c,\mathcal{G})}  K_{\pi}^{\free}[a_1,\dots,a_k],
	\]
	which shows that $(A_v)_{v \in [N]}$ are $\mathcal{G}$-independent.

	(1) $\implies$ (2).  Assume (1).  The fact that $(A_{\iota_j(v)})_{v \in [n_j]}$ are $\mathcal{G}^j$ independent is straightforward to verify by applying the $\mathcal{G}$-independence condition to a product $a_1 \dots a_k$ where $a_i \in A_{\iota_j(c(i))}$ for some $c: [k] \to [n_j]$.

	We next have to prove that $B_1$, \dots, $B_m$ are $\tilde{\mathcal{G}}$-independent.  We first show a preliminary claim.  Suppose that $c: [k] \to [N]$, and let $\tilde{c}$ be as in Lemma \ref{lem: decomposition for associativity}.  Let $a_j \in A_{c(j)}$ for $j = 1$, \dots, $k$.  By reversing the manipulations in the previous argument, (1) implies that
	\[
		\varphi(a_1 \dots a_k) = \sum_{\sigma \in \mathcal{P}(\tilde{c},\mathcal{\tilde{G}})^0} K_\sigma^{\Bool}[a_1,\dots,a_k].
	\]
	The proof of Proposition \ref{prop: formula via Boolean cumulants} implies that
	\begin{equation} \label{eq: associativity intermediate}
		\varphi(a_1 \dots a_k) = \sum_{\pi \in \mathcal{P}(\tilde{c},\mathcal{\tilde{G}})} K_\pi^{\free}[a_1,\dots,a_k].
	\end{equation}

	Now to prove $\tilde{\mathcal{G}}$-independence of $B_1$, \dots, $B_m$, fix $k \in \bN$ and $c: [k] \to [m]$.  Let $b_i \in B_{c(i)}$ for $i = 1$, \dots, $k$, and we need to show that
	\begin{equation} \label{eq: associatity last goal}
		\varphi(b_1 \dots b_k) = \sum_{\sigma \in \mathcal{P}(c,\tilde{\mathcal{G}})} K_\sigma^{\free}[b_1,\dots,b_k].
	\end{equation}
	Since $B_j$ is generated by $(A_{\iota_j(i)})_{i \in [n_j]}$, we can use multilinearity to reduce to the case where
	\[
		b_i = a_{i,1} \dots a_{i,m_i}
	\]
	for some $m_j \in \bN$ and $a_{i,1}$, \dots, $a_{i,m_i}$ coming from one of the algebras $(A_{\iota_{c(i)}(v)})_{v \in [n_{c(i)}]}$.  Let $\overline{c}: [m_1 + \dots + m_k] \to [m]$ be as in Lemma \ref{lem: product formula}.  Since $a_{1,1} \dots a_{1,m_1} \dots \dots a_{k,1} \dots a_{k,m_k}$ is a product of elements from the individual algebras $(A_i)_{i \in [N]}$, we can apply the preliminary claim \eqref{eq: associativity intermediate}, which with the current objects and notation, translates to,
	\[
		\varphi(a_{1,1} \dots a_{1,m_1} \dots \dots a_{k,1} \dots a_{k,m_k}) = \sum_{\pi \in \mathcal{P}(\overline{c},\mathcal{G})} K_\pi^{\free}[a_{1,1},\dots,a_{1,m_1},\dots \dots, a_{k,1},\dots,a_{k,m_k}].
	\]
	Then using Lemma \ref{lem: product formula}, this becomes \eqref{eq: associatity last goal}.
\end{proof}

\section{Hilbert space model} \label{sec:hilbertspacemodel}

\subsection{Construction of the product Hilbert space}

By a \emph{pointed Hilbert space}, we mean a pair $(H,\xi)$ where $H$ is a Hilbert space and $\xi$ is a unit vector.  Our goal is to describe a product Hilbert space construction that allows us to realize bigraph independence.  As in the case of free-tensor graph independence ($\varepsilon$-independence) \cite{mlotkowski2004}, the Fock-like free product Hilbert space needs to be modified to arrange commutation for the elements that will be tensor independent.  Thus, we will similarly need to several of the same definitions pertaining to words over the vertex set of our graph (such as \emph{reduced} and \emph{equivalent} words with respect to $\mathcal{G}$), as well as the new definition of \emph{permissible} words.
We collect them in the following.
\begin{definition}
	Let $\mathcal{G} = (\mathcal{V},\mathcal{E}_1,\mathcal{E}_2)$ be a bigraph, and let $(H_v,\xi_v)$ be pointed Hilbert spaces.  Let $\mathcal{W}$ be the set of words $v_1 \dots v_k$ on the alphabet $\mathcal{V}$.
	\begin{itemize}
		\item A word $w = v_1 \dots v_k$ is \emph{reduced} if whenever $i_1 < i_2$ and $v_{i_1} = v_{i_2}$, there exists $j$ between $i_1$ and $i_2$ such that $v_j \neq v_{i_1}$ and $(v_{i_1},v_j) \not \in \mathcal{E}_2 \cap \mathcal{E}_1 \cap \overline{\mathcal{E}}_1$ ($v_{i_1}$ and $v_j$ are not in tensor relation).
		\item For $w = v_1 \dots v_k$ and $w' = v_1' \dots v_k'$, we say that $w \sim_0 w'$ if $w'$ is obtained by switching two consecutive letters $v_i$ and $v_{i+1}$ such that $(v_i,v_{i+1}) \in \mathcal{E}_2 \cap \mathcal{E}_1 \cap \overline{\mathcal{E}}_1$.
		\item We say $w$ and $w'$ are \emph{equivalent}, or $w \sim w'$ if there exists a sequence of words $w \sim_0 w_1 \sim_0 \dots \sim_0 w_{k-1} \sim w'$.  In other words, $\sim$ is the transitive closure of $\sim_0$.
		\item A word $v_1 \dots v_k$ is \emph{permissible} if whenever $i_1 < i_2$, we have $(v_{i_2},v_{i_1}) \in \mathcal{E}_1$.
	\end{itemize}
\end{definition}

Note that the notions of reduced words and equivalence are exactly the same as in \cite{mlotkowski2004,CaFi2017} for the graph $(\mathcal{V}, \mathcal{E}_2 \cap \mathcal{E}_1 \cap \overline{\mathcal{E}}_1)$. 

\begin{lemma} \label{lem: word facts}
	Fix a bigraph $\mathcal{G}$, and let $w$ and $w'$ be words.
	\begin{enumerate}
		\item If $w \sim w'$, then $w$ is reduced if and only if $w'$ is reduced.
		\item If $w \sim w'$, then $w$ is permissible if and only if $w'$ is permissible.
		\item Let $v \in \mathcal{G}$.  Then $vw \sim vw'$ if and only if $w \sim w'$.
	\end{enumerate}
\end{lemma}

\begin{proof}
	(1) is already known from \cite{mlotkowski2004,CaFi2017}.

	(2) Suppose that $w'$ is obtained from $w$ by swapping $v_j$ and $v_{j+1}$, where $(v_j,v_{j+1}) \in \mathcal{E}_2 \cap \mathcal{E}_1 \cap \overline{\mathcal{E}}_1$.  Assume $w$ is permissible, so for every $i_1 < i_2$, we have $(v_{i_2},v_{i_1}) \in \mathcal{E}_1$.  In $w'$, all pairs of indices are in the same order as before except for $j$ and $j+1$, so the condition of permissibility is preserved for all the other pairs.  But it is also preserved for the pair $j < j+1$ because both $(v_j,v_{j+1})$ and $(v_{j+1},v_j)$ are in $\mathcal{E}_1$.

	(3) It is clear that if $w \sim w'$, then $vw \sim vw'$.  Conversely, if $vw \sim vw'$, then consider the sequence of transformations $\tilde{w}_1$, \dots, $\tilde{w}_k$ from $vw$ to $vw'$.  Let $w_j$ be given by deleting the first occurrence of $v$ from $\tilde{w}_j$.  Then $w_1$, \dots, $w_k$ witnesses equivalence of $w$ and $w'$.
\end{proof}
Fix $\mathcal{G}$ and let $W_{\mathcal{G}}$ be the set of permissible reduced words (including the empty word).  We also fix a set $W_0$ of representatives for the equivalence classes of permissible words.  (For instance, we could choose the element that is minimal in the lexicographical order from each equivalence class.)
Let $r: W_{\mathcal{G}} \to W_0$ be the function sending each word to its equivalent representative.  Fix pointed Hilbert spaces $(H_v,\xi_v)$ for $v \in \mathcal{V}$, and let $H_v^\circ = \{h \in H_v: \ip{h,\xi_v} = 0\}$.  For $w = v_1 \dots v_k \in W_{\mathcal{G}}$ with $k \geq 1$, let
\[
	H_w^\circ := H_{v_1}^\circ \otimes \dots \otimes H_{v_k}^\circ,
\]
and for the empty word $\emptyset$, set
\[
	H_{\emptyset}^\circ := \bC.
\]
For $w \in W_{\mathcal{G}}$, let
\[
	U_w: H_w^\circ \to H_{r(w)}^\circ
\]
be the canonical isomorphism obtained by permuting the tensorands according to the rearrangement of $w$ into $r(w)$.  (Note that $U_w$ is independent of the sequence of swaps, since it is uniquely determined by sending the $k$th occurrence of each letter $v$ in $w$ to the $k$th occurrence of $v$ in $r(w)$.)

\begin{definition}[Bigraph product Hilbert space] \label{def: G-product Hilbert space}
	\[
		H := \bigoplus_{w \in W_0} H_w^\circ, \qquad \xi := 1 \in \bC = H_{\emptyset}^\circ \subset H.
	\]
	We call $(H,\xi)$ a $\mathcal{G}$-independent product of the pointed spaces $(H_v,\xi_v)_{v \in \mathcal{V}}$. Note that $(H,\xi)$ depends on the choice of representatives $W_0$.
\end{definition}
For each $v \in \mathcal{V}$, define a (not necessarily unital) $*$-homomorphism $\lambda_v: B(H_v) \to B(H)$.
Given $w \in W_{\mathcal{G}}$, write $vw$ for the word obtained by left-appending $v$. Set
\begin{align*}
	W_v^{(1)} & := \{ w \in W_0 : vw \in W_{\mathcal{G}} \}, \quad W_v^{(2)} := \{ r(vw) : w \in W_v^{(1)} \},
\end{align*}
and
\begin{align*}
	W_v^{(3)} & := W_0 \setminus \big(W_v^{(1)} \cup W_v^{(2)}\big).
\end{align*}
By Lemma \ref{lem: word facts} (3), the map $w \mapsto r(vw)$ is a bijection $W_v^{(1)} \to W_v^{(2)}$.

For $w \in W_v^{(1)}$, using $H_v = \bC \xi_v \oplus H_v^\circ$, define
\[
	T_{v,w} :
	H_w^\circ \oplus H_{r(vw)}^\circ
	\xrightarrow{\ \id \oplus U_{vw}\ }
	H_w^\circ \oplus H_{vw}^\circ
	\cong (\bC \otimes H_w^\circ) \oplus (H_v^\circ \otimes H_w^\circ)
	\cong H_v \otimes H_w^\circ,
\]
where the middle identifications are canonical and $1 \in \bC$ maps to $\xi_v$. Set
\[
	\mathcal{H}^{(v)}_- \coloneqq \bigoplus_{w \in W_v^{(1)}} H_w^\circ,\qquad
	\mathcal{H}^{(v)}_+ \coloneqq \bigoplus_{w \in W_v^{(2)}} H_{w}^\circ,\qquad
	\mathcal{H}_0^{(v)} \coloneqq \bigoplus_{w \in W_v^{(3)}} H_{w}^\circ.
\]
Define the isometric inclusion
\[
	S_v := \bigoplus_{w \in W_v^{(1)}} T_{v,w}^* :
	H_v \otimes \mathcal{H}^{(v)}_-
	\longrightarrow
	\mathcal{H}^{(v)}_- \oplus \mathcal{H}^{(v)}_+
	\subseteq H.
\]

\begin{definition}[Bigraph product inclusions] \label{def: G-free product inclusions}
	For $A \in B(H_v)$, define
	\[
		\lambda_v(A) := S_v \bigl(A \otimes \mathrm{id}_{\mathcal{H}^{(v)}_-}\bigr) S_v^* \in B(H).
	\]
	Equivalently, for each $w \in W_v^{(1)}$, $\lambda_v(A)$ acts on
	$H_w^\circ \oplus H_{r(vw)}^\circ$ via
	\[
		\lambda_v(A)\upharpoonright_{H_w^\circ \oplus H_{r(vw)}^\circ}
		= T_{v,w}^* \bigl(A \otimes I_{H_w^\circ}\bigr) T_{v,w},
	\]
	and $\lambda_v(A)$ vanishes on $\mathcal{H}^{(v)}_0$.
\end{definition}

Our main result for this section is that this construction realizes $\mathcal{G}$-independence; we postpone the proof to \S \ref{subsec: computation of vectors}.

\begin{theorem} \label{thm: Hilbert realization}
	Fix a bigraph $\mathcal{G}$ and pointed Hilbert spaces $(H_v,\xi_v)$ for $v \in \mathcal{V}$.
	Let $W_0$ be as in Definition \ref{def: G-product Hilbert space}, and let $(H,\xi)$ be the corresponding $\mathcal{G}$-product space. Then the subalgebras $\lambda_v(B(H_v)) \subset B(H)$ are $\mathcal{G}$-independent with respect to the vector state $\langle \xi \,|\,\cdot\,\xi\rangle$.
\end{theorem}

We first remark that Theorem \ref{thm: Hilbert realization} implies Proposition \ref{prop: independent C star copies}.  Indeed, suppose we are given a bigraph $\mathcal{G}$ and $\mathrm{C}^*$-probability spaces $(A_v,\varphi_v)$ for $v \in \mathcal{V}$.  The GNS construction produces a Hilbert space $H_v$, a unit vector $\xi_v$ and a $*$-homomorphism $\pi_v: A_v \to B(H_v)$ such that $\varphi_v(x) = \ip{\xi_v, \pi_v(x) \xi_v}$.  Let $(H,\xi)$ be the product Hilbert space in Theorem \ref{thm: Hilbert realization}.  We then obtain $*$-homomorphisms $\tilde{\lambda}_v = \lambda_v \circ \pi_v: A_v \to B(H)$ such that the images $\tilde{\lambda}_v(A_v)$ are bigraph independent with respect to $\mathcal{G}$ in $B(H)$ with the state given by $\xi$.

\begin{remark}
	The construction in Theorem \ref{thm: Hilbert realization} naturally includes previous Hilbert space models for independence.
	\begin{itemize}
		\item When $\mathcal{E}_1 = \bar{\mathcal{E}}_1$, $\mathcal{E}_1 = \emptyset$, our $\mathcal{G}$-independent product reduces to $\Lambda$ free product of \cite{mlotkowski2004}, section 3.
		\item When $\mathcal{E}_2 = \emptyset$ and $\mathcal{E}_1=\bar{\mathcal{E}_1}=\mathcal{V}\times\mathcal{V}$, our $\mathcal{G}$-independent product reduces to the free product of pointed Hilbert spaces \cite{voiculescu2007multiplication}.
		\item When $\mathcal{E}_1 = \bar{\mathcal{E}}_1 = \mathcal{E}_2 = \mathcal{V}\times\mathcal{V}$, our $\mathcal{G}$-independent product reduces to the tensor product of pointed Hilbert spaces.
		\item When $\mathcal{E}_1 \cap \bar{\mathcal{E}}_1 = \emptyset
		      $ and $\mathcal{E}_1 = \mathcal{V}$, our $\mathcal{G}$-independent product reduces to the monotone product of pointed Hilbert spaces \cite{muraki2000monotonic}.
		\item When $\mathcal{E}_1 = \mathcal{E}_1 = \emptyset$, our $\mathcal{G}$-independent product reduces to the Boolean product of pointed Hilbert spaces \cite{SpeicherWoroudi1997BooleanConvolution}.
	\end{itemize}
\end{remark}

\begin{remark}
	We now relate our construction in the case that $\mathcal{E}_1\cap \bar{\mathcal{E}}_1 \setminus \Delta = \mathcal{E}_2$ with the construction for BMT independence in \cite{arizmendi2025bmt}.  We will see that although the Hilbert spaces need not coincide in general, the Hilbert space from our construction will embed isometrically into the one from \cite{arizmendi2025bmt}.

	To describe this construction, we need a fixed ordering of the vertices of $\mathcal{G}$.  In fact, we will assume without loss of generality that $\mathcal{V} = \{1,\dots,p\}$.  Consider the tensor product Hilbert space
	$$
		\tilde{H} = {H}_1 \otimes \cdots \otimes {H}_p,
	$$
	and injections $\pi_i : B(H_v) \to B(\tilde{H})$ given by
	$$
		\pi_{i}(A) = P_{i,1} \otimes \cdots \otimes P_{i,i-1} \otimes A \otimes P_{i,i+1} \otimes \cdots \otimes P_{i,p}
	$$
	where $P_{i,j} = I_{{H}_j}$ if  $(i,j) \in \mathcal{E}_1$ and $P_{i,j} = P_{\xi_j}$ otherwise.  We note that by expressing $H_v \cong \bC \oplus H_v^\circ$, we can write
	\begin{equation} \label{eq: tilde H expansion}
		\tilde{H} \cong \bigoplus_{k \in \{0,\dots,p\}} \bigoplus_{i_1 < \dots < i_k} H_{i_1}^\circ \otimes \dots \otimes H_{i_k}^\circ.
	\end{equation}
	The terms in the sum can be regarded as indexed by the power-set $\mathcal{P}([p])$.  We can then compare the terms appearing in this direct sum with those that appear in bigraph product Hilbert space $H$ from Theorem \ref{thm: Hilbert realization}.

	To this end, we note several facts about the permissible, reduced words associated to $\mathcal{G}$ in the case that $\mathcal{E}_2 = \mathcal{E}_1 \cap \mathcal{E}_1 \setminus \Delta$.
	\begin{itemize}
		\item \emph{In a permissible reduced word $w = v_1 \dots v_m$, each letter appears at most once.}  Indeed, suppose that there are two occurrences of the same letter, say $v_i = v_j$ with $i < j$.  Assume without loss of generality that $v_k \neq v_i$ for $k \in \{i+1,\dots,j-1\}$.  Let $k \in \{i+1,\dots,j-1\}$.  Since the word is permissible, we have $(v_k,v_i) \in \mathcal{E}_1$ and $(v_i,v_k) = (v_j,v_k) \in \mathcal{E}_1$.  Hence, $(v_i,v_k) \in \mathcal{E}_1 \cap \overline{\mathcal{E}}_1$, and so $(v_i,v_k) \in \mathcal{E}_2$.  Therefore, all the letters between $v_i$ and $v_j$ are adjacent by an edge in $\mathcal{E}_2$, which contradicts $w$ being reduced.
		\item \emph{Two permissible reduced words $w = v_1 \dots v_m$ and $w' = v_1' \dots v_m'$ are equivalent if and only if $\{v_1,\dots,v_m\} = \{v_1',\dots,v_m'\}$.}  The ``only if'' implication is immediate.  For the converse, we proceed by induction on $m$.  Suppose that $\{v_1,\dots,v_m\} = \{v_1',\dots,v_m'\}$.  Then there is a permutation $\sigma$ such that $v_{\sigma(j)}' = v_j$ for all $j$.  If $\sigma(j) < \sigma(1)$, then we have $(v_j,v_1) \in \mathcal{E}_1$ by permissibility of $w$ and at the same time $(v_1,v_j) = (v_{\sigma(1)}',v_{\sigma(j)}') \in \mathcal{E}_1$ by permissibility of $w'$; therefore, $(v_1,v_j) \in \mathcal{E}_2$.  Hence, $v_1' \dots v_{\sigma(1)-1}'$ are $\mathcal{E}_2$-adjacent to $v_{\sigma(1)}' = v_1$, and thus can be swapped with it to produce an equivalent word.  Therefore,
		      \[
			      w'' := v_{\sigma(1)}' v_1' \dots v_{\sigma(1)-1}' v_{\sigma(1)+1}' \dots v_m' \sim w'.
		      \]
		      By induction hypothesis, the words $v_1' \dots v_{\sigma(1)-1}' v_{\sigma(1)+1}' \dots v_m'$ and $v_2 \dots v_m$ are equivalent since they have the same set of letters.  Therefore, $w''$ is equivalent to $w$, and hence $w'$ is equivalent to $w$.
	\end{itemize}
	In light of these two observations, there is an injection from the set of permissible reduced words to the power-set of $\mathcal{V}$ given by mapping $v_1 \dots v_m$ to $\{v_1,\dots,v_m\}$.  The terms in the product Hilbert space $H$ thus match up with a subset of the terms in the expansion \eqref{eq: tilde H expansion} for $\tilde{H}$.  We thus naturally obtain an isometry $U: H \to \tilde{H}$ given by matching the corresponding terms in the direct sum and permuting the tensorands in each term $H_{v_1}^\circ \otimes \cdots \otimes H_{v_m}^\circ$ to match the order they appear in $\tilde{H}$.

    We claim that for $A \in B(H_v)$,
    \[
    U \lambda_i(A) = \pi_i(A) U.
    \]
    It suffices to check the equality on vectors that span the Hilbert space.  For this purpose, consider a word $w = v_1 \dots v_m\in \mathcal{W}$ such that $iw$ is permissible, and we will show equality on vectors in $H_w^\circ \oplus H_{r(iw)}^\circ$.  Let $j_1 \dots j_{\ell} i j_{\ell+1} \dots j_m$ be the increasing rearrangement of $v_1 \dots v_m$.  Consider the commutative diagram
    \[
    \begin{tikzcd}
    H_i \otimes H_w^\circ \arrow{r}{T_{i,w}} \arrow{d}{\id \otimes U} & H_w^\circ \oplus H_{r(iw)}^\circ \arrow{d}{U}\\
    H_i \otimes H_{j_1}^\circ \otimes \dots \otimes H_{j_m}^\circ \arrow{r}{\cong} & H_{j_1}^\circ \otimes \dots \otimes H_{j_{\ell}}^\circ \otimes H_i \otimes H_{j_{\ell+1}}^\circ \otimes \dots \otimes H_{j_m}^\circ \arrow{d}{\subseteq} \\
    & H_1 \otimes \dots \otimes H_p,
    \end{tikzcd}
    \]
    where the bottom horizontal map is the permutation of the tensorands, and the bottom vertical map is inclusion.  The maps in the commutative square are isomorphisms given by distributing tensor products over direct sums and permuting the tensorands in the natural way, so commutativity is clear.  The action of $\lambda_i(A)$ on the top right space $H_w^\circ \oplus H_{r(iw)}^\circ$ corresponds to $A \otimes \id$ on the top left space by Definition \ref{def: G-free product inclusions}.  The action of $\pi_i(A)$ on the $H_1 \otimes \dots \otimes H_p$ is given by $P_{i,1} \otimes \dots \otimes P_{i,i-1} \otimes A \otimes P_{i,i+1} \otimes P_{i,p}$.  Note that since $iw$ is permissible, all the letters $v_j$ satisfies $(v_j,i) \in \mathcal{E}_1$ and so $P_{i,v_j} = 1$.  We thus see that the restriction of $\pi_i(A)$ to the subspace $H_{j_1}^\circ \otimes \dots \otimes H_{j_{\ell}}^\circ \otimes H_i \otimes H_{j_{\ell+1}}^\circ \otimes \dots \otimes H_{j_m}^\circ$ in the bottom right part of the square is simply $\id \otimes \dots \otimes \id \otimes A \otimes \id \otimes \dots \otimes \id$.  Hence, the action of $\pi_i(A)$ translates into $A \otimes \id$ on the space in the bottom left of the square.  Comparing this with $A \otimes \id$ on the top left part of the square, we see that the action of $\lambda_i(A)$ on the top right corresponds to the action of $\pi_i(A)$ on the bottom right, as desired.
    
    Having shown the equality on the span of $H_w^\circ \oplus H_{r(iw)}^\circ$ where $iw$ is permissible, we now consider the orthogonal complement of the Hilbert space, which is the span of $H_w^\circ$ where $w \in \mathcal{W}$ such that $w$ does not equal $r(iw')$ and $iw$ is not permissible.  In this case, the action of $\lambda_i(A)$ on $H_w^\circ$ is zero.  On the other hand, letting $j_1 \dots j_m$ be the increasing rearrangement of $w$, $U$ maps $H_w^\circ$ into $H_{j_1}^\circ \otimes \dots \otimes H_{j_m}^\circ$.  Since $iw$ is not permissible, there exists some $s$ such that $(j_s,i) \not \in \mathcal{E}_1$, and so $P_{i,j_s} = P_{\xi_{j_s}}$ which vanishes on $H_{j_s}^\circ$.  This implies that $\pi_i(A)$ acts by zero on $H_{j_1}^\circ \otimes \dots \otimes H_{j_m}^\circ$.  Hence, $\pi_i(A) U$ and $U \lambda_i(A)$ are both zero on this subspace.

    Thus, although the Hilbert space constructed here is smaller than $H_1 \otimes \dots \otimes H_p$, we have an isometric inclusion $U$ such that $U \lambda_i(A) = \pi_i(A)U$ and $U$ maps the state vector to the state vector.  In particular, the cyclic subspaces generated by the state vector in the two constructions under the action of some family of operators on the $B(H_j)$'s will be isomorphic, and the joint moments of such operators will agree.
\end{remark}

\subsection{Inductive computation of the vectors} \label{subsec: computation of vectors}

In order prove Theorem \ref{thm: Hilbert realization}, we want to show that
\begin{equation} \label{eq: moment formula reverse}
	\ip{\xi, \lambda_{c(k)}(a_k) \dots \lambda_{c(1)}(a_1) \xi} = \sum_{\pi \in \mathcal{P}_k(c,\mathcal{G})^0} \prod_{B \in \pi} K^{\Bool}_{|B|}[a_j: j \in B].
\end{equation}
Here we have written the indices in reverse order for notational convenience in our inductive arguments.  Hence, in the Boolean cumulants the indices $a_j: j \in B$ should be understood to be in \emph{decreasing} order for each block as well.  We will in fact give a combinatorial formula in general for the element
\[
	\lambda_{c(k)}(a_k) \dots \lambda_{c(1)}(a_1) \xi,
\]
since this is more suitable to compute by induction.

We introduce first some more preparatory notation.  Let $P_v \in B(\mathcal{H}_v)$ be the rank-one projection onto $\bC \xi_v$, and let $Q_v = 1 - P_v$.  Let
\begin{align*}
	a_j^{(0,0)} & = P_{c(j)} a_j P_{c(j)},\quad
	a_j^{(0,1)}  = P_{c(j)} a_j Q_{c(j)}         \\
	a_j^{(1,1)} & = Q_{c(j)} a_j Q_{c(j)}, \quad
	a_j^{(1,0)}  = Q_{c(j)} a_j P_{c(j)}.
\end{align*}
Then $a_j = a_j^{(0,0)} + a_j^{(0,1)} + a_j^{(1,0)} + a_j^{(1,1)}$.  We may thus write
\begin{equation} \label{eq: 4 term expansion}
	\ip{\xi, \lambda_{c(k)}(a_k) \dots \lambda_{c(1)}(a_1) \xi} = \sum_{\delta_1,\epsilon_1,\dots,\delta_k,\epsilon_k \in \{0,1\}} \ip{\xi, \lambda_{c(k)}(a_k^{(\delta_k,\epsilon_k)}) \dots \lambda_{c(1)}(a_1^{(\delta_1,\epsilon_1)}) \xi}.
\end{equation}
Our goal is to show that certain of the terms in the sum vanish, while the others correspond to partitions in $\pcg$ and evaluate to the product of Boolean cumulants in the asserted formula.  Note that $a_j^{(\delta,\epsilon)}$ annihilates $\mathcal{H}_{c(j)}^\circ$ when $\epsilon = 0$ and annihilates $\bC \xi_{c(j)}$ when $\epsilon = 1$, and its image is contained in $\bC \xi_{c(j)}$ when $\delta = 0$ and $\mathcal{H}_{c(j)}^\circ$ when $\delta = 1$.

Examining the definition of the maps $\lambda_v$ in Definition \ref{def: G-free product inclusions}, we conclude the following.  Let $w = v_1 \dots v_k$.

\begin{fact} \label{obs: where they map} ~
	\begin{itemize}
		\item $\lambda_{c(j)}(a^{(0,0)})$ maps $H_w^\circ$ into itself if $w \in W_{c(j)}^{(1)}$ and vanishes on $H_w^\circ$ otherwise.
		\item $\lambda_{c(j)}(a^{(1,0)})$ maps $H_w^\circ$ into $H_{r(c(j)w)}^\circ$ if $w \in W_{c(j)}^{(1)}$, and vanishes on $H_w^\circ$ otherwise.
		\item $\lambda_{c(j)}(a^{(0,1)})$ maps $H_w^\circ$ into $H_{w'}^\circ$ if $w = r(c(j)w') \in W_{c(j)}^{(2)}$ (where $w' \in W_{c(j)}^{(1)}$), and vanishes on $H_w^\circ$ otherwise.
		\item $\lambda_{c(j)}(a^{(1,1)})$ maps $H_w^\circ$ into itself if $w \in W_{c(j)}^{(2)}$, and vanishes on $H_w^\circ$ otherwise.
	\end{itemize}
\end{fact}

With this information in mind, we can then consider the effect of applying several operators $\lambda_{c(j)}(a^{(\delta_j,\epsilon_j)})$ consecutively to the state vector $\xi$, and thus determine which direct summand of the Hilbert space $\mathcal{H}$ contains the vector
\[
	\lambda_{c(j)}(a_j^{(\delta_1,\epsilon_1)}) \dots \lambda_{c(1)}(a_1^{(\delta_1,\epsilon_1)}) \xi
\]
for each $j \leq k$.   First, to keep track of the number of tensorands, we introduce \emph{for each vertex} $v \in \mathcal{G}$ a \emph{height function} $h_v$ associated to the sequence of indices $(\delta_j,\epsilon_j)_{1 \leq j \leq k}$ and a coloring $c = c_1\cdots c_k$.  Let, for any $ m \in \llbracket 0,k \rrbracket $:
\[
	h_v(m) = \sum_{j=1}^m \mathbbm{1}_{c(j)=v} (\delta_i - \epsilon_i).
\]
Note that $h_v(0) = 0$, and $h_v(j+1) - h_v(j) \in \{-1,0,1\}$ and if $|h_v(j+1)-h_v(j)|=1$, then $c(j)=v$.
By inductive application of the observations above, one can show that
$$
	\lambda_{c(j)}(a_j) \dots \lambda_{c(1)}(a_1) \xi \in H_{w(j)}^{\circ}
$$
for some word $w(j) \in W_0$ with $h_v(j)$ \emph{occurrences of the letter} $v$ (allowing the possibility that the vector is zero):
$$
	h_v(j) = \{ i \in \llbracket 1,|w(j)| \rrbracket  \colon w_i(j)=v \}.
$$

If $h_v(j)$ is ever $-1$, then the first time that $h_v(j) = -1$, we are applying an `annihilation operator' $\lambda_{c(j)}(a_j^{(0,1)})$ (with $c(j) = v$) to a vector which does not have any tensorands from $H_v^\circ$, which results in :
\[
	\lambda_{c(j)}(a_j^{(\delta_j,\epsilon_j)}) \dots \lambda_{c(1)}(a_1^{(\delta_1,\epsilon_1)}) \xi= 0,~ \textrm{if } h_v(j)=-1.
\]
Hence also, if $h(i) < 0$ for any $i \leq j$, then $\lambda_{c(j)}(a_j^{(\delta_j,\epsilon_j)}) \dots \lambda_{c(1)}(a_1^{(\delta_1,\epsilon_1)}) \xi = 0$.  Furthermore, at the last step, for the inner product to be nonzero, $\lambda_{c(k)}(a_k^{(\delta_k,\epsilon_k)}) \dots \lambda_{c(1)}(a_1^{(\delta_1,\epsilon_1)}) \xi$ must be in $\bC \xi$, and hence
\[
	h_v(k) = 0.
\]

Therefore, in the expansion \ref{eq: 4 term expansion}, only the summands which have nonnegative height functions $h_v$, for each $v \in \mathcal{G}$, with $h_v(k) = 0$ will remain.  Also, whenever $h_v$ hits $0$ at some point $j$ and $h(j+1)=0$ too, one must have $\delta_{j+1} = \varepsilon_{j+1} =0$ for the vector not to be $0$.

We want to express these in terms of non-crossing partitions of the sets $\{c_i = v\}$, for each $v \in \mathcal{G}$ and see these partitions are compatible in the sense that they yield a single partition of $[k]$ in $\mathcal{P}(c,\mathcal{G})^0$.

Thus, we recall the following fact, which is a generalization of the well-known bijection between non-crossing \emph{pair} partitions and Dyck paths.  A similar statement is given in \cite[Lemma 4.24]{JekelLiu2020}.  In our setting, we visualize the indices $1$, \dots, $k$ as running from \emph{right} to \emph{left}.  The idea is that the $(0,0)$ indices correspond to singletons, the $(1,0)$ indices correspond to the right endpoints of blocks, the $(0,1)$ indices correspond to the left endpoints, and the $(1,1)$ indices correspond to middle elements of a block.  However, since in the larger argument we consider building partition inductively by adding on new indices to the left side, then at the intermediate stages we have to allow that some blocks are ``unfinished'' or are missing their left endpoints.

\begin{definition}
	An \emph{unfinished partition} $\pi$ is a partition $\pi$ of $[k]$ together with a labeling of each of block as ``finished'' or ``unfinished.''  Similarly, an \emph{unfinished non-crossing partition} is an unfinished partition such that $\pi$ is non-crossing.
\end{definition}

An unfinished partition $\pi \in [k]$ has a restriction to $[k-1]$, which is also an unfinished partition. Let $\pi$ be an unfinished partition of $[k]$ and let $B$ be the block containing $k$. We define the restriction $\tilde{\pi}$ of $\pi$ to $[k-1]$, by declaring that all blocks of $\pi$ different from $B$ are blocks of $\tilde{\pi}$ and if $B \neq \{k\}$, $B\backslash \{k\}$ is a unfinished block of $\tilde{\pi}$. This construction will be used in the proof of our key result, Lemma \ref{lem: unfinished moments}.

\begin{lemma} \label{lem: partition path bijection}
	There is a bijection between the following sets:
	\begin{enumerate}
		\item The set of unfinished non-crossing partitions of $[k]$.
		\item The set of sequences $(\delta_1,\varepsilon_1)$, \dots, $(\delta_k, \varepsilon_k)$ such that the height function
		      \[
			      h(j) = \sum_{i=1}^j (\delta_i - \epsilon_i),~ j \in \llbracket 0, k \rrbracket ,
		      \]
		      is always nonnegative and $h(j)=h(j+1)=0$ implies $\varepsilon_{j+1} = \delta_{j+1} = 0$
	\end{enumerate}
	The bijection is described as follows:
	\begin{enumerate}
		\item $j$ is a finished singleton block of $\pi$ if and only if $(\delta_j,\varepsilon_j) = (0,0)$.
		\item $j$ is the right endpoint of a block that is not a finished singleton block if and only if $(\delta_j,\varepsilon_j) = (1,0)$.
		\item $j$ is the left endpoint of a finished non-singleton block if and only if $(\delta_j,\varepsilon_j) = (0,1)$.
		\item $j$ is not the right endpoint of a singleton block and not the left endpoints of a finished block if and only if $(\delta_j,\varepsilon_j) = (1,1)$.
	\end{enumerate}
\end{lemma}

\begin{proof}
	Let us proceed by induction on $k$. The case $k=1$ is trivial, if $h(1)=0$, then $(\delta_1,\varepsilon_1)=(0,0)$ and the only unfinished non-crossing partition of $[1]$ is $\{\{1\}\}$ which is a finished singleton block. If $h(1)=1$, then $(\delta_1,\varepsilon_1)=(1,0)$ and the only unfinished non-crossing partition of $[1]$ is $\{\{1\}\}$ which is an unfinished block.
	Suppose the result holds for $k-1$ and let us prove it for $k$. Let $(\delta_1,\varepsilon_1), \ldots, (\delta_k,\varepsilon_k)$ be a sequence such that the height function $h$ is nonnegative and $h(j)=h(j+1)=0$ implies $\varepsilon_{j+1} = \delta_{j+1} = 0$. Let $\tilde{h}$ be the restriction of $h$ to $\llbracket 0, k-1 \rrbracket$ and let $\tilde{\pi}$ be the unfinished non-crossing partition of $[k-1]$ associated to $(\delta_1,\varepsilon_1), \ldots, (\delta_{k-1},\varepsilon_{k-1})$ by the induction hypothesis. We define the unfinished non-crossing partition $\pi$ of $[k]$ associated to $(\delta_1,\varepsilon_1), \ldots, (\delta_k,\varepsilon_k)$ as follows:
	\begin{itemize}
		\item If $(\delta_k,\varepsilon_k) = (0,0)$, then $\pi$ is obtained from $\tilde{\pi}$ by adding the finished singleton block $\{k\}$.
		\item If $(\delta_k,\varepsilon_k) = (1,0)$, then $\pi$ is obtained from $\tilde{\pi}$ by adding the unfinished block $\{k\}$.
		\item If $(\delta_k,\varepsilon_k) = (1,1)$, then $h(k-1) > 0$ and $\pi$ is obtained from $\tilde{\pi}$ by adding $k$ to the unfinished block of $\tilde{\pi}$ containing the largest element.
		\item If $(\delta_k,\varepsilon_k) = (0,1)$, then $h(k-1) > 0$ and $\pi$ is obtained from $\tilde{\pi}$ by adding $k$ to the unfinished block of $\tilde{\pi}$ containing the largest integer, making this block finished.
	\end{itemize}
	The converse map is obtained by reversing the above construction. This concludes the proof.
\end{proof}

\begin{example}
	Consider the unfinished non-crossing partition $\pi = \{\{6,5,4\},\{3\},\{2\},\{1\}\}$ of $[6]$, where the block $\{6,5,4\}$ is unfinished.  The associated sequence $(\delta_1,\varepsilon_1), \ldots, (\delta_6,\varepsilon_6)$ is given by
	\[
		(0,0), (0,0), (0,1), (1,1), (1,1), (1,0),
	\]
	since $h(0) = 0$, $h(1) = 0$, $h(2) = 0$, $h(3) = 1$, $h(4) = 2$, $h(5) = 1$, and $h(6) = 0$.
\end{example}
\begin{observation} \label{obs: multicolor partition from path}
	Consider a sequence $(\delta_1,\varepsilon_1)$, \dots, $(\delta_k, \varepsilon_k) \in \{0,1\}^2$ together with a coloring $c: [k] \to \mathcal{V}$.  Let $J_v = c^{-1}(v) = \{j: c(j) = v\}$.  Suppose that $h_v \geq 0$.  By applying Lemma \ref{lem: partition path bijection} to the restriction of $\pi$ to $J_v$, we obtain an unfinished non-crossing partition $\pi_v$ of the index set $J_v$.  We hence obtain an unfinished partition $\pi = \bigcup_{v \in \mathcal{V}} \pi_v$, whose restriction to each $J_v$ is non-crossing.
\end{observation}

We now describe the analogue of $\mathcal{P}(c,\mathcal{G})^0$ for unfinished partitions.  The definition is motivated by considering the partition as part of a larger, completed partition; the unfinished blocks will have more elements added to the left, and these future elements must be accounted for in the conditions of Definition \ref{def: compatible partitions}.

If $V$ is an unfinished block, we should slightly modify the definition of its convex hull, ${\rm Conv}(V)$, and set it equal to ${\rm Conv}(\{k\} \cup V)$. This has the effect of adding all points to the left of the block $V$ to the convex hull of all the points in $V$.

Define $\mathcal{E}_1^{\pi_u} = \{ (U,V) : V \cap {\rm Conv}(U) \neq \emptyset \}$ and $\mathcal{E}^{\pi_u}_2 = \mathcal{E}_1 \cap \bar{\mathcal{E}}_1$.
A block $U$ of $\pi_u$ is nested in another block $V$ of $\pi_u$ if $(U,V) \in \mathcal{E}_1^{\pi_u}\backslash\mathcal{E}_2^{\pi_u}$. For example, if $U$ is unfinished and the minimum of $V$ is greater than the maximum of $U$ (again, for the order $k > \cdots > 1$), $V$ is nested in $U$.
An unfinished block is pictured, along with its convex hull, with a segment extending to the left, see Figure \ref{fig:unfinished partition}.

\begin{figure}
	\includegraphics[width=0.4\textwidth]{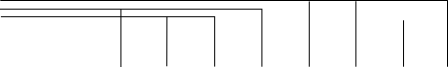}
	\caption{\label{fig:unfinished partition} An unfinished partition. Unfinished blocks are pictured with a segment extending to the left.}
\end{figure}

\begin{definition}
	\label{def:unfinishedpartition}
	Let $\pi$ be an unfinished partition of $[k]$.

	\begin{itemize}
		\item We say that $\pi \in \mathcal{P}(c,\mathcal{G})_u$ if the coloring $c$ defines a morphism from the bigraph $(\pi,\mathcal{E}^{\pi_u}_1,\mathcal{E}_2^{\pi_u})$ to $(\mathcal{G},\mathcal{E}_1,\mathcal{E}_2)$ :
		      \begin{itemize}
			      \item[(2')] If $j > i$ and $i$ is in an unfinished block, then $(c(i),c(j)) \in \mathcal{E}_1$.
			      \item[(3')] If $j_1 > i > j_2$ and $j_1 \sim_\pi j_2$ and $i$ is in an unfinished block, then $(c(j_1),c(i)) \in \mathcal{E}_2$.
		      \end{itemize}

		\item Two blocks $U$ and $V$ of an unfinished partition $\pi_u \in \pcg_u$ are said to be \emph{adjacent} if $U$ is nested in $V$, $c(U) = c(V)$ and for any other block $W$ with $U$ nested in $W$ and $V$ nested in $U$, $(c(U),c(W)) \in \mathcal{E}_2$.

		\item We say that $\pi \in \mathcal{P}(c,\mathcal{G})^0_u$ if $\pi \in \pcg_u$ and $\pi$ has no adjacent blocks.

	\end{itemize}
\end{definition}

We are now ready to prove our lemma, computing the action of the operators $\lambda_c,~c\in\mathcal{G}$ on the vacuum vector $\xi$.  This computation is parallel to the one given in \cite[Lemma 3.17]{jekel2024general} for the setting of digraph independences.

\begin{lemma} \label{lem: unfinished moments}
	Fix a bigraph $\mathcal{G}$, a sequence $(\delta_1,\varepsilon_1),\ldots,(\delta_k,\varepsilon_k)\in\{0,1\}^2$, and a coloring $c:[k]\to\mathcal{V}$, and assume $h_v(m)\ge 0$ for all $v\in\mathcal{V}$ and $m\in\{0,\ldots,k\}$.

	Let $\pi$ be the unfinished partition associated to $(\delta,\varepsilon,c)$ (Observation \ref{obs: multicolor partition from path}). List the unfinished blocks $B_1,\ldots,B_m$ so that $\min B_1<\cdots<\min B_m$ (right-to-left), set $v_j=c(B_j)$, and define $w=v_m\cdots v_1$.

	If $\pi\in\mathcal{P}(c,\mathcal{G})_u^0$, then $w\in W_{\mathcal{G}}$ and
	\begin{multline*}
		\lambda_{c(k)}\!\big(a_k^{(\delta_k,\epsilon_k)}\big)\cdots
		\lambda_{c(1)}\!\big(a_1^{(\delta_1,\epsilon_1)}\big)\,\xi
		=
		U_w\!\left[
			\bigg(\prod_{j\in B_m} Q_{v_m} a_j\bigg)\xi_{v_m}\otimes\cdots\otimes
			\bigg(\prod_{j\in B_1} Q_{v_1} a_j\bigg)\xi_{v_1}
			\right]\! \\
		\times \prod_{\substack{B\in\pi\\ B\ \text{finished}}} K^{\Bool}_{|B|}(a_j:j\in B).
	\end{multline*}

	If $\pi\notin\mathcal{P}(c,\mathcal{G})_u^0$, then
	\[
		\lambda_{c(k)}\!\big(a_k^{(\delta_k,\epsilon_k)}\big)\cdots
		\lambda_{c(1)}\!\big(a_1^{(\delta_1,\epsilon_1)}\big)\,\xi
		= 0.
	\]
\end{lemma}
\begin{proof}
	Inductive proof on $k$.
	Suppose the result holds for $k-1$ and prove it for $k$. Assume $\pi \in \pcg^0_u$ with $c=c(k) \tilde{c}$. Since the \emph{restriction} $\tilde{\pi}$ of $\pi$ to set $\{k-1,\ldots,1\}$ is in $\mathcal{P}(\tilde{c},\mathcal{G})_u^0$, $\tilde{c} \in W_{\mathcal{G}}$ and from the inductive hypothesis:
	\begin{multline*}
		\lambda_{c(k)}(a_k^{(\delta_k,\epsilon_k)}) \dots \lambda_{c(1)}(a_1^{(\delta_1,\epsilon_1)}) \xi \\
		=\lambda_{c(k)}(a_k^{(\delta_k,\epsilon_k)}) \bigg\{
		U_{\tilde{w}}\bigg[ \bigg( \prod_{j \in \tilde{B}_m} Q_{\tilde{v}_m} a_j\bigg) \xi_{\tilde{v}_m} \otimes \dots \otimes \bigg( \prod_{j \in \tilde{B}_1} Q_{\tilde{v}_1} a_j \bigg) \xi_{\tilde{v}_1}\bigg] \\
		\prod_{B \,\in\, \tilde{\pi} \,\textrm{finished}} K_{\Bool,|B|}(a_j: j \in B)
		\bigg\},
	\end{multline*}
	where $\tilde{B}_1,\ldots, \tilde{B}_m$ are the unfinished block of the restriction $\tilde{\pi}$.

	We have four cases. Set $v=c(k)$.

	$\bullet (\delta_k,\varepsilon_k) = (1,1)$. The height function $h_v$ is positive at $k$, $h_v(k)>0$, indicating that $v$ occurs in $\tilde{w}$.
	Let $B$ be the unfinished block of $\pi$ containing $k$ (it is necessarily unfinished since $(\delta_k,\varepsilon_k) = (1,1))$. Since $\pi$ is colorwise non-crossing, $B$ contains the left-most occurrence of $v$ in $\tilde{c}$.
	$B\backslash \{k\}$ is also an unfinished block of $\tilde{\pi}$. Say
	$
		\tilde{B}_j = B \backslash \{k\}.
	$

	What we have to prove now is that $r(\tilde{w}) = r(v\bar{w})$, which is equivalent to saying that in $\tilde{w}$, the left-most occurrence of $v$ is in tensor relation with every color in $\tilde{w}$ to its left : those are among the $\tilde{v}$, the colors of the unfinished blocks of $\tilde{\pi}$. More precisely, those are the colors $\tilde{v}_{j+1},\ldots,\tilde{v}_m$. By definition, $(B_j,B) \in \mathcal{E}^{\pi}_u$ and since $\pi \in \mathcal{P}(c,\mathcal{G})_u$, we get that $(\tilde{v_j},v) \in \mathcal{E}_2$.

	$\bullet (\delta_k,\varepsilon_k)=(1,0)$, $k$ is right end point of an unfinished block. By definition, the string of colors of the unfinished blocks is $v \tilde{v}_1\cdots \tilde{v}_m$. We have to prove that $w \in W_{\mathcal{G}}$. By the induction hypothesis, $\tilde{w} \in W_{\mathcal{G}}$, it remains to prove:
	\begin{itemize}
		\item  if $v=w_j \in \tilde{w}$, there exists $i:k<i<j$ and $(v_i,v)\not\in \mathcal{E}_2$.
		\item $(v,w_\ell) \in \mathcal{E}_1$
	\end{itemize}
	The second item follows from the fact that $\pi \in \mathcal{P}(c,\mathcal{G})_u$, more precisely $(1')$ of Definition \ref{def:unfinishedpartition}. Let $v=w_j \in \tilde{w}$, since $k$ is the right-end of a block of $\pi$, there exists a second, unfinished block $\tilde{B}_\ell$, among the $\tilde{B}$'s with color $v$. From the definition of $\mathcal{P}(c,\mathcal{G})_u^0$, there exists a third block $\tilde{B}_{\ell}$ nested in between $B$ and $\tilde{B}_{\ell}$ whose color is not in $\mathcal{E}_2$. Since the color of this last block is between $v$ and $\tilde{v}_{\ell}$, $w$ is reduced with $w=v\tilde{w}$. By definition,
	\begin{multline*}
		\lambda_{c(k)}(a_k^{(\delta_k,\epsilon_k)}) \bigg\{
		U_{\tilde{w}}\bigg[ \bigotimes_{1\leq \ell \leq m} \bigg( \prod_{j \in \tilde{B}_m} Q_{\tilde{v}_m} a_j\bigg) \xi_{\tilde{v}_m} \bigg]
		\prod_{B \in \pi ~\text{finished}} K_{\Bool,|B|}(a_j: j \in B).
		\bigg\} \\
		=
		U_{{w}}\bigg[ Q_{v}a_k\xi_v\bigotimes_{1\leq \ell \leq m} \bigg( \prod_{j \in \tilde{B}_m} Q_{\tilde{v}_m} a_j\bigg) \xi_{\tilde{v}_m} \bigg]
		\prod_{B \in \pi ~\text{finished}} K_{\Bool,|B|}(a_j: j \in B).
	\end{multline*}

	$\bullet (\delta_k,\varepsilon_k)=(1,0)$, in that case $k$ is the left end point of a block of $\pi$. Hence, the unfinished blocks of $\pi$ are among the $\tilde{B}$'s, the unfinished blocks of the restriction $\tilde{\pi}$. From the induction hypothesis, $\tilde{w} \in \mathcal{W}$. The word $w$ is obtain from the word $\tilde{w}$ by deleting one color $\tilde{v}_j=v$, the color of the block finished at $k$, which we assume to be $\tilde{B}_j$. Each of the block $\tilde{B}_m,\ldots,\tilde{B}_{j+1}$ has a crossing with $B$, since $\pi \in \mathcal{P}(c,\mathcal{G})_u$, $(v,\tilde{v}_m),\ldots, (v,\tilde{v}_{j+1})\in\mathcal{E}_2$. Hence, $w$ remains reduced. The property of being admissible still holds after deleting letters of an admissible word and thus $w$ is admissible. We saw that all of the colors $\tilde{v}_{j+1}, \ldots, \tilde{v}_{m}$ are in tensor relation with $v$, hence $\tilde{w} = r(vw)$. From the definition of the action of $\lambda_{c(k)}(a_k^{(\delta_k,\varepsilon_k)})$ :
	\begin{multline*}\lambda_{c(k)}(a_k^{(\delta_k,\epsilon_k)}) \bigg\{
		U_{\tilde{w}}\bigg[ \bigotimes_{1\leq \ell \leq m} \bigg( \prod_{j \in \tilde{B}_m} Q_{\tilde{v}_m} a_j\bigg) \xi_{\tilde{v}_m} \bigg]
		\prod_{B \in \pi ~\text{finished}} K_{\Bool,|B|}(a_j: j \in B).
		\bigg\} \\
		=
		U_{\tilde{w}\backslash v_j}\bigg[\bigotimes_{\ell\neq j} \bigg( \prod_{j \in \tilde{B}_m} Q_{\tilde{v}_m} a_j\bigg) \xi_{\tilde{v}_m} \bigg]
		\langle \xi_v \,,\,a_{k}\prod_{\ell\in \tilde{B}_j} Q_{{v}} a_\ell \xi_{{v}} \rangle \prod_{B \in \pi ~\text{finished}} K_{\Bool,|B|}(a_j: j \in B)
	\end{multline*}
	Since \cite{JekelLiu2020}
	$$
		\langle \xi_v \,,\,a_{k}\prod_{\ell\in \tilde{B}_j} Q_{{v}} a_\ell \xi_{{v}} \rangle = K^{\Bool}(a_{k}, a_\ell : \ell \in B_j) )
	$$
	the formula follows.

	$\bullet (\delta_k,\varepsilon_k)=(0,0)$. In that case, $\{k\}$ is a finished singleton. In that case, $\tilde{w} = \tilde{w}$. Thus, from the induction hypothesis, we have $w \in W_{\mathcal{G}}$. Let us check now that $v\tilde{w} \in W_{\mathcal{G}}$. Since $\pi \in \mathcal{P}(c,\mathcal{G})$, from point $(2')$ in Definition \ref{def:unfinishedpartition}, $(\tilde{v}_1,v),\ldots,(\tilde{v}_m,v) \in \mathcal{E}_1$. Moreover, of $v_j=v$ for some $j$, from the definition of the set $\mathcal{P}(c,\mathcal{G})^0_u$, we a third block $B_{\ell}$ nested between $\{k\}$ and $B_j$, $\ell < j$, with its colors $\tilde{v}_\ell$ in not in tensor relation with $v$. Since $\tilde{v}_{\ell}$ is in between $v$ and $\tilde{v}_j$ in $v\tilde{w}$, we get that $v\tilde{w}$ is reduced. From the definition of the action of $\lambda_{c(k)}(a_k^{\delta_k,\varepsilon_k})$, we obtain
	\begin{multline*}
		\lambda_{c(k)}(a_k^{(\delta_k,\epsilon_k)}) \bigg\{
		U_{\tilde{w}}\bigg[ \bigotimes_{1\leq \ell \leq m} \bigg( \prod_{j \in \tilde{B}_m} Q_{\tilde{v}_m} a_j\bigg) \xi_{\tilde{v}_m} \bigg]
		\prod_{B \in \pi ~\text{finished}} K_{\Bool,|B|}(a_j: j \in B).
		\bigg\} \\
		=\langle\xi_v\,,\, a_k \xi_v\rangle \bigg\{
		U_{\tilde{w}}\bigg[ \bigotimes_{1\leq \ell \leq m} \bigg( \prod_{j \in \tilde{B}_m} Q_{\tilde{v}_m} a_j\bigg) \xi_{\tilde{v}_m} \bigg]
		\prod_{B \in \pi ~\text{finished}} K_{\Bool,|B|}(a_j: j \in B).
		\bigg\}
	\end{multline*}
	The formula follows from the equality $ \langle\xi_v\,,\, a_k \xi_v\rangle = K^{\rm Bool}_1(a_k)$.

	The proof of the first assertion of the Theorem is complete, let us prove now the second assertion. If $\pi \not\in \mathcal{P}(c,\mathcal{G})$, there exists an integer $\ell \leq k$ such that the restriction of $\pi$ to $\{1,\ldots,\ell\}$ is in $\mathcal{P}(c,\mathcal{G})$ but the restriction to $\{1,\ldots,\ell+1\}$ is not. Hence, without loss of generality, suppose the restriction of $\pi$ to $\{1,\ldots,k-1\}$ is in $\mathcal{P}(c,\mathcal{G})$. We apply the first assertion of the theorem to the restriction $\tilde{\pi}$ and we are left to proving that the action of $\lambda_{c(k)}(a_k)^{(\delta_k,\varepsilon_k)}$ is equal to zero. Again, four cases occur.

	$ \bullet (\delta_k,\varepsilon_k)= (1,0),(0,0)$. In these cases, $k$ is the right end point of a singleton block, un-finished or finished. If $\pi \in \mathcal{P}(c,\mathcal{G})$, this means either $(\tilde{v}_j,v) \not\in\mathcal{E}_1$ for some $j \leq m$ (say $j=1$ for concreteness) or $\{k\}$ is adjacent to an unfinished block of $\tilde{\pi}$ with same color $v$. In the former case, the word $v\tilde{w}$ does not belong to $W^{(1)}_{\mathcal{G}}$, hence $\lambda_{v}(a_k^{(\delta_k,\varepsilon_k)})$ sends every vector in $H_{r(\tilde{w})}$ to zero, from its definition. In the latter case, the word $v\tilde{w}$ is not reduced and thus does not belong to $\mathcal{W}^{(1)}_\mathcal{G}$.

	$ \bullet (\delta_k,\varepsilon_k)= (0,1),(1,1)$ In these cases, $k$ is the left end-point of a block of $\pi$ that were an unfinished block of $\pi$, or $k$ is added to an unfinished block of $\tilde{\pi}$. In that case $\tilde{\pi} \in \mathcal{P}(c,\mathcal{G})$ means there is a crossing between blocks that should not have been, $(v,\tilde{v}_\ell) \not\in \mathcal{E}_2$ while $\tilde{v}_j=v$ for some $\ell \leq j \leq m$, for concreteness, let us suppose $\ell=1$ and $j =2$.
	In this case, the equivalence class of $\tilde{w}$ does not contain a word starting with $v$ : $v$ can not be exchanged with $\tilde{v}_1$. Therefore, $\lambda^{(\delta_k,\varepsilon_k)}(a_k^{(\delta_k,\varepsilon_k)}$ acts as zero on $H_{r(\tilde{w})}$.
\end{proof}

By applying the proposition above to the case where $h_v(k) = 0$ for all $v \in \mathcal{G}$, we obtain Theorem \ref{thm: Hilbert realization} as a consequence of Lemma \ref{lem: unfinished moments}.

\section{Random matrix model} \label{sec: random matrix model}
This section is dedicated to the proof of Theorem \ref{thm: matrix models main}. We shall first reduce Theorem \ref{thm: matrix models main} to Theorem \ref{thm: matrix models reduction} below, and then prove Theorem \ref{thm: matrix models reduction}.

\subsection{A quick reduction}
Let us set
$
	\tr_N^{c} = \tr_N^{\otimes S_{c}^{(1)}}
$
under the notation of Theorem \ref{thm: matrix models main}.
\begin{theorem}~ \label{thm: matrix models reduction}
	\begin{itemize}
		\item Consider two sets $\mathcal{S}$ and $\mathcal{V}$.
		\item For each $v \in \mathcal{V}$, let $S_v^{(1)}$, $S_v^{(2)}$, $S_v^{(3)}$ be a partition of $\mathcal{S}$ such that $S_v^{(1)}$ is non-empty.
		\item Let $\mathcal{G} = (\mathcal{V},\mathcal{E}_1,\mathcal{E}_2)$, where
		      \[
			      \mathcal{E}_1 = \{(v,w): S_v^{(1)} \cap S_w^{(2)} = \varnothing \}, \qquad \mathcal{E}_2 = \{(v,w): S_v^{(1)} \cap S_w^{(1)} = \varnothing \}.
		      \]
		\item Let $\Mn$ denote the set of $N \times N$ complex matrices; fix a unit vector $\xi \in \bC^N$; and consider $\Mn^{\otimes \mathcal{S}}$ as a non-commutative probability space with the state given by the vector $\xi^{\otimes \mathcal{S}}$.
		\item Define the (not necessarily unital) $*$-homomorphism
		      \[
			      \lambda_v^{(N)}: \Mn^{\otimes S_v^{(1)}} \to \Mn^{\otimes \mathcal{S}}, \qquad \lambda_v^{(N)}(A) = A \otimes (\xi \xi^*)^{\otimes S_v^{(2)}} \otimes I_N^{\otimes S_v^{(3)}},
		      \]
		      where $\xi \xi^*$ denotes the rank-one projection onto the span of $\xi$.
		\item For each $v \in \mathcal{V}$, let $U_v^{(N)}$ be a Haar random unitary matrix in $\Mn^{\otimes S_v^{(1)}}$, such that the $U_v^{(N)}$'s are independent of each other.
	\end{itemize}
	Let $k \in \bN$ and $c: [k] \to \mathcal{V}$.  For $j = 1$, \dots, $k$, let $A_j^{(N)} \in \Mn^{\otimes S_{c(j)}^{(1)}}$ such that
	\[
		\sup_N \norm{A_j^{(N)}} < \infty.
	\]
	Then almost surely
	\begin{multline} \label{eq: matrix convergence statement reduced}
		\lim_{N \to \infty} \biggl| \ip{\xi^{\otimes \mathcal{S}}, \lambda_{c(1)}^{(N)}(U_{c(1)}^{(N)} A_1^{(N)} (U_{c(1)}^{(N)})^*) \dots \lambda_{c(k)}^{(N)}(U_{c(k)}^{(N)}A_k^{(N)}(U_{c(k)}^{(N)})^*) \xi^{\otimes \mathcal{S}}} \\ - \sum_{\pi \in \mathcal{P}(c,\mathcal{G})} \prod_{B \in \pi} K_{|B|}^{\free,\, \tr_N^{{c(B)}}}(A_j^{(N)}: j \in B) \biggr| = 0;
	\end{multline}
	here $c(B)$ denotes the common value of $c(j)$ for $j \in B$, and the free cumulants are computed in $(\Mn^{\otimes S_{c(B)}^{(1)}}, \tr_N^{{c(B)}})$.
\end{theorem}

\begin{proof}[Proof of Theorem \ref{thm: matrix models main} from Theorem \ref{thm: matrix models reduction}]
	There are three differences in the setup of Theorem \ref{thm: matrix models reduction}:
	\begin{enumerate}
		\item We replace the elements $p(\mathbf{B}_{c(j)}^{(N)}) = \lambda_{c(j)}^{(N)}(p(\mathbf{A}_{c(j)}^{(N)}))$ simply by $\lambda_{c(j)}^{(N)}(A_j^{(N)})$.
		\item We assume the input matrices $A_j^{(N)}$ are deterministic.
		\item We compute cumulants with respect to the trace rather than the vector $\xi^{\otimes \mathcal{S}}$ in \eqref{eq: matrix convergence statement main}, and remove \eqref{eq: matrix trace versus vector main}
	\end{enumerate}
	Let us start by assuming that Theorem \ref{thm: matrix models reduction} holds and perform each modification in succession.  For item (3), note that we can \eqref{eq: matrix convergence statement reduced} to the case where all $c$ is constant, say $c(j) = v$ for all $j$.  This yields that
	\[
		\lim_{N \to \infty}  \left| \ip{\xi^{\otimes \mathcal{S}}, \lambda_v^{(N)}(U_v^{(N)}A_1^{(N)}(U_v^{(N)})^* \dots U_v^{(N)} A_k^{(N)} (U_v^{(N)})^*) \xi^{\otimes \mathcal{S}}} - \sum_{\pi \in \mathcal{NC}_k} K_{\pi}^{\free,\tr_N^{v}}(A_1^{(N)},\dots,A_k^{(N)}) \right| = 0.
	\]
	Indeed, since the color is constant, $\mathcal{P}(c,\mathcal{G})$ reduces to $\mathcal{NC}_k$.  Hence, by the free moment-cumulant formula,
	\begin{equation} \label{eq: matrix vector versus trace reduced 1}
		\lim_{N \to \infty}  \left| \ip{\xi^{\otimes \mathcal{S}}, \lambda_v^{(N)}(U_v^{(N)}A_1^{(N)}(U_v^{(N)})^* \dots U_v^{(N)}A_k^{(N)}(U_v^{(N)})^*) \xi^{\otimes \mathcal{S}}} - \tr_N^{v}(A_1^{(N)} \dots A_k^{(N)}) \right|.
	\end{equation}
	Hence, the moments of $U_v^{(N)}A_j^{(N)}(U_v^{(N)})^*$ with respect to $\xi$ are almost surely asymptotically the same as the moments of $A_j^{(N)}$ with respect to the trace, on the matrices from a given color $v$.  Note that each moment and each cumulant is almost surely bounded as $N \to \infty$ since $\sup_N \norm{A_j^{(N)}} < \infty$.  By M{\"o}bius inversion, the agreement of moments with respect to the vector state and the trace implies agreement of free cumulants with respect to the vector state and the trace, and so almost surely
	\begin{multline*}
		\lim_{N \to \infty}  \biggl| \sum_{\pi \in \mathcal{NC}_k} K_{\free,\pi}^{\xi^{\otimes \mathcal{S}}}(\lambda_{c(1)}^{(N)}(U_{c(1)}^{(N)}A_1^{(N)}(U_{c(1)}^{(N)})^*),\dots, \lambda_{c(k)}^{(N)}(U_{c(k)}^{(N)}A_k^{(N)}(U_{c(k)}^{(N)})^*)) \\
		- \sum_{\pi \in \mathcal{P}(c,\mathcal{G})} \prod_{B \in \pi} K_{|B|}^{\free,\tr_N^{{c(B)}}}(A_{c(j)}^{(N)}: j \in B) \biggr| = 0.
	\end{multline*}
	Hence also \eqref{eq: matrix convergence statement reduced} becomes almost surely
	\begin{multline} \label{eq: matrix convergence statement reduced 2}
		\lim_{N \to \infty} \biggl| \ip{\xi^{\otimes \mathcal{S}}, \lambda_{c(1)}^{(N)}(U_{c(1)}^{(N)}A_1^{(N)}(U_{c(1)}^{(N)})^*) \dots \lambda_{c(k)}^{(N)}(U_{c(k)}^{(N)}A_k^{(N)}(U_{c(k)}^{(N)})^*) \xi^{\otimes \mathcal{S}}} \\
		- \sum_{\pi \in \mathcal{NC}_k} K_{\pi}^{\free,\xi^{\otimes \mathcal{S}}}(\lambda_{c(1)}^{(N)}(U_{c(1)}^{(N)}A_1^{(N)}(U_{c(1)}^{(N)})^*,\dots, \lambda_{c(k)}^{(N)}(U_{c(k)}^{(N)}A_k^{(N)}(U_{c(k)}^{(N)})^*)) \biggr|.
	\end{multline}
	Meanwhile, equation \ref{eq: matrix vector versus trace reduced 1} will eventually become \ref{eq: matrix trace versus vector main}.

	For item (2), let us now show that \eqref{eq: matrix convergence statement reduced 2} holds when the matrices $A_j^{(N)}$ are random, they are jointly (for all $j$ and $N$) independent of the Haar unitaries, and the hypothesis $\sup_N \norm{A_j^{(N)}} < \infty$ holds almost surely.  The independence assumption means we can assume without loss of generality that our random variables are realized on a product probability space $\Omega_1 \times \Omega_2$ where $\omega_1 \in \Omega_1$ governs the $A_j^{(N)}$'s and $\omega_2 \in \Omega_2$ the Haar unitaries.  For every $\omega_1 \in \Omega_1$ such that $\sup_N \norm{A_v^{(N)}} < \infty$, we can apply the conclusion of the first step to show that that \eqref{eq: matrix convergence statement reduced 2} holds for almost every $\omega_2 \in \Omega$.  Then by the Fubini--Tonelli theorem, \eqref{eq: matrix convergence statement reduced 2} holds for almost every $(\omega_1,\omega_2) \in \Omega_1 \times \Omega_2$.

	For item (1), now let $\mathbf{A}_i^{(N)} = (A_{v,i}^{(N)})_{i \in I_v}$ be as in Theorem \ref{thm: matrix models main} and let
	\[
		\mathbf{B}_v^{(N)} = (\lambda_v^{(N)}(U_v^{(N)} A_{v,i}^{(N)} (U_v^{(N)})^*))_{i \in I_v}.
	\]
	Given polynomials $p_1$, \dots, $p_k$ with no constant term, consider the matrices $\widehat{A}_j^{(N)} = p_j(A_{c(j),i}^{(N)}: i \in I_{c(j)})$.  Note that
	\[
		p_j(\mathbf{B}_{c(j)}^{(N)}) = \lambda_{c(j)}^{(N)}( U_{c(j)}^{(N)} \widehat{A}_j^{(N)} (U_{c(j)}^{(N)})^*).
	\]
	Moreover, $\sup_N \norm{\widehat{A}_j^{(N)}} < \infty$ almost surely.  Therefore, by applying \eqref{eq: matrix convergence statement reduced 2} to $\widehat{A}_j^{(N)}$, we obtain \eqref{eq: matrix convergence statement main} from Theorem \ref{thm: matrix models main}.  Similarly, by applying \eqref{eq: matrix vector versus trace reduced 1}, we obtain \eqref{eq: matrix trace versus vector main}.
\end{proof}

The next three subsections are devoted to proving Theorem \ref{thm: matrix models reduction}, and the rest of the section is organized as follows:
\begin{itemize}
	\item In \S \ref{subsec: Weingarten}, we prepare the combinatorial setup and give background on the Weingarten calculus for Haar random unitaries.
	\item In \S \ref{subsec: convergence in expectation}, we give the main combinatorial argument that shows convergence in expectation for \eqref{eq: matrix convergence statement reduced}.
	\item In \S \ref{subsec: almost sure convergence}, we upgrade convergence in expectation to almost sure convergence using standard concentration-of-measure techniques, concluding the proof of Theorem \ref{thm: matrix models reduction} and hence Theorem \ref{thm: matrix models main}.
\end{itemize}

\subsection{Weingarten calculus} \label{subsec: Weingarten}

We begin with recalling integration formulae on the unitary group and fixing the notations.

\begin{itemize}
	\item For $k\geq 1$, let $\gamma=(1\,2\,\ldots\,k)$ be the full cycle in $S_k$.
	\item For a color word $v_1\cdots v_k$, let ${\rm Stab}(v_1\cdots v_k)\subseteq S_k$ be the subgroup stabilizing the word (respectively to the left action).
	\item For each site $s\in \mathcal{S}$, define index sets
	      \[
		      J^{(1)}_s(v_1\cdots v_k)=\{j\in[k]: s\in S^{(1)}_{v_j}\},\qquad
		      J^{(2)}_s(v_1\cdots v_k)=\{j\in[k]: s\in S^{(2)}_{v_j}\}.
	      \]
	\item Any $\sigma\in{\rm Stab}(v_1\cdots v_k)$ preserves $J^{(1)}_s$, so we write $\sigma_s$ for its restriction to $S_{J^{(1)}_s}$. Let $\gamma_s$ be the full cycle on $J^{(1)}_s$
	      : it sends an element of $J_s^{(1)}$ to the next element in their canonical cyclic order. We extend permutations on $J^{(1)}_s$ to $[k]$ by fixing indices outside $J^{(1)}_s$ and denote this extension by $\overline{\sigma}_s$.
	\item For each color $v\in\mathcal{V}$, set
	      \[
		      B_v=\{j\in[k]: v_j=v\},\qquad \pi_{v_1\cdots v_k}=\{B_v: v\in\mathcal{V},\, B_v\neq\emptyset\}.
	      \]
	\item For $\sigma\in S_k$, let ${\rm Cyc}(\sigma)$ be the set of cycles of $\sigma$.
\end{itemize}
To lighten the notations, we will often drop references to the word $v_1\ldots v_k$. We want now to recall the following important formula for integrating polynomials against the Haar measure on the group of unitaries of $\mathbb{C}^{N}$,
$N\geq 1$:
\begin{align}
	\label{eqn:integration}
	\int_{\mathbb{U}(N)} u_{i_1j_1}\cdots u_{i_kj_k}\bar{u}_{i'_1j'_1}\cdots \bar{u}_{i'_kj'_k} {\rm d}u= \sum_{\sigma,\tau \in S_k}\delta_{j,\sigma\cdot j'}\delta_{i,\tau\cdot j'} {\rm Wg}(\sigma\tau^{-1},N)
\end{align}
The interested reader is directed to \cite{collins2006integration} for a detailed account of the relations of the above formula with the representation theory of the unitary and symmetric groups. We limit ourselves here to recalling the asymptotic of the function ${\rm Wg}$ appearing in the formula.
The function ${\rm Wg}(\cdot, N)$ is the Weingarten function. For each $(\sigma,\tau)$, ${\rm Wg}(\sigma\tau^{-1})$ is the
coefficient $(\sigma,\tau)$ of the matrix inverse of the Gram matrix of $ \{ \sigma, \sigma \in S_k\}$ with respect to the normalized Hilbert-Schmidt product in the canonical representation of $S_k$
acting on the $k-$ fold tensor product of $\mathbb{C}^N$, details can be found e.g. in \cite{collins2006integration}. Let us recall the following well-known asymptotic of
the Weingarten function:
\begin{align}
	\label{eqn:asymptotic}
	{\rm Wg}(\sigma,N) = \mu(\sigma)N^{-k-|{\rm Cy}(\sigma)|}+O(N^{-2 -k - |{\rm Cy}(\sigma)|})
\end{align}
where $\mu\colon S_k \to \mathbb{R}$ coincides with the Möbius function of the poset of non-crossing partitions ${\rm NC}([k])$ when $\sigma$ is a non-crossing partition (each orbit has the cyclic order of $[k]$ and ${\rm Cy}(s)$ is a non-crossing partition of $[k]$).
For further use, let us define the following modification of the Weingarten function, depending on the choice of a word on colors $c_1\cdots c_k$:
\begin{align}
	\label{eqn:tildeweingarten}
	\tilde{\rm Wg}(\alpha) = \prod_{c\in\mathcal{V}}{\rm Wg}(\alpha |_{B_c}, N^{|S_c^{(1)}|})
\end{align}
From \cite{charlesworth2021matrix}, the following asymptotic holds for $\tilde{\rm Wg}(\alpha)$ :
\begin{align}
	\label{eqn:wgtilde}
	\tilde{\rm Wg}(\alpha) = \mu(\alpha)\prod_{s\in S} N^{-|J^{(1)}_s|-|{\rm Cyc}(\sigma_s)|} + O(\frac{1}{N^2})
\end{align}

\subsection{Convergence in expectation} \label{subsec: convergence in expectation}

We are ready to prove our Theorem \ref{thm: matrix models main}, this will be a consequence of the following proposition.
\begin{proposition}
	\label{thm:mainthm} Consider the same setup and hypotheses as Theorem \ref{thm: matrix models reduction}. As $N$ goes to infinity,
	The following convergence in expectation holds:
	\begin{multline}
		\label{eqn:initialformula}
		\mathbb{E}\left[\langle \xi^{\otimes\,\mathcal{S}} \,,\, \lambda^{(N)}_{c(1)}(U^{(N)}_{c(1)} A^{(N)}_1 (U_{c(1)}^{(N)})^{\star}) \cdots  \lambda^{(N)}_{c(k)}(U^{(N)}_{c(k)} A^{(N)}_k (U_{c(k)}^{(N)})^{\star})\rangle\right] \\
		= \sum_{\pi \in \pcg}\hspace{-0.25cm}K^{\free, \tr_N}_{\pi}(A^{(N)}_1,\ldots,A^{(N)}_k) +O(\frac{1}{N^2})
	\end{multline}
	where
	the $K^{\free, \rm tr}_{\pi}(A^{(N)}_1,\ldots,A^{(N)}_k)$ are the partitioned free cumulants of the $A_1,\ldots,A_k$ for \emph{the normalized traces on} $\Mn^{\otimes\, S_c^{(1)}}$ (the normalization factor is $N^{-|S^{(1)}_{c}|})$, $c \in \mathcal{V}$ and
\end{proposition}
\begin{proof} We drop the superscript $(N)$ for notational simplicity. We choose an orthonormal basis of $\mathbb{C}^{N}$ so that $\xi$ is the first vector of that basis. To make things more explicit, we will re-write the product on the left-hand side of \eqref{eqn:initialformula} to make the projection $\xi\xi^*$ to appear. We will make the abuse of writing each
	$A_j$ and each unitary $U_{c_j}$ as an element of $\Mn^{\otimes \mathcal{S}}$ via the unital inclusion
	$$
		\widehat{\lambda}_c: \Mn^{\otimes\,S^{(1)}_c} \to \Mn^{\otimes\,\mathcal{S}},\quad A \mapsto A \otimes I_N^{\,\otimes\,S_c^{(2)}\cup S_c^{(3)}}.
	$$
	We first write $A^{(N)}_{\ell}  = (a^{(\ell)}_{j_{\ell}j_{\ell}'})_{j_{\ell},j'_{\ell} \in [N]^{|S|}}$ in the chosen basis, $U_{c_\ell} = (u^{(c_\ell)}_{i_{\ell}i'_{\ell}})_{i_{\ell},i'_{\ell} \in [N]^{|S|}}$, $\ell \in [k]$.
	We now add an unitary matrix $U_0$ and $A_0$ to the two sequences $(U_{c(1)},\ldots,U_{c(k)})$ and $(A_1,\ldots, A_k)$ :
	$$
		A_0 = (\xi^{\star}\xi^{\star})^{\otimes\,|S|} = \big(\prod_{s\in S}\delta_{j_0[s], j'_0[s] = 1,1}\big)_{j_0,j'_0 \in [N]^{|S|}}, \quad U_0 = e^{i\theta}I_N^{\,\otimes\,S},\quad \theta \sim  {\rm Unif}([0,2\pi))
	$$
	We also set $v_0 = 0$ and assume $0\not\in \mathcal{V}$. We define the subsets of $\mathcal{S}$ associated to $0$ as
	$$S_0^{(1)} = \emptyset,\quad S_0^{(2)}=S.$$

	With this definition, $0$ is the only color with $S^{(1)}_0=\emptyset$.
	The left side of \eqref{eqn:initialformula} factorizes over the blocks of $\ker{c_0\cdots c_k}= (B_{c})_{c\in\mathcal{V}\cup \{0 \}}$ as explained below.
	First
	\begin{multline}
		\label{eqn:rewriting}
		\mathbb{E}\left[{\rm Tr}\big((\xi^{\star}\xi^{\star})^{\otimes\,S} (\xi^{\star}\xi^{\star})^{\otimes\,S^{(2)}_{c_1}}U_1A_{1}U_1^{\star} (\xi^{\star}\xi^{\star})^{\otimes\,S^{(2)}_{c_1}} (\xi^{\star}\xi^{\star})^{\otimes\,S^{(2)}_{c_2}} \cdots  (\xi^{\star}\xi^{\star})^{\otimes\,S^{(2)}_{c_k}}U_kA_{j}U_k^{\star} (\xi^{\star}\xi^{\star})^{\otimes\,S^{(2)}_{c_k}}\big)\right] \\
		=\mathbb{E}\bigl[{\rm Tr}\big((\xi^{\star}\xi^{\star})^{\otimes\,S}U_0(\xi^{\star}\xi^{\star})^{\otimes\,S}A_0(\xi^{\star}\xi^{\star})^{\otimes\,S}U^\star_0(\xi^{\star}\xi^{\star})^{\otimes\,S}                                                                                                                                                                  \\
		(\xi^{\star}\xi^{\star})^{\otimes\,S^{(2)}_{c_1}}U_{c_1}A_{1}U_{c_1}^{\star} (\xi^{\star}\xi^{\star})^{\otimes\,S^{(2)}_{c_1}} (\xi^{\star}\xi^{\star})^{\otimes\,S^{(2)}_{c_2}} \cdots  (\xi^{\star}\xi^{\star})^{\otimes\,S^{(2)}_{c_k}}U_{c_k}A_{k}U_{c_k}^{\star} (\xi^{\star}\xi^{\star})^{\otimes\,S^{(2)}_{c_k}}\big)\bigr]
	\end{multline}
	What we have gained here is that now all terms in the product look the same, this will streamline the analysis. Now, we write down the product by using the coefficients of the $A's$ in the chosen basis:
	\begin{align}
		\eqref{eqn:rewriting}= \!\!\!\!\!\!\sum_{\substack{i_{\ell},i'_{\ell},j_\ell,j'_{\ell} \\ \in [N]^{|S|}}}\!\!\mathbb{E}\bigg[u^{(c_\ell)}_{i_\ell j_\ell}a_{j_\ell j_\ell'}\bar{u}^{(c_{\ell})}_{i'_\ell j'_\ell}\prod_{\ell = 0}^k \delta_{i'_{\ell}=i_{\ell+1}}  \prod_{s\in S_{c_\ell}^{(2)}}\delta_{j_{\ell}[s]=i_{\ell}[s]=1}\delta_{j'_{\ell}[s]=i'_{\ell}[s]=1} \prod_{s \in S^{(3)}_{c_{\ell}}} \delta_{j_{\ell}[s] = i_{\ell}[s]}\delta_{j'_{\ell}[s] = i'_{\ell}[s]} \label{eqn:zeroline}
		\bigg]
	\end{align}
	The second and third product account, respectively, for the compression by the projection $\xi\xi^{\star}$ and the fact that the $A's$ have their factor corresponding to indices in $S^{(3)}$ equal to the identity $I_N$. The last expectation factorizes over the blocks of the partition $\pi$, and we use the integration formula \eqref{eqn:integration} to infer further:
	\begin{align}
		 & \hspace{0cm}=\!\!\!\!\!\sum_{\substack{\ell =  0\ldots k                                     \\ i_{\ell},i'_{\ell},j_\ell,j'_{\ell} \\ \in [N]^{|S|}}}a^{(0)}_{j_0j'_0}\cdots a^{(k)}_{j_kj'_k} \sum_{\sigma, \tau \in {\rm Stab}(0c_1\cdots c_k)} {\rm} {\rm \tilde{Wg}}(\sigma\tau^{-1},N) \label{eqn:firstline}                                                                                                                                                                                                                   \\
		 & \nonumber \times \prod_{c\in \mathcal{V}\cup \{ 0 \}}\prod_{\ell \in B_{c}}
		\prod_{s\in S_{c}^{(1)}} \delta_{{j_\ell}[s],\sigma\cdot{j_\ell}'[s] }\delta_{{i_\ell}[s],{\tau \cdot i_\ell}'[s] }
		\prod_{s\in S^{(2)}_{c}} \delta_{{{j_\ell}[s] = {i_\ell}[s]=1}} \delta_{{j'_\ell}[s] = {i'_\ell}[s]=1}
		\prod_{s\in S_{c}^{(3)}}\delta_{{j_\ell}[s] = {i_\ell}[s]} \delta_{{j'_\ell}[s] = {i'_\ell}[s]} \\
		 & \times \prod_{\ell=0}^{k} \delta_{i'_{\ell} = i_{\ell+1}}, \nonumber
	\end{align}
	with the usual convention $ k+1 = 0$ in the last line.

	Let us now analyze which tuples $i,i',j,j'$ contribute to the first sum in \eqref{eqn:firstline}
	The product of the delta functions on the third line of equation \eqref{eqn:firstline} implies that contributing tuples $i$ and $i'$ to the sum should satisfy
	$$
		i'= Z^{-1}\cdot i
	$$
	In equation \eqref{eqn:firstline}, we are going to perform, first, the summation over the tuples $i$ and, in last, the summation over the tuples $j,j'$. Hence, given a tuple $i$, define the following range for the tuples $j$:
	\begin{align}
		V_{{i}} = \{ {j}\colon \forall \ell, \forall s\in S^{(3)}_{c_{\ell}}, \,j_{\ell}[s] = i_{\ell}[s], \quad \forall s\in S^{(2)}_{c_{\ell}}, \quad j_{\ell}[s]=1\}
	\end{align}

	We have included in the definition of $V_i$ the effect of the products of the delta functions on the second line of  \eqref{eqn:firstline}  on the contributing tuples $j$. The product in \eqref{eqn:firstline} together with the fact that $a_{j_\ell[s]j'_\ell[s]}=0$ if $j_\ell[s]\neq j'_\ell[s]$, $s\in S^{(3)}_{c}$, $c \in\mathcal{V}$ implies $i_\ell[s] = i_\ell'[s]$. We infer:

	\begin{align}
		\eqref{eqn:firstline}
		 & \nonumber =\!\!\!\!\!\!\!\!\!\!\sum_{\sigma,\tau \in {\rm Stab}(0c_1\cdots c_k)} \!\!\!\!\!\!\!{\rm \tilde{Wg}}(\sigma\tau^{-1},N) \sum_{{i}\in [N]^{k+1}} \Delta_{{i}, (Z^{-1})\tau \cdot {i}} \prod_{c\in \mathcal{V}}\prod_{\ell \in B_{c}}
		\prod_{s\in S_{c}^{(3)}} \delta_{i_{\ell}[s]=i'_{\ell}[s]} \prod_{s\in S_{c}^{(2)}} \delta_{i_{\ell}[s]=i'_{\ell}[s]=1}                                                                                                                           \\
		 & \hspace{8cm}\times \sum_{j,j' \in V_{{i}}} a^{(0)}_{j_0j'_0}\cdots a^{(k)}_{j_kj'_k} \,\Delta_{{j},\sigma \cdot {j'}} \label{eqn:sumjjprime}
	\end{align}
	where, for any $x,y \in [N]^k$:
	\begin{align}
		\Delta_{{x}, {y}} = \prod_{c\in \mathcal{V}}\prod_{\ell \in B_{c}}
		\prod_{s\in S_{c}^{(1)} } \delta_{x_{\ell}[s]=y_{\ell}[s]}.
	\end{align}
	We continue with computing the sum over $j,j'$ in \eqref{eqn:sumjjprime}. First, it does not depend on $i$ (since for $s\in S_{c}^{(3)}$, $a_{j_\ell[s]j_{\ell}[s]}=1$).
	Now, the sum $\eqref{eqn:sumjjprime}$ factorizes over the orbit of $\sigma$, each orbit contribute with a (non-normalized) trace of the corresponding product of the $A's$:
	\begin{align*}
		\sum_{j,j'\in V_{{i}}} a^{(0)}_{j_0j'_0{}}\cdots a^{(k)}_{j_kj'_k{}} \Delta_{{j},\,\sigma \cdot {j'}} & = {\rm tr}_{\sigma |_{B_{c}}}(A^{(N)}_1,\ldots,A^{(N)}_k) \prod_{c \in \mathcal{V}} N^{|S^{(1)}_{c}||{\rm Cyc}(\sigma^{-1}_{|_{B_{c}}})|}
		\\ &= {\rm tr}_{\sigma}(A^{(N)}_1,\ldots,A^{(N)}_k) \prod_{s\in S} N^{|{\rm Cyc}(\sigma_s)|}
	\end{align*}
	We turn to the sum (at this step, our analysis of the contributing terms to \eqref{eqn:firstline} becomes genuinely different from the work \cite{charlesworth2021matrix}):
	\begin{align}
		\label{eqn:theothersum}
		\sum_{{i}} \Delta_{{i}, (\tau Z^{-1}) \cdot {i}} \prod_{c\in \mathcal{V}}\prod_{\ell \in B_{c}}
		\prod_{s\in S^{(3)}_{c}} \delta_{i_{\ell}[s]=i'_{\ell}[s]} \prod_{s\in S_{c}^{(2)}} \delta_{i_{\ell}[s]=i'_{\ell}[s]=1}.
	\end{align}
	The contributing terms to the above sum are the tuples ${i}$ satisfying, for each color $c$, for any $\ell \in B_{c}$:
	\begin{align}
		 & {i}_{\ell}[s] = {i}_{\tau^{-1}(\ell)+1}[s], &  & s \in S^{(1)}_{c},\label{eqn:conditionone}   \\
		 & {i}_{\ell}[s] = {i}_{\ell+1}[s],            &  & s \in S_{c}^{(3)},\label{eqn:conditionthree} \\
		 & {i}_{\ell}[s] = {i}_{\ell+1}[s]=1,          &  & s \in  S^{(2)}_{c},\label{eqn:conditiontwo}.
	\end{align}

	The tuples $i$ meeting the conditions \eqref{eqn:conditionone}-\eqref{eqn:conditiontwo} are constant on the cycles of the permutation $Z^{-1}\bar{\tau}_s$ of the set $\{0,\ldots,k\}$.
	Besides, on cycles of $Z^{-1}\bar{\tau}_s$ containing an integer $j \in \{0,\ldots,k\}$ with $s\in S^{(2)}_{c_j}$, the tuple $i$ is constant and is equal to $1$. Thus, we infer the following formula for the sum \eqref{eqn:theothersum}:
	$$
		\eqref{eqn:theothersum}=\prod_{s\in S}N^{|{\rm Cyc}((Z^{-1}\bar{\tau_s}))|-|\{ \,x\,\in\,{\rm Cyc}(Z^{-1}\tau_s)\,:\,J_s^{(2)} \,\cap \,x \neq \,\emptyset \}|}.
	$$
	We obtain finally:
	\begin{align}
		 & \nonumber \mathbb{E}{\rm Tr}\big((\xi^{\star}\xi^{\star})^{\otimes\,S} (\xi^{\star}\xi^{\star})^{\otimes\,S^{(2)}_{c_1}}U_1A_{1}U_1^{\star} (\xi^{\star}\xi^{\star})^{\otimes\,S^{(2)}_{c_1}} (\xi^{\star}\xi^{\star})^{\otimes\,S^{(2)}_{c_2}} \cdots  (\xi^{\star}\xi^{\star})^{\otimes\,S^{(2)}_{c_k}}U_kA_{j}U_k^{\star} (\xi^{\star}\xi^{\star})^{\otimes\,S^{(2)}_{c_k}}\big)  \nonumber \\
		 & = \sum_{\sigma,\tau \in {\rm Stab}(0c_1\cdots c_k)} \hspace{-0.5cm}\tilde{\rm Wg}(\sigma\tau^{-1},N)\,{\rm tr}_{\sigma}(A_1,\ldots,A_k)                                                                                                                                                                                                                                                        \\
		 & \hspace{5cm}\times\prod_{s\in S} N^{|{\rm Cyc}((Z^{-1}\bar{\tau}_s))|-|\{ \,x\,\in\,{\rm Cyc}(Z^{-1}\bar{\tau}_s)\,:\,J_s^{(2)} \,\cap \,x \neq \,\emptyset \}|+|{\rm Cyc}(\sigma_s)|}
	\end{align}
	Recall the definition \eqref{eqn:tildeweingarten} of $\tilde{\rm Wg}(\sigma\tau^{-1})$.
	Let us rewrite, as in \cite{charlesworth2021matrix}, the exponent of $N$ in the last formula. Since $\bar{\tau}(J^{(1)}_s) \subset J^{(1)}_s$, $Z^{-1}\bar{\tau}_{s}$ has the same number of cycles as $Z_s^{-1}\bar{\tau}_{s}$ ; let $i \in J^{(1)}_s$, then take $q\geq 1$ the smallest integer such
	that $Z^{-q}\bar{\tau_s}(i) \in J^{(1)}_s$, by definition $(Z^{-1}\bar{\tau}_s)^{q}(i) = Z_s^{-1}\bar{\tau_s}(i) $. Thus the orbit of $i$ under $Z^{-1}_s\tau_{s}$ is the intersection of the the orbit of $Z^{-1}\bar{\tau}_s$
	with $J^{(1)}_s$. If $j \notin J^{(1)}_s$, then take the least integer $q\geq 1$ such that $Z^{-q}(i) \in J^{(1)}_s$, since $Z^{-q}(i)= (Z\bar{\tau}_s)^{-q}(i)$, we can apply the previous reasoning to this new element. Hence, the number of cycles of $Z^{-1} \bar{\tau_s}$ and $Z^{-1}_s\tau_s$ coincide.
	We infer
	\begin{align}
		 & |{\rm Cyc}((Z^{-1}\bar{\tau}_s))|+|{\rm Cyc}(\sigma_s)| - 1 - |J^{(1)}_s| - |\sigma_s\tau^{-1}_s|\nonumber \\
		 & = |{\rm Cyc}((Z^{-1}\bar{\tau}_s))|- |\sigma_s| - 1 -  |\sigma_s\tau^{-1}_s| \nonumber                     \\
		 & = - (|\sigma_s| + |\sigma_s^{-1}\tau_s| + |\tau_s^{-1}Z_s| - |Z_s|)		\label{eqn:equalitycycles}.
	\end{align}
	Finally, by inserting the estimate \eqref{eqn:wgtilde} and using \eqref{eqn:equalitycycles}, we obtain that \eqref{eqn:equalitycycles} is equal to, up to a remainder term in $O(\frac{1}{N^2})$,
	\begin{multline*}
		\sum_{\substack{\sigma,\tau \in \\{\rm Stab}(0c_1\cdots c_k)}} \hspace{-0.5cm}\mu(\sigma\tau^{-1}) {\rm tr}_{\sigma}(A_1,\ldots,A_k)
		\times\prod_{s\in S} N^{-(|\sigma_s| + |\sigma_s^{-1}\tau_s| + |\tau_s^{-1}Z_s| - |Z_s| + |\{ x \in {\rm Cyc}(Z^{-1}\bar{\tau}_s) : \exists j \in J_s^{(2)} \cap x\}| - 1)}.
	\end{multline*}
	To finish the proof of \ref{thm:mainthm}, we have to minimize the exponent:
	$$
		(|\sigma_s| + |\sigma_s^{-1}\tau_s| + |\tau_s^{-1}Z_s| - |Z_s| + |\{ x \in Z^{-1}\tau_s : \exists j \in J_s^{2} \cap x\}| - 1)
	$$
	For any $s\in S$:
	\begin{align}
		|\sigma_s | + |\sigma_s^{-1}\tau_s| + |\tau_s^{-1}c_s| & - |c_s| + |\{ x \in Z^{-1}\tau_s : \exists j \in J_s^{2} \cap x\}| - 1  \nonumber                                                                   \\
		                                                       & =\underset{\geq 0}{\big(|\sigma_s | + |\sigma_s^{-1}\tau_s| - |\tau_s|\big)} \label{eqn:firstterm}                                                  \\
		                                                       & \hspace{1cm}+ \underset{\geq 0}{\big(|\tau_s|+|\tau_s^{-1}Z_s| - |Z_s|\big)} \label{eqn:secondterm}                                                 \\
		                                                       & \hspace{2cm}+ \underset{\geq 0}{\big(|\{ x \in {\rm Cyc}(Z^{-1}\bar{\tau}_s) :  J_s^{(2)} \cap x \neq \emptyset \}| - 1\big)} \label{eqn:thirdterm}
	\end{align}
	Recall that when $\sigma$ is any non-crossing partition (seen as a permutation in $S_k$), $Z^{-1}\tau$ is called the Kreweras complement of $\tau$ \cite{kreweras1972partitions}  and notated {\rm Kr}($\sigma$). The Kreweras complement is a non-crossing partition as well and is the maximal partition such that $\sigma\sqcup {\rm Kr}(\sigma)$ is a non-crossing partition of the set $\{0,0', 1,1',\ldots,k,k'\}$.

	The third term in \eqref{eqn:thirdterm} is non-negative since $0 \in J^{(2)}_s$.
	The second term \eqref{eqn:secondterm} is minimal if $\tau_s$ is a non-crossing partition in $S_{J_s^{(1)}}$, or equivalently, if $\bar{\tau}_s$ is a non-crossing partition in $S_k$. This means that all the cycles of $\tau_s$ has the cyclic order induced by the cyclic order of $\{0,\ldots,k\}$ and the partition whose blocks are the cycles is non-crossing. This is a well-known result that can be traced back to the work of Biane \cite{biane1998representations}, see also \cite{gabriel2015combinatorial}.

	The first term \eqref{eqn:firstterm} is minimal when $\sigma_s$ is a non-crossing partition, and is less (for the containment order) than the non-crossing partition $\tau_s$.

	In that case, both terms \eqref{eqn:firstterm} and \ref{eqn:secondterm}vanish. Let $\sigma$ and $\tau$ be two permutations in ${\rm Stab}(0c_1\cdots c_k)$ such that for any $s \in S$, $\sigma_s$ and $\tau_s$ are non-crossing partitions with $\sigma_s \leq \tau_s$ and the third term \eqref{eqn:thirdterm} is minimal.
	This means the set $J_s^{(2)}$ is contained in the support of only one cycle of the Kreweras complement of $\tau_s$.
	In that case, this third term vanishes.

	Since $\tau_{s}$ is non-crossing for every $s \in S$, this means for two blocks of $\tau$ with colors $c$ and $c'$ to have a crossing, one must have $S^{(1)}_{c} \cap S^{(1)}_{c'} = \emptyset$ i.e $(c,c') \in \mathcal{E}_2$. Hence, the property $\eqref{def: compatible partitions item:three}$ in Definition \eqref{def: compatible partitions} is satisfied for $\tau$.

	Since $0'$ has to be in the same block of ${\rm Kr}(\tau_s)$ containing ${\ell'}$, ${\ell} \in J_s^{(2)}$, for the third term to be minimal, no block of $\tau$ with color $c$, $s \in S^{(1)}_c$, contains elements $v \neq w$ with $0<v\leq \ell < w$ fwith $s \in S^{(2)}_{c_{\ell}}$. Hence property \ref{def: compatible partitions item:two} is satisfied.

	From this discussion and the fact that $\sigma$ has to satisfy properties \eqref{def: compatible partitions item:two} and \eqref{def: compatible partitions item:three} as well ($\sigma_s$ being finer that $\tau_s$ for any $s \in S$), we infer the following formula:
	\begin{align*}
		 & \mathbb{E}{\rm Tr}\big((\xi^{\star}\xi)^{\otimes\,S} (\xi^{\star}\xi)^{\,\otimes\, S^{(2)}_{c_1}}U_1A_{1}U_1^{\star} (\xi^{\star}\xi)^{\otimes\,S^{(2)}_{c_1}} \cdots  (\xi^{\star}\xi)^{\otimes\,S^{(2)}_{c_k}}U_kA_{k}U_k^{\star} (\xi^{\star}\xi)^{\otimes\,S^{(2)}_{c_k}}\big)\nonumber \\
		 & \hspace{1.cm}= \sum_{\substack{\sigma, \tau \in \pcg                                                                                                                                                                                                                                        \\ \sigma_s \leq \tau_s,\,\forall s \in S}} \mu(\sigma^{-1}\tau) {\rm tr}_{\sigma}(A_1,\ldots,A_k) + O({N^{-2}})
	\end{align*}
	Because $\sigma_{s} \leq \tau_s$ for all $s\in S$ if and only if $\sigma_{|B_{c}} \leq \tau_{|B_{c}}$ for all $c \in\mathcal{V}$,  we obtain:
	\begin{align*}
		 & =\sum_{\substack{\sigma, \tau \,\in \,\pcg, \\ \sigma_{|B_{c}} \in \, \leq\, \tau_{|B_{c}}}}
		\mu(\sigma^{-1}\tau) {\rm tr}_{\sigma}(A^{(N)}_1,\ldots,A^{(N)}_k) + O({N^{-2}}).
	\end{align*}
	Now, since the following equivalence holds :
	$$\sigma, \tau \,\in \,\pcg, \, \sigma_{|B_{c}}\, \leq\, \tau_{|B_{c}} \Leftrightarrow \sigma \in \mathcal{P}([k]),\,\tau \,\in \,\pcg,\, \sigma_{|B_{c}} \in {\rm NC}(B_c),\, \sigma_{|B_{c}}\, \leq\, \tau_{|B_{c}}, c \in \mathcal{V},$$ we infer:
	\begin{align*}
		 & =\sum_{\substack{\tau \,\in \,{\rm NC}(c,K),                      \\ \sigma \in \mathcal{P}([k]),\\ \sigma_{|B_c} \in {\rm NC}(B_c)\\ \sigma_{|B_{c}}\, \leq\, \tau_{|B_{c}}}}
		\mu(\sigma^{-1}\tau) {\rm tr}_{\sigma}(A_1,\ldots,A_k) + O({N^{-2}}) \\
		 & =\sum_{\tau \in {\rm NC}(c,K)}
		K^{\rm free, \,\tr_N}_{\tau}(A_1,\ldots,A_k) + O({N^{-2}})
	\end{align*}

	The last equality follows from the free moment-cumulant formula.
\end{proof}

\subsection{Almost sure convergence} \label{subsec: almost sure convergence}

To complete the proof of Theorem \ref{thm: matrix models reduction} (and hence of Theorem \ref{thm: matrix models main}), we only need to upgrade the convergence of expectation from Proposition \ref{thm:mainthm} to almost sure convergence.  This follows by standard techniques from concentration of measure, which has become a standard tool for random matrix theory since \cite{BAG1997,GZ2000}.

We recall that the unitary group $\mathbb{U}_N$, equipped with the Haar measure and the Riemannian metric associated to the inner product $\ip{\cdot,\cdot}_{\Tr_N}$, satisfies the log-Sobolev inequality with constant $6/N$ \cite[Theorem 5.16]{Meckes2019}.  Using e.g.\ \cite[Corollary 5.7]{Ledoux2001}, \cite[Theorem 5.9]{Meckes2019}, one can easily deduce that the product space $\prod_{v \in \mathcal{V}} \mathbb{U}_{N^{\# S_v^{(1)}}}$, where the Riemann metric is the sum of the original Riemann metrics over the different components, with the product measure, satisfies the log-Sobolev inequality with constant $\max_{v \in \mathcal{V}} 6 / N^{\# S_v^{(1)}} \leq 6 / N$.  This in turn implies that it satisfies the Herbst concentration estimate with constant $6/N$; see e.g.\ \cite[Lemma 2.3.3]{AGZ2009}, \cite[Theorem 5.5]{Meckes2019}.  (See \cite[\S 5.3]{freePinsker} for a similar application of these results with more explanation.)

\begin{lemma}[Herbst concentration] \label{lem: concentration}
	Let $\mathcal{G}$ be a bigraph, and let $(U_v^{(N)})_{v \in \mathcal{V}}$ be Haar unitaries as in Theorem \ref{thm: matrix models reduction}.  Let $f: \prod_{v\in\mathcal{V}}\mathbb{U}_{N^{\# S_v^{(1)}}} \to \mathbb{C}$ be $L$-Lipschitz with respect to the geodesic distance associated to the direct sum of the Riemannian metrics $\ip{\cdot,\cdot}_{\Tr_{N^{\# S_v^{(1)}}}}$, $v \in \mathcal{V}$.  Then
	\[
		\mathbb{P}(|f(U_v^{(N)}: v \in \mathcal{V}) - \mathbb{E}[f(U_v^{(N)}: v \in \mathcal{V})]| \geq \delta) \leq 4 e^{-N \delta^2 / 12 L^2}
	\]
\end{lemma}

\begin{proof}[Proof of Theorem \ref{thm: matrix models reduction}]
	We start by relating the geodesic distance in Lemma \ref{lem: concentration} to other distances better suited to random matrix theory.
	Since the inner product on $\prod_{v\in\mathcal{V}}\mathbb{U}_{N^{\# S_v^{(1)}}}$ is the restriction of the sum of the Hilbert-Schmidt scalar products $\ip{\cdot,\cdot}_{\Tr_{N^{\# S_v^{(1)}}}}$ on $\mathbb{M}_{N^{\# S^{(1)}_v}}$, the geodesic distance, which we will denote by $d_{\operatorname{geo}}$, is bounded above by the Hilbert--Schmidt distance $d_{HS}$ given by
	\[
		d_{HS}((U_v)_{v \in \mathcal{V}}, (U_v')_{v \in \mathcal{V}})^2 = \sum_{v \in \mathcal{V}} \Tr_{N^{\# S_v^{(N)}}}((U_v - U_v')^*(U_v - U_v)).
	\]
	Next, let $d_\infty$ be the operator-norm distance given by
	\[
		d_\infty((U_v)_{v \in \mathcal{V}}, (U_v')_{v \in \mathcal{V}}) = \max_{v \in \mathcal{V}} \norm{U_v - U_v'}.
	\]
	Since the Hilbert--Schmidt norm dominates the operator norm, we have $d_\infty \leq d_{HS}$.  The upshot of this discussion is that $d_\infty \leq d_{\operatorname{geo}}$, and therefore, any function which is $L$-Lipschitz with respect to the operator norm must be $L$-Lipschitz with respect to $d_{\operatorname{geo}}$.

	Now consider the function $f: \prod_{v \in \mathcal{V}} \mathbb{U}_{N^{\# S_v^{(1)}}} \to \bC$ given by
	\[
		f^{(N)}(U_v: v \in \mathcal{V}) = \ip{\xi^{\otimes \mathcal{S}} \,,\, \lambda_{c(1)}^{(N)}(U_{c(1)} A_1^{(N)} U_{c(1)}^*) \dots \lambda_{c(k)}^{(N)}(U_{c(k)} A_k^{(N)}U_{c(k)}^*) \xi^{\otimes \mathcal{S}}}.
	\]
	Since we assume that $\norm{A_j^{(N)}}$ is bounded by some constant (independent of $N$) and the unitaries are of course bounded in operator norm by $1$, it is straightforward to check that
	\[
		\lambda_{c(1)}^{(N)}(U_{c(1)} A_1^{(N)} U_{c(1)}^*) \dots \lambda_{c(k)}^{(N)}(U_{c(k)} A_k^{(N)}U_{c(k)}^*)
	\]
	is a Lipschitz function of $(U_v)_{v \in \mathcal{V}}$ with respect to $d_\infty$ in the domain and the operator norm in the target space.  Since $\xi^{\otimes \mathcal{S}}$ is a unit vector, we then see that $f^{(N)}$ is Lipschitz with respect to $d_\infty$, hence with respect to $d_{\operatorname{geo}}$, with a Lipschitz constant $L$ independent of $N$.  Therefore, by Lemma \ref{lem: concentration},
	\[
		\mathbb{P}(|f^{(N)}(U_v^{(N)}: v \in \mathcal{V}) - \mathbb{E}[f^{(N)}(U_1^{(n)},\dots,U_m^{(n)})]| \geq N^{-1/3}) \leq 4 e^{-N \cdot N^{-2/3} / 12 L^2}.
	\]
	Because $\sum_{N=1}^\infty e^{-N^{1/3} / 12 L^2} < \infty$, the Borel--Cantelli lemma implies that almost surely, for sufficiently large $N$, we have
	\[
		|f^{(N)}(U_v^{(N)}: v \in \mathcal{V}) - \mathbb{E}[f^{(N)}(U_v^{(N)}: v \in \mathcal{V})]| < N^{-1/3},
	\]
	and thus, almost surely,
	\begin{multline} \label{eq: a s convergence from concentration}
		\lim_{N \to \infty} \biggl| \ip{\xi^{\otimes \mathcal{S}}, \lambda_{c(1)}^{(N)}(U_{c(1)}^{(N)} A_1^{(N)} (U_{c(1)}^{(N)})^*) \dots \lambda_{c(k)}^{(N)}(U_{c(k)}^{(N)}A_k^{(N)}(U_{c(k)}^{(N)})^*) \xi^{\otimes \mathcal{S}}} \\ - \mathbb{E} \ip{\xi^{\otimes \mathcal{S}}, \lambda_{c(1)}^{(N)}(U_{c(1)}^{(N)} A_1^{(N)} (U_{c(1)}^{(N)})^*) \dots \lambda_{c(k)}^{(N)}(U_{c(k)}^{(N)}A_k^{(N)}(U_{c(k)}^{(N)})^*) \xi^{\otimes \mathcal{S}}} \biggr| = 0.
	\end{multline}
	Now combining \eqref{eq: a s convergence from concentration} with \eqref{eqn:initialformula} from Proposition \ref{thm:mainthm} and the triangle inequality, we obtain the claim \eqref{eq: matrix convergence statement reduced} asserted in Theorem \ref{thm: matrix models reduction}.
\end{proof}

\bibliographystyle{alpha}
\bibliography{references}

\end{document}